\documentclass[11pt,reqno]{amsart}
\usepackage{amsthm,amsmath,amssymb}
\usepackage{stmaryrd,mathrsfs}
\usepackage{esint}
\usepackage{tikz}
 \usepackage[utf8]{inputenc}
        \usepackage{dutchcal}
\usepackage{upgreek}

\usepackage{accents}
\usepackage{fullpage}
\usepackage[mathscr]{euscript}
\usepackage[breaklinks,pdfstartview=FitH]{hyperref}
\renewcommand{\le}{\leqslant}
\renewcommand{\ge}{\geqslant}
\renewcommand{\leq}{\leqslant}

\renewcommand{\setminus}{\smallsetminus}

\newcommand{\ud}[0]{\,\mathrm{d}}

\newcommand{\abs}[1]{|#1|}

\newcommand{\Norm}[2]{\|#1\|_{#2}}

\newcommand{\supp}[0]{\operatorname{supp}}

\newcommand{\sign}[0]{\operatorname{sgn}}

\newcommand{\R}{\mathbb{R}}

\newcommand{\N}{\mathbb{N}}
\newcommand{\Z}{\mathbb{Z}}

\newcommand{\Lip}{\mathrm{Lip}}
\newcommand{\Taylor}{\operatorname{Taylor}}
\newcommand{\e}{\varepsilon}
\renewcommand{\d}{\delta}
\renewcommand{\subset}{\subseteq}
\newcommand{\eqdef}{\stackrel{\mathrm{def}}{=}}
\renewcommand{\div}{\operatorname{div}}
\newcommand{\n}{\{1,\ldots,n\}}
\newcommand{\1}{\mathbf 1}
\renewcommand{\supset}{\supseteq}
\renewcommand{\hat}{\widehat}
\renewcommand{\gamma}{\upgamma}
\renewcommand{\Pr}{\mathsf{Prob}}
\newcommand{\E}{\mathbb{E}}
\renewcommand{\sign}[0]{\operatorname{sign}}
\newcommand{\spn}{\mathrm{span}}
\newcommand{\f}{\phi}
\newcommand{\Proj}{\mathsf{Proj}}
\renewcommand{\emptyset}{\varnothing}

\renewcommand*{\dot}[1]{%
  \accentset{\mbox{\large\bfseries .}}{#1}}

\theoremstyle{plain}
\newtheorem{theorem}{Theorem}
\newtheorem{proposition}[theorem]{Proposition}
\newtheorem{corollary}[theorem]{Corollary}
\newtheorem{lemma}[theorem]{Lemma}
\newtheorem{conjecture}[theorem]{Conjecture}

\theoremstyle{definition}

\theoremstyle{remark}
\newtheorem{remark}[theorem]{Remark}

\newtheorem{question}[theorem]{Question}

\makeatletter
\@namedef{subjclassname@2010}{%
  \textup{2010} Mathematics Subject Classification}
\makeatother

\begin{document}

\title{Heat flow and quantitative differentiation}

\author[T. Hyt\"onen]{Tuomas Hyt\"onen}
\address{(T.H.) University of Helsinki, Department of Mathematics and Statistics, P.O.B.~68 (Gustaf H\"all\-str\"omin katu~2b), FI-00014 Helsinki, Finland}
\email{tuomas.hytonen@helsinki.fi}

\author[A.~Naor]{Assaf Naor}
\address{(A.N.) Princeton University,
Department of Mathematics,
Fine Hall, Washington Road,
Princeton, NJ 08544-1000, USA}
\email{naor@math.princeton.edu}

\date{\today}

\thanks{T.H. was supported by the ERC Starting Grant ``AnProb'' and the Academy of Finland, CoE in Analysis and Dynamics Research. A.~N. was supported by BSF grant
2010021, the Packard Foundation and the Simons Foundation.}
\subjclass[2010]{}


\maketitle
\setcounter{tocdepth}{2}
\vspace{-0.25in}
\begin{abstract}
For every Banach space $(Y,\|\cdot\|_Y)$ that admits an equivalent uniformly convex norm we prove that there exists $c=c(Y)\in (0,\infty)$ with the following property. Suppose that $n\in \N$ and that $X$ is an $n$-dimensional normed space with unit ball $B_X$. Then for every $1$-Lipschitz function $f:B_X\to Y$ and for every $\e\in (0,1/2]$ there exists a radius $r\ge \exp(-1/\e^{cn})$, a point $x\in B_X$ with $x+rB_X\subset B_X$, and an affine mapping $\Lambda:X\to Y$ such that $\|f(y)-\Lambda(y)\|_Y\le \e r$ for every $y\in x+rB_X$. This is an improved bound for a fundamental quantitative differentiation problem that was formulated by Bates, Johnson, Lindenstrauss, Preiss and Schechtman (1999), and consequently it yields a new proof of Bourgain's discretization theorem (1987) for uniformly convex targets. The strategy of our proof is inspired by Bourgain's original approach to the discretization problem, which takes the affine mapping $\Lambda$ to be the first order Taylor polynomial of a time-$t$ Poisson evolute of an extension of $f$ to all of $X$  and argues that, under appropriate assumptions on $f$, there must exist a time $t\in (0,\infty)$ at which $\Lambda$ is (quantitatively) invertible. However, in the present context we desire a more stringent conclusion, namely that $\Lambda$ well-approximates $f$ on a macroscopically large ball, in which case we show that for our argument to work one cannot use the Poisson semigroup. Nevertheless, our strategy does succeed with the Poisson semigroup replaced by the heat semigroup. As a crucial step of our proof, we establish a new uniformly convex-valued Littlewood--Paley--Stein $\mathcal{G}$-function inequality for the heat semigroup; influential work of Mart\'inez, Torrea and Xu~(2006) obtained such an inequality for subordinated Poisson semigroups but left the important case of the heat semigroup open. As a byproduct, our proof also yields a new and simple approach to the classical Dorronsoro theorem  (1985) even for real-valued functions.
\end{abstract}

\tableofcontents


\section{Introduction}

Denote the unit ball of a Banach space $(Y,\|\cdot\|_Y)$ by $B_Y\eqdef \{y\in Y:\ \|y\|_Y\le 1\}$. Recall that the norm $\|\cdot\|_Y$ is said to be uniformly convex if for every $\e\in (0,2]$ there exists $\d\in (0,1]$ such that
\begin{equation}\label{def:uniform convexity}
\forall\, x,y\in B_Y,\qquad \|x-y\|_Y\ge \e\implies \|x+y\|_Y\le 2(1-\d).
\end{equation}

Following Bates, Johnson, Lindenstrauss, Preiss and Schechtman~\cite{BJLPS}, given a pair of Banach spaces $(X,\|\cdot\|_X)$ and $(Y,\|\cdot\|_Y)$, the space $\Lip(X,Y)$ of Lipschitz functions from $X$ to $Y$ is said to have the {\em uniform approximation by affine property} if for every $\e\in (0,\infty)$ there exists $r\in (0,1)$ such that for every $1$-Lipschitz function $f:B_X\to Y$ there exists a radius $\rho\ge r$, a point $x\in X$ with $x+\rho B_X\subset B_X$, and an affine mapping $\Lambda:X\to Y$ such that $\|f(y)-\Lambda(y)\|_Y\le \e \rho$ for every $y\in x+\rho B_X$. Denote the supremum over those $r$ by $r^{X\to Y}(\e)$. The following theorem is due to~\cite{BJLPS}.

\begin{theorem}[Bates--Johnson--Lindenstrauss--Preiss--Schechtman]\label{thm:BJLPS quote}
Suppose that $(X,\|\cdot\|_X)$ and $(Y,\|\cdot\|_Y)$ are Banach spaces with $\dim(X)<\infty$. Then $\Lip(X,Y)$ has the uniform approximation by affine property {\bf if and only if} $Y$ admits an equivalent uniformly convex norm.
\end{theorem}

A function is differentiable if it admits arbitrarily good (after appropriate rescaling) affine approximations on infinitesimal balls. The uniform approximation by affine property was introduced as a way to make this phenomenon quantitative by requiring that the affine approximation occurs on a macroscopically large ball of a definite size that is independent of the specific  $1$-Lipschitz function. In addition to being a natural question in its own right, obtaining such quantitative information has geometric applications. In particular, Bates, Johnson, Lindenstrauss, Preiss and Schechtman introduced it in order to study~\cite{BJLPS}   nonlinear quotient mappings between Banach spaces. Here we obtain the following improved lower bound on the modulus $r^{X\to Y}(\e)$ in the context of Theorem~\ref{thm:BJLPS quote}.

\begin{theorem}\label{thm:main intro}
Suppose that $(Y,\|\cdot\|_Y)$  is a Banach space that admits an equivalent uniformly convex norm. Then there exists  $c=c(Y)\in (0,\infty)$ such that for every $n\in \N$, every $n$-dimensional normed space $(X,\|\cdot\|_X)$, and every $\e\in (0,1/2]$ we have $r^{X\to Y}(\e)\ge \exp(-1/\e^{cn})$.
\end{theorem}
Theorem~\ref{thm:main intro} answers Question~8 of~\cite{HLN} positively. We defer the description of (and comparison to) previous related works to later in the Introduction, after additional notation is introduced so as to facilitate such a discussion; see Section~\ref{sec:previous} below. It suffices to summarize at the present juncture that the proof of Bates, Johnson, Lindenstrauss, Preiss and Schechtman in~\cite{BJLPS} did not yield any quantitative information on $r^{X\to Y}(\e)$. The first bound on this quantity, which is weaker as $n\to \infty$ than that of Theorem~\ref{thm:main intro},  was obtained by Li and and the second named author in~\cite{LiNaor:13}. A bound that is similar to that of Theorem~\ref{thm:main intro} (though weaker in terms of the implicit dependence on the geometry of $Y$) was obtained by Li and  both authors in~\cite{HLN} under an analytic assumption on $Y$ that is strictly more stringent than the requirement that it admits an equivalent uniformly convex norm (which, by Theorem~\ref{thm:BJLPS quote}, is the correct setting for quantitative differentiation). The best known~\cite{HLN} upper bound  on $r^{X\to Y}(\e)$ asserts that for every $p\in [2,\infty)$ there exists a uniformly convex Banach space $Y$ such that if $X$ is an $n$-dimensional Hilbert space then  $r^{X\to Y}(\e)\le \exp(-n(K/\e)^p)$ for every $\e\in (0,\e_0]$, with $K,\e_0>0$ being universal constants. The above estimates (both upper and lower) on $r^{X\to Y}(\e)$  are the best known even when $X$ and $Y$ are both Hilbert spaces and, say, $\e=1/4$.

Our proof of Theorem~\ref{thm:main intro} has conceptual significance that goes beyond the mere fact that it yields an asymptotically improved bound in the maximal possible generality. Firstly, Theorem~\ref{thm:main intro} expresses the best quantitative differentiation result that is obtainable by available approaches, relying on a definitive local approximation estimate of independent interest (see Section~\ref{sec:local approcximation intro} below) that we derive here as a crucial step towards Theorem~\ref{thm:main intro}. Briefly, it seems that the $L_p$ methods (and the corresponding Littlewood--Paley theory) that were used thus far have now reached their limit with  Theorem~\ref{thm:main intro}, and in order to obtain a better lower bound on  $r^{X\to Y}(\e)$ (if at all possible) one would need to work directly with $L_\infty$ estimates, which would likely require a markedly different strategy. Secondly, our proof of Theorem~\ref{thm:main intro} contains contributions to Littlewood--Paley theory that are of significance in their own right. We rely on a novel semigroup argument (yielding as a side-product a new approach to classical results in harmonic analysis even for scalar-valued functions), but it turns out that our strategy is sensitive to the choice of semigroup, despite the semigroup's purely auxiliary role towards the  geometric statement of Theorem~\ref{thm:main intro}. Specifically, our argument fails for the Poisson semigroup (even when $Y$ is a Hilbert space) but does work for the heat semigroup. As a key step, we desire a Littlewood--Paley--Stein estimate for the corresponding  $\mathcal{G}$-function for mappings that take values in uniformly convex Banach spaces. Such a theory has been developed for the Poisson semigroup initially by Xu~\cite{Xu98}, and in a definitive form in important work of Mart\'inez, Torrea and Xu~\cite{MTX}. The availability of~\cite{MTX} has already played a decisive role in purely geometric questions~\cite{LN14}, and it is therefore tempting to also try to use it in our context, but it turns out that obtaining the vector-valued  Littlewood--Paley--Stein inequality for the heat semigroup was left open in~\cite{MTX}. We remedy this by proving new Littlewood--Paley--Stein   $\mathcal{G}$-function estimates for the heat semigroup with values in uniformly convex targets, and using them to prove Theorem~\ref{thm:main intro}. The rest of the Introduction is devoted to a formal explanation of the above overview.

\subsection{Bourgain's strategy for the discretization problem}\label{sec:bourgain} Prior to stating the analytic results that we obtain here as steps towards the proof of Theorem~\ref{thm:main intro}, it would be beneficial  to first present a geometric question due to Bourgain~\cite{Bour:Lip}, known today as {\em Bourgain's discretization problem}, since it served both as  inspiration for our subsequent proofs, as well as one of the reasons for our desire to obtain a lower bound on the modulus $r^{X\to Y}(\e)$. The formal link between the uniform approximation by affine property and Bourgain's discretization problem was clarified in~\cite{LiNaor:13}, but the idea to use semigroup methods in the present context is new, motivated by an approach that Bourgain took within the proof of his discretization theorem in~\cite{Bour:Lip}. As an interesting ``twist," we shall show that a ``vanilla" adaptation of Bourgain's approach to our setting does not work, and in the process of overcoming this difficulty we shall obtain new results in vector-valued Littlewood--Paley theory.

The (bi-Lipschitz) distortion of a metric space $(M,d_M)$ in a metric space $(Z,d_Z)$ is denoted (as usual) by $c_Z(M)\in [1,\infty]$. Thus, the quantity $c_Z(M)$ is the infimum over those $D\in [1,\infty]$ for which there exists an embedding $\f:M\to Z$ and (a scaling factor) $s\in (0,\infty)$ such that $sd_M(x,y)\le d_Z(\f(x),\f(y))\le Dsd_M(x,y)$ for every $x,y\in M$. When $Z$ is a Hilbert space (of the same density character as $M$), $c_Z(M)$ is called the Euclidean distortion of $M$ and is denoted $c_2(M)$.

Fix $n\in \N$. Let $(X,\|\cdot\|_X)$ be an $n$-dimensional normed space and let $(Y,\|\cdot\|_Y)$ be an arbitrary infinite dimensional Banach space. Bourgain's discretization problem asks for a lower estimate on the largest possible $\d\in (0,1)$ such that for any $\d$-net $\mathcal{N}_\d\subset B_X$ of $B_X$ we have $c_Y(X)\le 2c_Y(\mathcal{N}_\d)$.  Thus, the question at hand is to find the coarsest possible discrete approximation of $B_X$ with the property that if it embeds into $Y$ with a certain distortion then the entire space $X$ also embeds into $Y$ with at most twice that distortion (the factor $2$ is an arbitrary choice; see~\cite{GNS} for a generalization).  Bourgain's discretization theorem~\cite{Bour:Lip,GNS} (see also Chapter~9 of the monograph~\cite{Ost13}) asserts that
\begin{equation}\label{eq:bourgain's bound}
\d\ge \exp\left(-c_Y(X)^{Kn}\right)\ge \exp\left(-n^{Kn}\right),
\end{equation}
where $K\in [1,\infty)$ is a universal constant. The second inequality in~\eqref{eq:bourgain's bound} holds true because we always have $c_Y(X)\le \sqrt{n}$ by John's theorem~\cite{Joh48} and Dvoretzky's theorem~\cite{Dvo60}.

The above discretization problem was introduced in~\cite{Bour:Lip} as an alternative (quantitative) approach to an important rigidity theorem of Ribe~\cite{Rib76}. Additional applications to embedding theory appear in~\cite{NS07,GNS,NS16}.  To date, the bound~\eqref{eq:bourgain's bound} remains the best known, even under the additional restriction that $Y$ is uniformly convex.  When $Y$ is uniformly convex, a different proof that $\d \ge \exp(-n^{Kn})$ for some $K=K(Y)\in [1,\infty)$  was obtained in~\cite{LiNaor:13} using the uniform approximation by affine property, and our Theorem~\ref{thm:main intro} yields the stronger estimate $\d \ge \exp(-c_Y(X)^{Kn})$ by~\cite[Remark~1.1]{LiNaor:13}.

 The proof of~\eqref{eq:bourgain's bound} in~\cite{Bour:Lip} starts with a bi-Lipschitz embedding $\f:\mathcal{N}_\delta\to Y$ and proceeds to construct an auxiliary mapping $f:X\to Y$. This is achieved through  {\em Bourgain's almost extension theorem}~\cite{Bour:Lip}, which is a nontrivial step but for the present purposes we do not need to recall the precise properties of $f$ other than to state that $f$ is Lipschitz, compactly supported, and that it well-approximates $\f$ on the net $\mathcal{N}_\d$. Having obtained a mapping $f$ that is defined on all of $X$, \cite{Bour:Lip} proceeds to examine the evolutes $\{P_tf\}_{t\in (0,\infty)}$ of $f$ under the Poisson semigroup $\{P_t\}_{t\in (0,\infty)}$, i.e.,
$$
\forall\, x\in \R^n,\qquad P_tf(x)\eqdef p_t*f(x)=\int_{\R^n} p_t(y)f(x-y)\ud y,
$$
where the Poisson kernel $p_t:\R^n\to [0,\infty)$ is given by
$$
\forall(x,t)\in \R^n\times (0,\infty),\qquad p_t(x)\eqdef \frac{\Gamma\left(\frac{n+1}{2}\right)t}{\left(\pi t^2+\pi|x|^2\right)^{\frac{n+1}{2}}}.
$$
Note that here we implicitly identified $X$ with $\R^n$, with $|\cdot|$ being the standard Euclidean norm on $\R^n$; this issue will become important later, as discussed in  Sections~\ref{sec:X geom}  and~\ref{sec:local approcximation intro} below.

A clever argument (by contradiction) in~\cite{Bour:Lip}  now shows that since $f$ is close to $\f$ on the net $\mathcal{N}_\d$ and $\f$ itself is bi-Lipschitz, provided that the granularity $\d$ of the net $\mathcal{N}_\d$ is small enough there must exist a time $t\in (0,\infty)$ and a location $x\in X$ such that the derivative of the Poisson evolute $P_tf$ at $x$ is a (linear) bi-Lipschitz embedding of $X$ into $Y$, with distortion at most a constant multiple of the distortion of $\f$. Here, since $P_tf$ is obtained from $f$ by averaging and $f$ is Lipschitz, the fact that its derivative is Lipschitz is automatic. The difficulty is therefore to show that this derivative is invertible with good control on the operator norm of its inverse.

If an affine mapping is invertible on a sufficiently fine net of $B_X$ then it is also invertible globally on $X$. So, if one could show that the first order Taylor polynomial of $P_tf$ at $x$ is sufficiently close to $f$  on a sub-ball of $B_X$ (and hence also close to $\f$ on the intersection of that sub-ball with $\mathcal{N}_\d$) whose radius is at least a sufficiently large constant multiple of $\d$, then this would imply the desired (quantitative) invertibility of the derivative of $P_tf$ at $x$. Here, due to scale-invariance, ``sufficiently close" means closeness after normalization by the radius of the sub-ball. This is the reason why a good lower bound on the modulus $r^{X\to Y}(\e)$ is helpful for Bourgain's discretization problem. Of course, one cannot hope to prove the bound~\eqref{eq:bourgain's bound} in this way in full generality, since~\eqref{eq:bourgain's bound} holds for any Banach space $Y$ while by Theorem~\ref{thm:BJLPS quote} we know that for $r^{X\to Y}(\e)$ to be positive we need $Y$ to admit an equivalent uniformly convex norm (thus this approach is doomed to fail when, e.g., $Y=\ell_1$).

Nevertheless, when $Y$ is uniformly convex one could take the fact that Bourgain's strategy does succeed as a hint to  try to use the first order Taylor polynomial of $P_tf$ as the affine mapping that is hopefully close to $f$ on some macroscopically large sub-ball, thus obtaining a lower bound on $r^{X\to Y}(\e)$. This motivates the approach of the present article, eventually leading to Theorem~\ref{thm:main intro}.

An important issue here is that in any such argument one must find a way to use the fact that $Y$ admits an equivalent uniformly convex norm, which by the work of Mart\'inez, Torrea and Xu~\cite{MTX} is equivalent to the validity of a certain $Y$-valued Littlewood--Paley--Stein inequality for the Poisson semigroup; see Section~\ref{sec:LPS} below for a precise formulation. So, since the only underlying assumption on $Y$ is equivalent to a certain $L_q$ estimate, it is natural to use it to bound  the $L_q$ distance of $f(x)$ to the first order Taylor polynomial of $P_tf$ at $x$, for an appropriate measure on the pairs $(x,t)\in X\times (0,\infty)$ of locations and times. By scale-invariance considerations, one arrives at a natural candidate $L_q$ inequality that asserts that an appropriately normalized distance from $f(x)$ to the first order Taylor polynomial of $P_tf$ at $x$ is a Carleson measure; see  Section~\ref{sec:local approcximation intro}  below.

However, in Section~\ref{sec:computations} below we show that the desired $L_q$ inequality does not hold true even when $Y$ is a Hilbert space. The computations of Section~\ref{sec:computations} do suggest that for our purposes it would be better to use the heat semigroup in place of the Poisson semigroup. Unfortunately, the possible validity of the vector-valued Littlewood--Paley--Stein inequality for the heat semigroup for uniformly convex targets was previously unknown, being left open in~\cite{MTX} as part of a more general question that remains open in its full generality. So, as a key tool of independent interest, in the present article we also establish the desired vector-valued Littlewood--Paley--Stein inequality for the heat semigroup (which, by a standard subordination argument, is formally stronger than the corresponding inequality of~\cite{MTX} for the Poisson semigroup); see   Section~\ref{sec:LPS} below. With this tool at hand, we proceed to prove Theorem~\ref{thm:main intro} using the heat semigroup via the strategy outlined above.

\subsection{Geometric invariants}

Theorem~\ref{thm:main intro} is a consequence of the analytic statement that is contained in Theorem~\ref{thm:lip doro intro} below. To formulate it, we need to first introduce  notation related to (well-studied) geometric parameters that govern the ensuing arguments. We also recall the following standard conventions for asymptotic notation. Given $a,b\in (0,\infty)$, the notations
$a\lesssim b$ and $b\gtrsim a$ mean that $a\le cb$ for some
universal constant $c\in (0,\infty)$. The notation $a\asymp b$
stands for $(a\lesssim b) \wedge  (b\lesssim a)$. If we need to allow for dependence on parameters, we indicate this by subscripts. For example, in the presence of an auxiliary parameter $q$, the notation $a\lesssim_q b$ means that $a\le c_qb$, where $c_q\in (0,\infty)$ is allowed to depend only on $q$, and similarly for the notations $a\gtrsim_q b$ and $a\asymp_q b$.

\subsubsection{The geometry of $Y$}\label{sec:Y geom} Despite the fact that in the definition~\eqref{def:uniform convexity} of uniform convexity the parameter $\d\in (0,\infty)$ is allowed to have an arbitrary dependence on $\e\in (0,2]$, the following  deep theorem of Pisier~\cite{Pisier:1975} asserts that  by passing to an equivalent norm one can always ensure that $\d$ is at least a constant multiple of a fixed power of $\e$.

\begin{theorem}[Pisier's renorming theorem]\label{thm:pisier renorming} Suppose that $(Y,\|\cdot\|_Y)$ is a uniformly convex Banach space. Then there exists a norm $\|\cdot\|$ on $Y$ that is equivalent to $\|\cdot\|_Y$ (thus there are $a,b\in (0,\infty)$ such that $a\|y\|_Y\le \|y\|\le b\|y\|_Y$ for all $y\in Y$) and constants $C,q\in [2,\infty)$ such that for every $x,y\in Y$ with $\|x\|,\|y\|\le 1$ we have $\|x+y\|\le 2-\frac{1}{C^q}\|x-y\|^q$.
\end{theorem}
In the literature, a Banach space that admits an equivalent uniformly convex norm is often called a superreflexive Banach space.  Also, the conclusion of Theorem~\ref{thm:pisier renorming} is commonly referred to as the assertion that $Y$ admits an equivalent norm with modulus of uniform convexity of power type $q$.

For norms that satisfy the conclusion of Theorem~\ref{thm:pisier renorming}, Pisier proved~\cite{Pisier:1975} the following important martingale inequality. To state it, recall that a sequence of $Y$-valued random variables $\{M_k\}_{k=1}^\infty$ on a probability space $(\mathscr{S},\mathscr{F},\mu)$ is said to be a martingale if there exists an increasing sequence of sub-$\sigma$-algebras $\mathscr{F}_1\subset \mathscr{F}_2\subset\ldots\subset \mathscr{F}$ such that $\E[M_{k+1}|\mathscr{F}_k]=M_k$ for every $k\in \N$. Here $\E[\,\cdot\, |\mathscr{F}_k]$ stands for the conditional expectation relative to the $\sigma$-algebra $\mathscr{F}_k$ and we are assuming that $M_k\in L_1(\mu;Y)$ for every $k\in \N$, where for $q\in [1,\infty]$ the corresponding vector-valued Lebesgue space $L_q(\mu;Y)$ consists (as usual) of all the $\mathscr{F}$-measurable mappings $f:\mathscr{S}\to Y$ for which $\|f\|_{L_q(\mu;Y)}^q=\int_{\mathscr{S}}\|f\|_Y^q\ud \mu<\infty$.

\begin{theorem}[Pisier's martingale inequality]\label{thm:pisier inequality} Fix $C\in (0,\infty)$ and $q\in [2,\infty)$. Suppose that $(Y,\|\cdot\|_Y)$ is a Banach space such that $\|x+y\|_Y\le 2-\frac{1}{C^q}\|x-y\|_Y^q$ for every $x,y\in Y$ with $\|x\|_Y,\|x\|_Y\le 1$. Then every martingale $\{M_k\}_{k=1}^\infty\subset L_q(\mu;Y)$  satisfies
\begin{equation}\label{eq:pisier's ineq}
\bigg(\sum_{k=1}^\infty \|M_{k+1}-M_k\|_{L_q(\mu;Y)}^q\bigg)^{\frac{1}{q}}\lesssim C\sup_{k\in \N} \|M_k\|_{L_q(\mu;Y)}.
\end{equation}
\end{theorem}
For the proof of~\eqref{eq:pisier's ineq} as stated above (i.e., with the constant factor that appears in the right hand side of~\eqref{eq:pisier's ineq} being proportional to the constant $C$ of the assumption on $Y$), see~\cite[Section~6.3]{MN14} combined with the proof of~\cite[Proposition~7]{BCL94} (the case $q=2$ of this argument is due to K. Ball~\cite{Bal92}).

Inspired by Theorem~\ref{thm:pisier inequality}, Pisier introduced the following terminology in~\cite{Pisier:1986}. Given a Banach space $(Y,\|\cdot\|_Y)$ and $q\ge 2$, the martingale cotype $q$ constant of $Y$, denoted $\mathfrak{m}_q(Y,\|\cdot\|_Y)\in [1,\infty]$ or simply $\mathfrak{m}_q(Y)$ if the norm is clear from the context, is the supremum of $(\sum_{k=1}^\infty \int_{\mathscr{S}}\|M_{k+1}-M_k\|_{Y}^q\ud \mu)^{1/q}$ over all martingales $\{M_k\}_{k=1}^\infty\subset L_q(\mu;Y)$  with $\sup_{k\in \N} \int_{\mathscr{S}}\|M_k\|_{Y}^q\ud \mu=1$ (and over all probability spaces $(\mathscr{S},\mathscr{F},\mu)$). If $\mathfrak{m}_q(Y)<\infty$ then we say that $Y$ has martingale cotype $q$. Pisier's work~\cite{Pisier:1986} yields the remarkably satisfactory characterization that $Y$ admits an equivalent norm whose modulus of uniform convexity has power type $q$ if and only if $Y$ has martingale cotype $q$ (with the relevant constants being within universal constant factors of each other).

The UMD constant of a Banach space $(Y,\|\cdot\|_Y)$, commonly denoted $\beta(Y,\|\cdot\|_Y)\in (0,\infty]$ or simply $\beta(Y)$ if the norm is clear from the context, is the infimum over those $\beta\in (0,\infty]$ such that for every martingale $\{M_k\}_{k=1}^\infty\subset L_2(\mu;Z)$, every $n\in \N$ and every $\e_1,\ldots,\e_{n}\in\{-1,1\}$ we have
$$
\bigg\|M_1+\sum_{k=1}^{n}\e_k(M_{k+1}-M_k)\bigg\|_{L_2(\mu;Y)}\le \beta \|M_{k+1}\|_{L_2(\mu;Y)}.
$$
If $\beta(Y)<\infty$ then $Y$ is said to be a UMD space.  There exist~\cite{Pis75} uniformly convex Banach spaces that are not UMD, and there even exist such Banach lattices~\cite{Bou83,Qiu}. If $Y$ is UMD then it admits an equivalent uniformly convex norm~\cite{Mau75}. As a quantitative form of this assertion (that will be used below), it follows from~\cite[Section~4.4]{HLN} that there exists $2\le q\lesssim \beta(Y)$ such that $\mathfrak{m}_q(Y)\lesssim \beta(Y)^2$.

\subsubsection{The geometry of $X$}\label{sec:X geom} Recalling that $(X,\|\cdot\|_X)$ is an $n$-dimensional (real) normed space, once we fix a Hilbertian norm $|\cdot|$ on $X$ we can identify it (as a real vector space) with $\R^n$. The specific choice of Euclidean structure will be very important later, but at this juncture we shall think of $|\cdot|$ as an arbitrary Hilbertian norm on $X$ and derive an inequality that holds in such (full) generality.

Throughout what follows, the scalar product of two vectors $x,y\in \R^n$ is denoted $x\cdot y\in \R$, the volume of a Lebesgue measurable subset $\Omega\subset \R^n$ is denoted $|\Omega|$, and integration with respect to the Lebesgue measure on $\R^n$ is indicated by $\mathrm{d} x$. The Euclidean unit ball in $\R^n$ is denoted
$B^n=\{x\in \R^n:\ |x|\le 1\}$. Thus $|B^n|=\pi^{n/2}/\Gamma(1+n/2)$. The Euclidean unit sphere is denoted (as usual) $S^{n-1}=\partial B^n=\{x\in \R^n:\ |x|=1\}$, integration with respect to the surface area measure on $S^{n-1}$ is indicated by $\mathrm{d}\sigma$ and, while slightly abusing notation, we denote the surface area of a Lebesgue measurable subset $A\subset S^{n-1}$ by $|A|$. Thus $|S^{n-1}|=n|B^n|=2\pi^{n/2}/\Gamma(n/2)$.

If $\Omega\subset \R^n$ is Lebesgue measurable and has positive finite volume  then it will be convenient to use the following notation for the average over $\Omega$ of an integrable function $f:\Omega\to \R$.
\begin{equation}\label{eq:fint def}
\fint_\Omega f(x)\ud x\eqdef \frac{1}{|\Omega|}\int_\Omega f(x)\ud x.
\end{equation}
Analogously, write $\fint_A \phi\ud \sigma\eqdef \frac{1}{|A|}\int_A \phi(\sigma)\ud \sigma$ for measurable $A\subset S^{n-1}$ and integrable $\phi:A\to \R$.

Using standard notation of asymptotic convex geometry (as in e.g.~\cite{BGVV14}),  denote\footnote{In the literature it is common to suppress the Euclidean norm $|\cdot|$ in this notation, but these quantities do depend on it. A possible more precise notation would have been to use $M(\|\cdot\|_X,|\cdot|)$ and $I_q(\|\cdot\|_X,|\cdot|)$. However, this more cumbersome notation isn't necessary here because the ambient Euclidean norm will always be clear from the context.}
\begin{equation}\label{eq:defIqMq}
M(X)\eqdef \fint_{S^{n-1}} \|\sigma\|_X \ud \sigma \qquad\mathrm{and}\qquad\forall\, q\in (0,\infty],\qquad  I_q(X)\eqdef  \bigg(\fint_{B_X} |x|^q\ud x\bigg)^{\frac{1}{q}} .
\end{equation}
In what follows, the quantity $I_q(X)M(X)$ has an important role. We shall present a nontrivial upper bound on it (for an appropriate choice of Euclidean norm $|\cdot|$ on $X$) as a quick consequence of powerful results from asymptotic convex geometry, and we shall formulate conjectures about the possible availability of better bounds; some of these conjectures may be quite difficult, however, because they relate to longstanding open problems in convex geometry. We postpone these discussions for the moment since it will be more natural to treat them after we present Theorem~\ref{thm:lip doro intro}.

Fixing from now on a target Banach space $(Y,\|\cdot\|_Y)$, for $p\in [1,\infty]$ the corresponding $Y$-valued Lebesgue--Bochner space on a domain  $\Omega\subset \R^n$ will be denoted $L_p(\Omega;Y)$ (the underlying measure on $\Omega$ will always be understood to be the Lebesgue measure). When  $Y=\R$ we shall use the usual simpler notation $L_p(\Omega;\R)=L_p(\Omega)$ for the corresponding scalar-valued function space.

The $Y$-valued heat semigroup on $\R^n$ will be denoted by $\{H_t\}_{t\in (0,\infty)}$. Thus, for every $t\in (0,\infty)$ and $f\in L_1(\R^n;Y)$ the function $H_tf:X\to Y$ is defined by
$$
\forall\, x\in \R^n,\qquad H_tf(x)\eqdef h_t*f(x)=\int_{\R^n} h_t(z)f(x-z)\ud z,
$$
where the corresponding heat kernel $h_t:\R^n\to [0,\infty)$ is given by
$$
\forall(t,x)\in (0,\infty)\times \R^n,\qquad h_t(x)\eqdef \frac{1}{(4\pi t)^{\frac{n}{2}}}e^{-\frac{|x|^2}{4t}}=\frac{1}{t^{\frac{n}{2}}}h_1\!\!\left(\frac{x}{\sqrt{t}}\right).
$$

The first order Taylor polynomial at a point $x\in \R^n$ of a differentiable function $f:\R^n\to Y$ will be denoted below by $\Taylor_x^1(f):\R^n\to Y$. Thus, $\Taylor_x^1(f)$ is the affine function given by
$$
\forall\, x,y\in \R^n,\qquad \Taylor_x^1(f)(y)\eqdef f(x)+(y-x)\cdot \nabla f(x),
$$
where for every $x,z\in \R^n$ we set (as usual) $z\cdot \nabla f(x)=\sum_{j=1}^n z_j\partial_j f(x)=\lim_{\e\to 0} (f(x+\e z)-f(x))/\e$ to be the corresponding $Y$-valued directional derivative of $f$.

Despite the fact that in our setting $\R^n$ is endowed with two metrics, namely those that are induced by $\|\cdot\|_X$ and $|\cdot|$, when discussing Lipschitz constants of mappings from subsets of $\R^n$ to $Y$ we will adhere to the convention that they are exclusively with respect to the metric that is induced by the norm $\|\cdot\|_X$. In particular, we shall use the following notation for a mapping $f:\R^n\to Y$.
$$
\|f\|_{\Lip(X,Y)}\eqdef \sup_{\substack{x,y\in \R^n\\ x\neq y}}\frac{\|f(x)-f(y)\|_Y}{\|x-y\|_X}\qquad\mathrm{and}\qquad \|f\|_{\Lip(B_X,Y)}\eqdef \sup_{\substack{x,y\in B_X\\ x\neq y}}\frac{\|f(x)-f(y)\|_Y}{\|x-y\|_X}.
$$
Hence, if $f$ is differentiable then $\|f\|_{\Lip(X,Y)}=\sup_{z\in \partial B_X} \|z\cdot\nabla f\|_{L_\infty(\R^n;Y)}$.

\subsection{Dorronsoro estimates}\label{sec:local approcximation intro} Our proof of Theorem~\ref{thm:main intro} uses  Theorem~\ref{thm:lip doro intro} below, which shows  that at most scales and locations the first order Taylor polynomial of a heat evolute of a $1$-Lipschitz function $f:\R^n\to Y$ must be close to $f$ itself.  Using standard terminology, our arguments imply that $(t^{-q}\fint_{x+tB_X}\|f(y)-\Taylor_x^1(H_{\gamma t^2}f)(y)\|_Y^q\ud y)\frac{\ud x\ud t}{t}$ is a Carleson measure for a certain $\gamma>0$.

\begin{theorem}\label{thm:lip doro intro} There exists a universal constant $\kappa\in [2,\infty)$ with the following property. Suppose that $q\in [2,\infty)$ and $n\in \N$, and that $(X,\|\cdot\|_X)$ and $(Y,\|\cdot\|_Y)$ are Banach spaces that satisfy $\dim(X)=n$ and $\mathfrak{m}_q(Y)<\infty$. Let $|\cdot|$ be any Hilbertian norm on $X$, thus identifying $X$ with $\R^n$. Define $\gamma,K\in (0,\infty)$ by
\begin{equation}\label{eq:def gamma}
\gamma=\gamma(q,X)\eqdef  \frac{I_q(X)}{\sqrt{n}M(X)}\qquad\mathrm{and}\qquad K=K(q,n,X,Y)\eqdef \kappa \sqrt[4]{n} \cdot \mathfrak{m}_q(Y) \sqrt{I_q(X)M(X)}.
\end{equation}
Then every compactly supported Lipschitz function $f:\R^n\to Y$ satisfies the following estimate.
\begin{equation}\label{eq:our new doro}
\bigg(\int_{\R^n}\int_0^\infty \fint_{x+tB_X}\frac{\|f(y)-\Taylor_x^1(H_{\gamma t^2}f)(y)\|_Y^q}{t^{q+1}}\ud y \ud t \ud x \bigg)^{\frac{1}{q}}\le K |\supp(f)|^{\frac{1}{q}}\|f\|_{\Lip(X,Y)}.
\end{equation}
\end{theorem}

\begin{remark}\label{rem:not poisson} As we discussed in Section~\ref{sec:bourgain}, the analogue of Theorem~\ref{thm:lip doro intro}  for the Poisson semigroup is not true. Specifically, in Section~\ref{sec:computations} we show that if $(\mathcal{H},\|\cdot\|_{\mathcal{H}})$ is a Hilbert space then for every nonconstant Lipschitz function $f:\ell_2^n\to \mathcal{H}$ and every $\gamma\in (0,\infty)$ the following integral diverges.
\begin{equation}\label{eq:poisson infinite}
\bigg(\int_{\R^n}\int_0^\infty \fint_{x+tB^n}\frac{\|f(y)-\Taylor_x^1(P_{\gamma t}f)(y)\|_{\mathcal{H}}^2}{t^{3}}\ud y \ud t \ud x\bigg)^{\frac{1}{2}}.
\end{equation}
Note that~\eqref{eq:our new doro} considers the heat evolute  of $f$ at time $\gamma t^2$ while~\eqref{eq:poisson infinite} considers the Poisson evolute of $f$ at time $\gamma t$ because these time choices are  determined by the requirement that when the argument of $f$ is rescaled the relevant quantities (namely, the left hand side of~\eqref{eq:our new doro} when $q=2$ and the quantity appearing in~\eqref{eq:poisson infinite}) have the same order of homogeneity as the right hand side of~\eqref{eq:our new doro} (when $q=2$).
\end{remark}

Inequality~\eqref{eq:our new doro} is not determined solely by intrinsic geometric properties of $X$ due to the auxiliary choice of the Hilbertian norm $|\cdot|$ on $X$,  which influences the quantities $\gamma$ and $K$ that appear in~\eqref{eq:def gamma}, as well as the meaning of the heat semigroup $\{H_t\}_{t\in [0,\infty)}$. To deduce an intrinsic statement from Theorem~\ref{thm:lip doro intro}, namely a statement that refers only to geometric characteristics of $X$ and $Y$ without any additional (a priori arbitrary) choices, let $\mathscr{A}(X,Y)$ denote the space of affine mappings from $X$ to $Y$. Then, continuing with the notations of Theorem~\ref{thm:lip doro intro}, for every $x\in X$ and $t\in (0,\infty)$ we have
$$
\fint_{x+tB_X}\frac{\|f(y)-\Taylor_x^1(H_{\gamma t^2}f)(y)\|_Y^q}{t^{q+1}}\ud y\ge \inf_{\substack{\Lambda\in \mathscr{A}(X,Y)\\\|\Lambda\|_{\Lip(X,Y)}\le \|f\|_{\Lip(X,Y)}}} \fint_{x+tB_X}\frac{\|f(y)-\Lambda(y)\|_Y^q}{t^{q+1}}\ud y,
$$
where we used the fact that in the above integrand, since $H_{\gamma t^2}f$ is obtained from $f$ by convolution with a probability measure, we have $\|H_{\gamma t^2}f\|_{\Lip(X,Y)}\le \|f\|_{\Lip(X,Y)}$, and consequently also the affine mapping $\Taylor_x^1(H_{\gamma t^2}f)$ has Lipschitz constant at most $\|f\|_{\Lip(X,Y)}$. Therefore~\eqref{eq:our new doro} implies that
\begin{multline}\label{eq:almost intrinsic}
\bigg(\int_{X}\int_0^\infty \inf_{\substack{\Lambda\in \mathscr{A}(X,Y)\\\|\Lambda\|_{\Lip(X,Y)}\le \|f\|_{\Lip(X,Y)}}}  \fint_{x+tB_X}\frac{\|f(y)-\Lambda(y)\|_Y^q}{t^{q+1}}\ud y \ud t \ud x \bigg)^{\frac{1}{q}}\\ \lesssim \sqrt[4]{n}\cdot \mathfrak{m}_q(Y) \sqrt{I_q(X)M(X)} \cdot |\supp(f)|^{\frac{1}{q}}\|f\|_{\Lip(X,Y)}.
\end{multline}

The inequality~\eqref{eq:almost intrinsic} depends on the auxiliary Hilbertian norm $|\cdot|$ on $X$ only through the quantity $I_q(X)M(X)$ that appears on the right hand side of~\eqref{eq:almost intrinsic}, and it is clearly in our interest to choose the Hilbertian structure on $X$ so as to make this quantity as small as possible. Since the definitions of $M(X)$ and $I_q(X)$ in~\eqref{eq:defIqMq} involve averagings, if for some $D\in [1,\infty)$ we have  $\|x\|_X\le|x|\le D\|x\|_X$ for every $x\in X$ then $I_q(X)M(X)\le  D$. By John's theorem~\cite{Joh48}, if $B^n$ is the ellipsoid of maximum volume contained in $B_X$ then  $D\le \sqrt{n}$, so it is always the case that $I_q(X)M(X)\le \sqrt{n}$ for some choice of Hilbertian structure on $X$. Of course, it would be better to choose here the Hilbertian norm $|\cdot|$ so as to minimize $D$, in which case (using a standard differentiation argument~\cite{BL00}) $D$ becomes the  Euclidean distortion $c_2(X)$. So, we always have $I_q(X)M(X)\le c_2(X)\le \sqrt{n}$ for some Euclidean norm $|\cdot|$ on $X$, but it turns out that this estimate is very crude. Firstly, one can improve it (up to constant factors) to the assertion that there exists a Euclidean norm $|\cdot|$ on $X$ for which $I_q(X)M(X)\lesssim_q T_2(X)$, where $T_2(X)\le c_2(X)$ is the Rademacher type $2$ constant of $X$; see Remark~\ref{rem:type 2} below, where the definition of Rademacher type is recalled and this estimate is justified. In terms of the dependence on the dimension $n$,  we have for example $c_2(\ell_\infty^n)=c_2(\ell_1^n)=\sqrt{n}$ while by a direct computation one sees that $I_q(\ell_\infty^n)M(\ell_\infty^n)\asymp_q\sqrt{\log n}$ and $I_q(\ell_1^n)M(\ell_1^n)\asymp_q 1$ (more generally, for $p\in [1,\infty)$ one computes that $I_q(\ell_p^n)M(\ell_p^n)\asymp_{p,q} 1$).\footnote{Here, and throughout the rest of this article, given an integer $n\in \N$ we shall use the nonconventional interpretation of the quantity $\log n$ as being equal to the usual natural logarithm when $n\ge 2$, but equal to $1$ when $n=1$. This is done only for the purpose of ensuring that all the ensuing statements are correct also in the one-dimensional setting without the need to write more cumbersome expressions. Alternatively, one can assume throughout that $n\ge 2$.} Also, if $X$ has a $C$-unconditional basis for some $C\in [1,\infty)$ then $I_q(X)M(X)\lesssim_q C^2\sqrt{\log n}$, as explained in Remark~\ref{rem:unconditional} below.

\begin{conjecture}\label{conj:sqrt log}
For every $n\in \N$ and $q\in [1,\infty)$, every $n$-dimensional normed space $(X,\|\cdot\|_X)$ admits a Hilbertian norm $|\cdot|$ with respect to which  we have $I_q(X)M(X)\lesssim_q \sqrt{\log n}$.
\end{conjecture}
The currently best known upper bound on $I_q(X)M(X)$ in terms of $n=\dim(X)$ occurs when the Hilbertian structure is chosen so as to make $X$ isotropic, where we recall that $X$ is said to be isotropic if the Hilbertian norm $|\cdot|$ satisfies $|B_X|=1$ and there is $L_X\in (0,\infty)$ such that
\begin{equation}\label{eq:def isotropic}
\forall\, y\in \R^n,\qquad \bigg(\int_{B_X} (x\cdot y)^2 \ud x\bigg)^{\frac12}=L_X|y|.
\end{equation}
Every  finite dimensional normed space $X$ admits a unique Hilbertian norm with respect to which it is isotropic. The quantity $L_X$ in~\eqref{eq:def isotropic} is called the isotropic constant of $X$; see the monograph~\cite{BGVV14} for more about isotropicity. In Section~\ref{sec:convex geometry} below we explain how a direct combination of (major) results in convex geometry shows that if $X$ is an $n$-dimensional  isotropic normed space and $q\in [1,\infty)$ then
\begin{equation}\label{eq:current best}
 n\ge q^2\implies I_q(X)M(X)\lesssim (n\log n)^{\frac25}.
\end{equation}

The restriction $n\ge q^2$ in~\eqref{eq:current best} corresponds to the most interesting range of parameters, but in Section~\ref{sec:convex geometry} we also present the currently best known bound when $n\le q^2$; see inequality~\eqref{eq:current best-q} below. We make no claim that~\eqref{eq:current best} is best possible, the main point being that~\eqref{eq:current best} is asymptotically better as $n\to \infty$ than the bound of $\sqrt{n}$ that follows from John's theorem. It is tempting to speculate that the upper bound on $I_q(X)M(X)$ of  Conjecture~\eqref{conj:sqrt log} holds true already when $X$ is isotropic. This refined version of Conjecture~\eqref{conj:sqrt log} seems challenging because, as we explain in Remark~\ref{rem:IqMq Lx} below, we always have $I_q(X)M(X)\gtrsim L_X$, so a positive answer would yield the estimate $L_X\lesssim \sqrt{\log n}$, which would be much stronger than the currently best known~\cite{Kla06} bound $L_X\lesssim \sqrt[4]{n}$ (the longstanding Slicing Problem~\cite{Bou86,Bal88,MP89} asks whether $L_X$ could be bounded from above by a universal constant).

By substituting~\eqref{eq:current best} into~\eqref{eq:almost intrinsic}  we obtain Theorem~\ref{thm:doro intrinsic} below, which is an intrinsic version of Theorem~\ref{thm:main intro}. Of course, any future improvement over~\eqref{eq:current best} (for any Hilbertian structure on $X$) would immediately imply an improved dependence on $n$ in Theorem~\ref{thm:doro intrinsic}.

\begin{theorem}[Intrinsic vector-valued Dorronsoro estimate]\label{thm:doro intrinsic}  Suppose that $q\in [2,\infty)$ and $n\in \N$ satisfy $n\ge q^2$. Let $(X,\|\cdot\|_X)$ and $(Y,\|\cdot\|_Y)$ be Banach spaces that satisfy $\dim(X)=n$ and $\mathfrak{m}_q(Y)<\infty$. Then every compactly supported $1$-Lipschitz function $f:X\to Y$ satisfies
\begin{equation}\label{eq:doro inf intrinsic version}
\bigg(\int_{X}\int_0^\infty \inf_{\substack{\Lambda\in \mathscr{A}(X,Y)\\\|\Lambda\|_{\Lip(X,Y)}\le 1}}  \fint_{x+tB_X}\frac{\|f(y)-\Lambda(y)\|_Y^q}{t^{q+1}}\ud y \ud t \ud x \bigg)^{\frac{1}{q}} \lesssim n^{\frac{9}{20}}\sqrt[5]{\log n}\cdot \mathfrak{m}_q(Y) |\supp(f)|^{\frac{1}{q}}.
\end{equation}
\end{theorem}
The above nomenclature arises from important classical work of Dorronsoro~\cite{Dorro:85}, who obtained Theorem~\eqref{thm:doro intrinsic} when $Y=\R$ (in which case $\mathfrak{m}_q(\R)\asymp 1$ for every $q\ge 2$) and $X=\ell_2^n$, but with much weaker (implicit) dependence on the dimension $n$.   As we shall see in Section~\ref{sec:computations} below, in the special case $X=\ell_2^n$, $Y=\ell_2$ and $q=2$, a more careful analysis yields the validity of~\eqref{eq:doro inf intrinsic version} with the right hand side being dimension independent (in forthcoming work of Danailov and Fefferman~\cite{DF16}, an even sharper result is obtained in this Hilbertian setting, yielding  the precise value of the implicit universal constant).  We do not know if it is possible to obtain a dimension independent version of Theorem~\ref{thm:doro intrinsic} in its full generality, but even if that were possible then it would not influence the statement of Theorem~\ref{thm:main intro} (only the value of the universal constant $c$ will be affected).

In their classical (scalar-valued) form, Dorronsoro  estimates are very influential in several areas, including singular integrals (e.g.~\cite{Jon89}), geometric measure theory (e.g.~\cite{DS93}), local approximation spaces (e.g.~\cite{Triebel:89}), PDE (e.g.~\cite{KMS14}), calculus of variations (e.g.~\cite{KM07}). Our semigroup proof of Theorem~\ref{thm:doro intrinsic} is via a strategy that differs from Dorronsoro's original approach~\cite{Dorro:85} as well as the subsequent approaches of Fefferman~\cite{Fef86}, Jones~\cite{Jon89}, Seeger~\cite{See89}, Triebel~\cite{Triebel:89}, Kristensen--Mingione~\cite{KM07} and Azzam--Schul~\cite{AS12}. Importantly, this semigroup strategy is what makes it possible for us to obtain for the first time the validity of Theorem~\ref{thm:doro intrinsic}  when $Y$ is superreflexive (i.e., $Y$ admits an equivalent uniformly convex norm), thus leading to Theorem~\ref{thm:main intro}. The best previously known result~\cite{HLN} was that a variant of Theorem~\ref{thm:doro intrinsic}  holds true in the more restrictive setting when $Y$ is a UMD Banach space; this was achieved in~\cite{HLN} via a (quite subtle) adaptation of Dorronsoro's original interpolation-based method~\cite{Dorro:85}, and we do not see how to make such an approach apply to superreflexive targets. Theorem~\ref{thm:superreflexive char 1} below shows that the validity of~\eqref{eq:doro inf intrinsic version} for any fixed $n,q,X$ and with any constant multiplying $|\supp(f)|^{1/q}$ in the right hand side of~\eqref{eq:doro inf intrinsic version}  implies that $Y$ is superreflexive, hence Theorem~\ref{thm:doro intrinsic}  as stated above yields a vector-valued Dorronsoro estimate in the maximal possible generality.

\begin{theorem}[Characterization of superreflexivity in terms of a Dorronsoro estimate]\label{thm:superreflexive char 1} The following conditions are equivalent for a Banach space $(Y,\|\cdot\|_Y)$.
\begin{enumerate}
\item $Y$  admits an equivalent uniformly convex norm ($Y$ is superreflexive).
\item There exists $q\in [2,\infty)$ and for every $n\in \N$ there exists $C=C(n,Y)\in (0,\infty)$ such that for every $n$-dimensional normed space $(X,\|\cdot\|_X)$ and every $1$-Lipschitz compactly supported function $f:X\to Y$ we have
    \begin{equation}\label{eq:superreflexiv char}
\int_{X}\int_0^\infty \inf_{\Lambda\in \mathscr{A}(X,Y)} \fint_{x+tB_X}\frac{\|f(y)-\Lambda(y)\|_Y^q}{t^{q+1}}\ud y \ud t \ud x  \le C|\supp(f)|.
\end{equation}
\item There exist $q,C\in [2,\infty)$, $n_0\in \N$ and an $n_0$-dimensional normed space $(X_0,\|\cdot\|_{X_0})$ such that~\eqref{eq:superreflexiv char} holds true for every $1$-Lipschitz compactly supported function $f:X_0\to Y$.
\end{enumerate}
\end{theorem}

It is natural to ask whether or not  one could refine Theorem~\ref{thm:superreflexive char 1} so as to yield a characterization of those Banach spaces $(Y,\|\cdot\|_Y)$ that admit an equivalent norm whose modulus of uniform convexity is of power type $q$, or equivalently that $Y$ has martingale cotype $q$.

\begin{question}\label{Q:deduce q conv}
Does the validity of~\eqref{eq:superreflexiv char} imply that $\mathfrak{m}_q(Y)<\infty$?
\end{question}
While we did not dedicate much effort to try to answer Question~\ref{Q:deduce q conv}, some partial results are obtained in Section~\ref{sec:deduce cotype}  below, including the assertion that if $Y$ is a Banach lattice that satisfies~\eqref{eq:superreflexiv char} then $\mathfrak{m}_{q+\e}(Y)<\infty$ for every $\e\in (0,\infty)$.

\begin{remark}
The literature also contains~\cite{Dorro:85,See89,Triebel:89} Dorronsoro estimates corresponding to local approximation by higher degree polynomials rather than by degree $1$ polynomials as in~\eqref{eq:doro inf intrinsic version}. We made no attempt to study such extensions in our setting, since the goal of the present article is the geometric application of Theorem~\ref{thm:main intro}. Nevertheless,  an inspection of our proofs reveals that they do yield mutatis mutandis vector-valued Dorronsoro estimates for local approximation by polynomials of any degree of sufficiently smooth functions with values in uniformly convex targets.
\end{remark}

\subsection{A local Dorronsoro estimate and $L_q$ affine approximation}\label{sec:local} To explain the link between Theorem~\ref{thm:main intro} and Theorem~\ref{thm:lip doro intro}, note first that Theorem~\ref{eq:superreflexiv char}  deals with functions that are defined on the unit ball $B_X$, while  Theorem~\ref{thm:lip doro intro} deals with functions that are defined globally on all of $\R^n$. So, in order to prove  Theorem~\ref{thm:main intro} we shall first establish Theorem~\ref{thm:local doro} below, which is a localized version of Theorem~\ref{thm:lip doro intro}. To state it, it will be notationally convenient (and harmless) to slightly abuse (but only when $n=1$) the notation for averages that was introduced in~\eqref{eq:fint def} as follows. Given $a,A\in (0,\infty)$ with $a<A$ and $\psi:\R\to \R$ such that the mapping $\rho\mapsto \psi(\rho)/\rho$ is in $L_1([a,A])$, denote the average of $\psi$ with respect to the measure $\frac{\mathrm{d}\rho}{\rho}$ over the interval $[a,A]$  by $\fint_a^A \psi(\rho) \frac{\ud\rho}{\rho}=\frac{1}{\log(A/a)}\int_a^A\psi(\rho) \frac{\ud\rho}{\rho}.$

\begin{theorem}[Local vector-valued Dorronsoro estimate]\label{thm:local doro} There is a universal constant $c\in (0,1/4)$ with the following properties.  Suppose that $q\in [2,\infty)$, $n\in \N$, and that $(X,\|\cdot\|_X)$ and $(Y,\|\cdot\|_Y)$ are Banach spaces that satisfy $\dim(X)=n$ and $\mathfrak{m}_q(Y)<\infty$. Let $|\cdot|$ be any Hilbertian norm on $X$, thus identifying $X$ with $\R^n$. Let $K\in (0,\infty)$ be defined as in~\eqref{eq:def gamma}, and define also
\begin{equation}\label{eq:def T}
T\eqdef \frac{c}{n^{\frac54}\sqrt{I_q(X)M(X)\log n}}.
\end{equation}
Then $T\le 1/(2n)$. Moreover, for every $1$-Lipschitz function $f:B_X\to Y$ and every $r\in (0,T^2]$,
\begin{equation}\label{eq:desired local}
\fint_r^{T} \bigg(\fint_{\left(1-\frac{1}{2n}\right)B_X}\inf_{\substack{\Lambda\in \mathscr{A}(X,Y)\\\|\Lambda\|_{\Lip(X,Y)}\le 2}}\frac{\fint_{x+\rho B_X}\|f(y)-\Lambda(y)\|_Y^q}{\rho^{q}}\ud y\ud x\bigg) \frac{\ud \rho}{\rho} \le \frac{(9Kn)^q}{|\log r|}.
\end{equation}
\end{theorem}
The assertion of Theorem~\ref{thm:local doro}  that $T\le 1/(2n)$ is needed only in order to make the integrals that appear in the left hand side of~\eqref{eq:desired local} well-defined, ensuring that for every $x\in (1-1/(2n))B_X$ and $\rho\in [r,T]$, every point $y$ in the ball $x+\rho B_X$ is also in $B_X$, i.e., $y$ is in the domain of $f$  and the integrand makes sense.  This upper bound on $T$ follows automatically from the restriction $c\le 1/4$, since it is always the case that $I_q(X)M(X)\ge 1/2$, as explained in  Corollary~\ref{coro:lower product} below.  The heart of the matter is therefore to obtain the estimate~\eqref{eq:desired local}. See Section~\ref{sec:assuming} below for the deduction of Theorem~\ref{thm:local doro} from Theorem~\ref{thm:lip doro intro}, where the additional information  in Theorem~\ref{thm:lip doro intro} that the approximating affine function $\Lambda$ is actually a heat evolute is used to show that $\Lambda$ is  $2$-Lipschitz, a property that we shall soon use crucially to deduce Theorem~\ref{thm:main intro}.

The deduction of Theorem~\ref{thm:main intro} from Theorem~\ref{thm:local doro}  is quick. Suppose that $f:B_X\to Y$ is $1$-Lipschitz and fix $\d\in (0,1/2)$. Let $C\in [9,\infty)$ be a (large enough) universal constant ensuring that if we define $r=\exp(-(CKn/\d)^q)$ then $r\le T^2$, where $K$ is defined in~\eqref{eq:def gamma}  and $T$ is defined in~\eqref{eq:def T}; the existence of such a universal constant follows immediately from the fact that $I_q(X)M(X)\ge 1/2$ and the definitions of $K$ and $T$. Now, by~\eqref{eq:desired local} there exists a radius  $\rho\ge r$, a point $x\in B_X$ with $x+\rho B_X\subset B_X$, and an affine mapping $\Lambda:X\to Y$ with $\|\Lambda\|_{\Lip(X,Y)}\le 2$ such that
\begin{equation}\label{eq:Lq approx deduced}
\bigg(\fint_{x+\rho B_X}\|f(y)-\Lambda(y)\|_Y^q\ud y\bigg)^{\frac{1}{q}}\le \d \rho.
\end{equation}

Next, the (simple) argument of~\cite[Section~2.1]{HLN} shows that given $\e\in (0,1/2)$, if the $L_q$-closeness of $f$ to $\Lambda$ that appears in~\eqref{eq:Lq approx deduced} is exponentially small in $n$  then necessarily $\|f(y)-\Lambda(y)\|_Y\le \e\rho$ for {\em every} $y\in x+\rho B_X$. Specifically, this holds true if $\d=(\eta \e)^{1+n/q}$ for a sufficiently small universal constant $\eta\in (0,1)$. Briefly, the reason for this fact is that since $\|f\|_{\Lip(B_X,Y)}\le 1$ and $\|\Lambda\|_{\Lip(X,Y)}\le 2$ we have $\|f-\Lambda\|_{\Lip(B_X,Y)}\le 3$, and consequently if $\|f(y_0)-\Lambda(y_0)\|_Y$ were larger than $\e\rho$ for some $y_0\in x+\rho B_X$ then it would follow that $\|f(y)-\Lambda(y)\|_Y$ is larger than a constant multiple of $\e\rho$ on a sub-ball of $x+\rho B_X$ of radius that is at least a constant multiple of $\e\rho$, thus making the left hand side of~\eqref{eq:Lq approx deduced} be greater than $(\eta\e)^{1+n/q}\rho$ for some universal constant $\eta\in (0,1)$. So, by choosing $\d=(\eta \e)^{1+n/q}$, recalling that we defined $r=\exp(-(CKn/\d)^q)$ and recalling also the definition of  $K$ in~\eqref{eq:def gamma}, if we set $\mathfrak{a}=\eta/(C\kappa)$ (with $\kappa$ being the universal constant of Theorem~\ref{thm:lip doro intro})  we obtain the following refined version of Theorem~\ref{thm:main intro}.

\begin{corollary}\label{cor:refined main} There exist universal constant $\mathfrak{a}\in (0,1)$  such that for every $q\in [2,\infty)$, $n\in \N$, and every two Banach spaces $(X,\|\cdot\|_X)$ and $(Y,\|\cdot\|_Y)$ that satisfy $\dim(X)=n$ and $\mathfrak{m}_q(Y)<\infty$, if $|\cdot|$ is any Hilbertian norm on $X$ and $\e\in (0,1/2)$ then
$$
r^{X\to Y}(\e)\ge \exp\left(-\frac{\left(  n^{\frac54} \sqrt{I_q(X)M(X)}\cdot \mathfrak{m}_q(Y)\right)^q}{(\mathfrak{a}\e)^{n+q}}\right).
$$
Consequently, since an appropriate choice of the Hilbertian norm $|\cdot|$ ensures that $I_q(X)M(X)\le \sqrt{n}$,  for every $\e\in (0,1/2)$  we have
\begin{equation}\label{eq:our better}
r^{X\to Y}(\e)\ge  \exp\left(-\frac{n^{\frac{3q}{2}} \mathfrak{m}_q(Y)^q  }{(\mathfrak{a}\e)^{n+q}}\right).
\end{equation}
\end{corollary}

\begin{remark}\label{rem:use lip Lambda} Observe that the fact that $\|\Lambda\|_{\Lip(X,Y)}\le 2$ in~\eqref{eq:Lq approx deduced} was used crucially in the above deduction of Corollary~\ref{cor:refined main} from Theorem~\ref{thm:local doro}. Any universal constant in place of $2$ would work just as well for the purpose of this deduction, but an upper bound on $\|\Lambda\|_{\Lip(X,Y)}$ that grows to $\infty$ with $n$ would result in an asymptotically weaker lower bound in Theorem~\ref{thm:main intro}.
\end{remark}

\begin{remark}\label{rem:Lq approx} In the setting of Theorem~\ref{thm:local doro}, notation for  an $L_q$ affine approximation modulus was introduced as follows in~\cite[Definition~1]{HLN}. Given $\d\in (0,1)$ let $r_q^{X\to Y}(\d)$ be the supremum over those $r\in [0,1]$ such that for every $1$-Lipschitz function $f:B_X\to Y$ there exists a radius  $\rho\ge r$, a point $x\in B_X$ with $x+\rho B_X\subset B_X$, and an affine mapping $\Lambda:X\to Y$ with $\|\Lambda\|_{\Lip(X,Y)}\le 3$ such that the estimate~\eqref{eq:Lq approx deduced} is satisfied. The constant $3$ of the definition of this modulus was chosen in~\cite{HLN} essentially arbitrarily (any universal constant that is at least $1+2\d$ would work equally well for the purposes of~\cite{HLN}). The above argument shows that in the setting of Theorem~\ref{thm:local doro} we have
\begin{equation*}\label{eq:rq UC}
\forall\, \d\in (0,1),\qquad r_q^{X\to Y}(\d)\ge \exp\left(-\frac{(CKn)^q}{\d^q}\right).
\end{equation*}
 As we mentioned earlier in the Introduction, in Section~\ref{sec:computations} below we show that if $X=\ell_2^n$, $Y=\ell_2$ and $q=2$ then~\eqref{eq:our new doro} holds true with $K$ replaced by a universal constant. This translates to the validity of~\eqref{eq:desired local} with $K$ replaced by a universal constant. By reasoning as above, we therefore deduce that there exists a universal constant $\mathfrak{c}\in (1,\infty)$ for which the following lower bound holds true.
 \begin{equation}\label{eq:r2 hilbert}
\forall\, \d\in (0,1),\qquad r_2^{\ell_2^n\to \ell_2}(\d)\ge \exp\left(-\frac{\mathfrak{c} n^2}{\d^2}\right).
 \end{equation}
This improves over the bound $r_2^{\ell_2^n\to \ell_2}(\d)\ge \exp(-\mathfrak{c} (n\log n)^2/\d^2)$ of~\cite{HLN}. The modulus $r_2^{\ell_2^n\to \ell_2}(\cdot)$ is currently not known to have geometric applications, but it is a natural Hilbertian  quantity and it would be of interest to determine its asymptotic behavior; see also Question~9 in~\cite{HLN}.
\end{remark}

\begin{remark} It is instructive  to examine the reason for the doubly exponential dependence on $n$ in Theorem~\ref{thm:main intro}. Continuing with the notation and hypotheses of Theorem~\ref{thm:local doro}, the above reasoning  proceeded in two steps. The first step deduced  the $L_q$ approximation~\eqref{eq:Lq approx deduced} on a sub-ball of radius at least $\exp(-(CKn/\d)^q)$. The second step argued that if $\d=(\eta \e)^{1+n/q}$ (with $\eta\in (0,1)$ being a universal constant) then the average $(\d \rho)$-closeness to $\Lambda$ that is exhibited in~\eqref{eq:Lq approx deduced} automatically ``upgrades" to $(\e\rho)$-closeness in $L_\infty(x+\rho B_X;Y)$. This second step trivially requires $\d$ to decay like $\e^{c_qn}$: Take for example $\Lambda=0$ and $f$ to be a real-valued $1$-Lipschitz function that vanishes everywhere except for a ball of radius $\e \rho$ on which its maximal value equals $2\e\rho$. A substitution of an exponentially decaying $\d$ into $r_q^{X\to Y}(\d)$ would at best result in a doubly exponential lower bound on $r^{X\to Y}(\e)$ if $r_q^{X\to Y}(\d)$ decays at least exponentially fast in $-1/\d^{c_q}$. This must indeed be the case in general, since there exist examples of spaces $X,Y$ that satisfy the assumptions of Theorem~\ref{thm:local doro} yet $r^{X\to Y}(\d)\le \exp(-(c/\d)^q)$, where $c\in (0,1)$ is a universal constant. Specifically, by~\cite[Lemma~16]{HLN} this holds for $X=\ell_\infty^n$ and $Y=\ell_q$. So, in order to obtain a better bound in Theorem~\ref{thm:main intro} it seems that one should somehow argue about $L_\infty$ bounds directly, despite the fact that in our setting the only assumption on $Y$  is that it has martingale cotype $q$, which is by~\cite{MTX} equivalent to an $L_q$ Littlewood--Paley inequality. This could be viewed as an indication that perhaps Theorem~\ref{thm:main intro} cannot be improved in the stated full generality, though we leave this very interesting question open.
\end{remark}

\subsection{Littlewood--Paley--Stein theory}\label{sec:LPS} An elegant and useful theorem of Mart\'inez, Torrea and Xu~\cite{MTX} asserts that a Banach space $(Y,\|\cdot\|_Y)$ has martingale cotype $q\ge 2$ if and only if we have
\begin{equation}\label{eq:MTX quote}
\forall\, n\in \N,\ \forall\, f\in L_q(\R^n;Y),\qquad  \|\mathcal{G}_qf\|_{L_q(\R^n;Y)}\lesssim_{n,Y} \|f\|_{L_q(\R^n;Y)},
\end{equation}
where $\mathcal{G}_qf:\R^n\to Y$ is the (generalized)  Littlewood--Paley--Stein $\mathcal{G}$-function which is defined by
\begin{equation*}\label{eq:def G function}
\forall\, x\in \R^n,\qquad \mathcal{G}_qf(x)\eqdef \bigg(\int_0^\infty \left\|t\partial_t P_tf(x)\right\|_Y^q \frac{\ud t}{t}\bigg)^{\frac{1}{q}},
\end{equation*}
with $\{P_t\}_{t\in (0,\infty)}$ being the Poisson semigroup. More generally, the Littlewood--Paley--Stein theory of~\cite{MTX} applies to abstract semigroups provided that they are so-called {\em subordinated diffusion semigroups}, of which the heat semigroup (in contrast to Poisson) is {\em not} an example.

Problem~2 of~\cite{MTX} asks whether or not~\eqref{eq:MTX quote} holds true when $Y$ has martingale cotype $q$ for any diffusion semigroup in the sense of Stein~\cite{Stein:topics} (with the implicit constant in~\eqref{eq:MTX quote} being allowed to also depend on the semigroup), in which case~\eqref{eq:MTX quote} would apply to the heat semigroup as well. This question remains open in full generality, with the best known partial result being due to~\cite{MTX}, where it is shown that the answer is positive if $Y$ is a Banach lattice of martingale cotype $q$. Here we obtain the following positive answer in the important special case of the heat semigroup.

\begin{theorem}[Temporal Littlewood--Paley--Stein inequality for the heat semigroup]\label{thm:LPtemporal} Fix $q\in [2,\infty)$ and $n\in \N$. Suppose that $(Y,\|\cdot\|_Y)$ is a Banach space that admits an equivalent norm with modulus of uniform convexity of power type $q$. Then for every $f\in L_q(\R^n;Y)$ we have
\begin{equation}\label{eq:temporal heat intro}
  \bigg(\int_0^\infty\Norm{t\partial_t H_t f}{L_q(\R^n;Y)}^q\frac{\ud t}{t}\bigg)^{\frac{1}{q}}
  \lesssim\sqrt{n}\cdot\mathfrak{m}_{q}(Y)\Norm{f}{L_q(\R^n;Y)}.
\end{equation}
\end{theorem}
By considering the direct sum of all the heat semigroups on $\R^n$ as $n$ ranges over $\N$, a positive answer to the abstract question~\cite[Problem~2]{MTX} would imply a dimension-independent bound
in~\eqref{eq:temporal heat intro}. However, at present it remains open whether or not the (mild) dimension dependence that appears in~\eqref{eq:temporal heat intro} can be removed altogether. While this is an interesting open question, it isn't relevant to the investigations of the present article because in order to prove Theorem~\ref{thm:lip doro intro} we actually need the following new result about the spatial derivatives of the heat semigroup.

\begin{theorem}[Spatial Littlewood--Paley--Stein inequalities for the heat semigroup]\label{thm:LPspatial} Fix $q\in [2,\infty)$ and $n\in \N$. Suppose that $(Y,\|\cdot\|_Y)$ is a Banach space that admits an equivalent norm with modulus of uniform convexity of power type $q$. Then for every $\vec{f}\in \ell_q^n(L_q(\R^n;Y))$ we have
\begin{equation}\label{eq:first spatial}
  \bigg(\int_0^\infty\left\|\sqrt{t}\operatorname{div}H_t\vec{f}\,\right\|_{L_q(\R^n;Y)}^q\frac{\ud t}{t}\bigg)^{\frac{1}{q}}
  \lesssim\sqrt{n}\cdot\mathfrak{m}_{q}(Y)\fint_{S^{n-1}}\Big\|\sigma\cdot\vec{f}\,\Big\|_{L_q(\R^n;Y)}\ud\sigma,
\end{equation}
where~\eqref{eq:first spatial} uses the following (standard) notation, in which $\vec{f}=(f_1,\ldots,f_n)$.
\begin{equation}\label{eq:def div}
H_t\vec{f}\eqdef(H_t f_1,\ldots,H_t f_n)\qquad\mathrm{and}\qquad  \operatorname{div}H_t\vec{f}\eqdef\sum_{j=1}^n\partial_j(H_t f_j).
\end{equation}
Moreover, for every $f\in L_q(\R^n;Y)$ and $z\in \R^n$ we have
\begin{equation}\label{eq:second spatial}
\bigg(\int_0^\infty\left\|\sqrt{t}(z\cdot\nabla)H_tf\right\|_{L_q(\R^n;Y)}^q\frac{\ud t}{t}\bigg)^{\frac{1}{q}}
  \lesssim\abs{z}\mathfrak{m}_q(Y)\Norm{f}{L_q(\R^n;Y)}.
\end{equation}
\end{theorem}

We stated Theorem~\ref{thm:LPtemporal} despite the fact that only Theorem~\ref{thm:LPspatial} is needed in our proof of Theorem~\ref{thm:lip doro intro} because it is directly related to~\cite[Problem~2]{MTX}, and moreover we shall establish Theorem~\ref{thm:LPtemporal} without much additional effort. It is well known that Littlewood--Paley--Stein inequalities as above for the time derivatives and spatial derivatives often come hand-in-hand, and indeed a version of Theorem~\ref{thm:LPspatial} for the spatial derivatives of the Poisson semigroup was deduced in~\cite{MTX}. However, for abstract diffusion semigroups as in~\cite[Problem~2]{MTX}  there is no intrinsic notion of  spatial derivatives.

The temporal Littlewood--Paley--Stein inequality~\eqref{eq:MTX quote} for the Poisson semigroup was previously used for geometric purposes in~\cite{LN14}. In that setting, the Poisson semigroup sufficed due to parabolic scaling that was afforded by the geometry of the Heisenberg group. Parabolic scaling also makes our proof of  Theorem~\ref{thm:lip doro intro} go through, but this time it occurs due to our use of the heat semigroup in place of the Poisson semigroup. The only step in our proof that uses the heat semigroup and fails for the Poisson semigroup occurs in equation~\eqref{eq:here heat} below, which was inspired by the proof of~\cite[Lemma~2.5]{LN14}. An inspection of that step reveals that we could have also worked with the fractional semigroup $t\mapsto\exp(-t(-\Delta)^\alpha)$ for any $\alpha\in (1/2,1)$ in place of the heat semigroup. In this fractional (hence subordinated) setting one can prove the required version of Theorem~\ref{thm:LPspatial} by adapting (in the spatial setting) the  argument of~\cite{MTX}, though for our purposes one needs to also take care to derive polynomial dependence on dimension, which we checked is possible but this leads to a significantly more involved and less natural argument than the one that we obtain below for the heat semigroup (Theorem~\ref{thm:LPspatial} formally implies the corresponding statement for these fractional semigroups as well as the Poisson semigroup because these semigroups are subordinated to the heat semigroup).

\subsection{Comparison to previous work}\label{sec:previous} The previously best known bound~\cite{LiNaor:13} in the setting of Theorem~\ref{thm:main intro}   was that there exists a universal constant $C\in (0,\infty)$ such that for every integer $n\ge 2$ if $(X,\|\cdot\|_X)$ in an $n$-dimensional normed space and $(Y,\|\cdot\|_Y)$ is a Banach space whose modulus of uniform convexity is of power type $q$ for some $q\in [2,\infty)$ then for every $\e\in (0,1/2]$ we have
\begin{equation}\label{eq:quote LN}
 r^{X\to Y}(\e)\ge \exp\left(-\frac{(Cn)^{20(n+q)}\mathfrak{m}_q(Y)^q\log\left(\frac{1}{\e}\right)}{\e^{2n+2q-2}}\right).
\end{equation}
So, our new bound~\eqref{eq:our better}  is  stronger than~\eqref{eq:quote LN} both as $\e\to 0$ and as $n\to \infty$.

Under the more stringent assumption that $(Y,\|\cdot\|_Y)$ is a UMD Banach space with UMD constant $\beta=\beta(Y)$, in~\cite{HLN} it was shown that for every $\e\in (0,1/2]$ we have
\begin{equation}\label{eq:beta HLN}
r^{X\to Y}(\e)\ge \exp\left(-\frac{(\beta n)^{c\beta}}{\e^{c(n+\beta)}}\right),
\end{equation}
where $c\in (0,\infty)$ is a universal constant. Since, as we recalled in the end of Section~\ref{sec:Y geom},  there exists $2\le q\lesssim \beta$ for which $\mathfrak{m}_q(Y)\lesssim \beta^2$, the estimate~\eqref{eq:beta HLN} is weaker than our new bound~\eqref{eq:our better}.

No vector-valued Dorronsoro estimate was previously known for  uniformly convex targets. So, Theorem~\ref{thm:doro intrinsic} and Theorem~\ref{thm:local doro} that are obtained here are qualitatively new statements that answer Question~8 in~\cite{HLN} and are definitive due to Theorem~\ref{thm:superreflexive char 1} and the results of Section~\ref{sec:deduce cotype} below. Previously, vector-valued Dorronsoro estimates were known only when  the target Banach space $Y$ is UMD. Specifically, by~\cite[Lemma~10]{HLN} and~\cite[Theorem~19]{HLN} there exists a universal constant $\kappa\in (0,\infty)$ with the following properties. Suppose that $(X,\|\cdot\|_X)$ is an $n$-dimensional normed space, equipped with the Hilbertian norm $|\cdot|$ that is induced by the ellipsoid of maximal volume that is contained in $B_X$ (John position), so as to identify $X$ with $\R^n$. Suppose also that $(Y,\|\cdot\|_Y)$ is a UMD Banach space with $\beta=\beta(Y)$ and $f:\R^n\to Y$ is $1$-Lipschitz and compactly supported. Then
\begin{equation}\label{eq:doro UMD}
\bigg(\int_{\R^n}\int_0^\infty \inf_{\substack{\Lambda\in \mathscr{A}(X,Y)\\\|\Lambda\|_{\Lip(X,Y)\le 1}}} \fint_{x+tB^n}\frac{\|f(y)-\Lambda(y)\|_Y^{\kappa\beta}}{t^{\kappa\beta+1}}\ud y \ud t \ud x \bigg)^{\frac{1}{\kappa\beta}} \lesssim
n^{\frac{5}{2}}\beta^{15}|\supp(f)|^{\frac{1}{\kappa\beta}}.
\end{equation}

The relatively large power of the UMD constant $\beta$ that occurs in the right hand side of~\eqref{eq:doro UMD} reflects the fact that the proof of~\eqref{eq:doro UMD} in~\cite{HLN} is quite involved, in particular using the UMD property of $Y$ fifteen times (also through equivalent formulations of the UMD property, like the boundedness of the $Y$-valued Hilbert transform). In contrast, the proof of  the Dorronsoro estimate~\eqref{eq:doro inf intrinsic version} that we obtain here does not only address  the correct generality (of all uniformly convex targets), but it also achieves this by a new and  simpler argument that seems like the correct approach to the problem at hand. This proof is different from, but not any more complicated than, the existing ones for $Y=\R$, and is of interest even as a new route to the original result of Dorronsoro \cite{Dorro:85}.

Among the approaches to Dorronsoro-type estimates that appeared in the literature, some seem to be inherently Hilbertian, such as Fefferman's identity (77) in~\cite{Fef86} or the argument of Kristensen and Mingione in~\cite{KM07}. Nevertheless, this point of view yields more precise estimates (actually, identities) in the Hilbertian setting~\cite{DF16}, where the problem of understanding quantitative differentiation remains  open, with the best known bound on, say, $r^{\ell_2^n\to \ell_2}(1/4)$ currently having the same asymptotic form as the general bound of Theorem~\ref{thm:main intro}. Our proof of Theorem~\ref{thm:lip doro intro} is closer in spirit to the approaches of Jones~\cite{Jon89}, Seeger~\cite{See89} and Triebel~\cite{Triebel:89} (which are related to each other in terms of the underlying principles), though we do not see how to use these approaches to obtain a proof of Theorem~\ref{thm:lip doro intro}. As we stated earlier, despite significant effort it seems that a vector-valued adaptation of Dorronsoro's original strategy~\cite{Dorro:85} requires the UMD property, as in~\cite{HLN}. Finally, a possible direction for future research would be to investigate whether the approaches of Schul~\cite{Sch09}, Azzam--Schul~\cite{AS12,AS14} and Li--Naor~\cite{LN14} could be adapted  so as to yield a Dorronsoro-type estimate for uniformly convex targets, though (if at all possible) it seems that this route could at best yield a version of~\eqref{eq:doro inf intrinsic version}  with weaker dependence on $n$ that is insufficient for proving Theorem~\ref{thm:main intro} (a version of~\eqref{eq:doro inf intrinsic version} that leads to a bound similar to~\eqref{eq:quote LN} might be within reach through such an approach).

Our estimate~\eqref{eq:doro inf intrinsic version}  is also a quantitative improvement of~\eqref{eq:doro UMD} because, as we recalled in the end of Section~\ref{sec:Y geom}, there exists a universal constant $\kappa\in (0,\infty)$ such that when $Y$ is a  UMD Banach space with $\beta=\beta(Y)$, if we set $q=\kappa\beta$ then we have $\mathfrak{m}_q(Y)\lesssim \beta^2$. Consequently, it follows formally from~\eqref{eq:doro inf intrinsic version} that a variant of~\eqref{eq:doro UMD} holds true (using intrinsic averaging; see Section~\ref{sec:intrinsic} below) with the quantity $n^{5/2}\beta(Y)^{15}$ in the right hand side replaced by the smaller quantity $n^{9/20}\sqrt[5]{\log n}\cdot \beta^2$.

\subsubsection{Intrinsic averages}\label{sec:intrinsic}
A convenient advantage of~\eqref{eq:doro inf intrinsic version} over~\eqref{eq:doro UMD} is that the averaging in the left hand side of~\eqref{eq:doro inf intrinsic version} occurs over the intrinsic balls $x+tB_X$ while the averaging in~\eqref{eq:doro UMD} is over the auxiliary Euclidean balls $x+tB^n$. This difference reflects a conceptual rather than technical geometric difficulty that arose in~\cite{HLN} and is circumvented here altogether due to the use of the heat semigroup rather than a more complicated (but natural) operator that was used in~\cite{HLN}. As we have seen in Section~\ref{sec:local}, the fact that we can work here with averages over intrinsic balls leads to a quick and  direct deduction of our bound on $r^{X\to Y}(\e)$, while the fact that in~\eqref{eq:doro UMD} the  averages are over Euclidean balls requires an additional argument that is carried out in~\cite[Section~3.1]{HLN} in order to relate~\eqref{eq:doro UMD} to the uniform approximation by affine property with the stated asymptotic dependence on $n$. In~\cite{HLN} (as well as in~\cite{Dorro:85,Fef86,DF16}) the affine function $\Lambda$ that is used to approximate $f$ on the ball $x+t B^n$ is $(\Proj\otimes \mathrm{Id}_Y)f$, where $\mathrm{Id}_Y$ is the identity on $Y$ and $\Proj$ is the orthogonal projection from $L_2(x+tB^n)$ onto its linear subspace $\mathscr{A}(X,\R)\cap L_2(x+tB^n)$ that consists of all the restrictions to $x+uB^n$ of affine functions from $\R^n$ to $\R$. Lemma~10 of~\cite{HLN} asserts that $\|(\Proj\otimes \mathrm{Id}_Y)f\|_{\Lip(x+tB^n,Y)}\le \|f\|_{\Lip(x+tB^n,Y)}$, but its proof uses the rotational symmetry of the Euclidean ball $x+tB^n$ and the analogous statement is unknown with $x+tB^n$ replaced by $x+tB_X$. The need to address such issues does not arise in our approach because the Lipschitz constant of a heat evolute can be easily bounded without any need for additional symmetries of $B_X$. Nevertheless, it would be independently interesting to understand Question~\ref{Q:proj lip} below since the operator $\Proj$ is a natural object whose Lipschitz properties are equivalent to the potential availability of certain approximate distributional symmetries of high dimensional centrally symmetric convex bodies.

\begin{question}\label{Q:proj lip} Let $(X,\|\cdot\|_X)$ be an $n$-dimensional normed space, equipped with a Hilbertian norm $|\cdot|$ with respect to which it is isotropic with isotropic constant $L_X$, i.e., $|B_X|=1$ and~\eqref{eq:def isotropic} holds true. Consider the orthogonal projection $\Proj$ from $L_2(B_X)$ onto the subspace of affine mappings $\mathscr{A}(X,\R)\cap L_2(B_X)$. Thus, for every $f\in L_2(B_X)$ and $x\in X$ we have
\begin{equation}\label{eq:def proj}
\Proj f (x)=\int_{B_X} f(z)\ud z+\frac{1}{L_X^2}\sum_{j=1}^n x_j\int_{B_X} z_jf(z)\ud z.
\end{equation}
The formula~\eqref{eq:def proj} also makes sense when $f\in L_1(B_X;Y)$ for any Banach space $(Y,\|\cdot\|_Y)$, and we shall use the notation $\Proj f$ in this case as well  (i.e.,   slightly abusing notation by identifying $\Proj\otimes \mathrm{Id}_Y$ with $\Proj$). Is it true that for every Banach space $Y$ and every Lipschitz function $f:B_X\to Y$ we have $\|\Proj f\|_{\Lip(B_X,Y)}\lesssim \|f\|_{\Lip(B_X,Y)}$?

In Section~\ref{sec:wass} below we show that it suffices to treat the above question when $Y=\R$, i.e., the following operator norm identity holds true  for every Banach space $(Y,\|\cdot\|_Y)$.
\begin{equation}\label{eq:reduction to 1 dim}
\|\Proj\|_{\Lip(B_X,Y)\to   \Lip(B_X,Y)}= \|\Proj\|_{\Lip(B_X,\R)\to   \Lip(B_X,\R)}.
\end{equation}
Moreover, the quantity  $\|\Proj\|_{\Lip(B_X,\R)\to   \Lip(B_X,\R)}$ has the following geometric interpretation.

For every $x\in X\setminus\{0\}$ let  $\mu_x^+,\mu_x^-$ be the probability measures supported on $B_X$ whose densities are given for every $y\in X$ by
\begin{equation}\label{eq:def muxplus}
\mathrm{d}\mu_{x}^+(y)\eqdef \frac{\max\{x\cdot y,0\}\1_{B_X}(y)}{\frac12\int_{B_X}|x\cdot z|\ud z}\ud y\qquad\mathrm{and}\qquad \mathrm{d}\mu_x^-(y)\eqdef \frac{\max\{-x\cdot y,0\}\1_{B_X}(y)}{\frac12\int_{B_X}|x\cdot z|\ud z}\ud y.
\end{equation}
Alternatively, if we consider the half spaces $H_x^+\eqdef\{y\in \R^n:\ x\cdot y\ge 0\}$,  $H_x^-\eqdef\{y\in \R^n:\ x\cdot y\le 0\}$ then $\mu_x^+, \mu_x^-$ are supported on the half balls $H_x^+\cap B_X, H_x^-\cap B_X$, respectively, and each of them has density $y\mapsto |x\cdot y|$ on the corresponding half ball. We are interested in the extent to which the convex body $B_X$ is approximately symmetric about the hyperplane $x^\perp\subset \R^n$ in the sense that the probability measures $\mu_x^+,\mu_x^-$ are distributionally close to each other. In Section~\ref{sec:wass} we show that
\begin{equation}\label{eq:wasserstein symmetry}
\|\Proj\|_{\Lip(B_X,\R)\to   \Lip(B_X,\R)}\asymp \sup_{x\in X\setminus \{0\}} \frac{|x|}{L_X\|x\|_X}\mathsf{W}_1^{\|\cdot\|_X}(\mu_x^+,\mu_x^-),
\end{equation}
where $\mathsf{W}_1^{\|\cdot\|_X}(\cdot,\cdot)$ denotes the Wasserstein-$1$ (transportation cost) distance  induced by $\|\cdot\|_X$. Consequently,  $\|\Proj\|_{\Lip(B_X,\R)\to   \Lip(B_X,\R)}=O(1)$ if and only if the Wasserstein-$1$ distance between $\mu_x^+$ and $\mu_x^-$ is at most a constant multiple of $L_X\|x\|_X/|x|$ for every $x\in X\setminus\{0\}$. The latter  statement implies that $B_X$ has the following  ``Wasserstein symmetry." By~\cite{MP89} we have $|x|/(nL_X)\lesssim \|x\|_X\lesssim |x|/L_X$ for every $x\in X$. Hence~\eqref{eq:wasserstein symmetry} implies that if $\|\Proj\|_{\Lip(B_X,\R)\to   \Lip(B_X,\R)}=O(1)$  then for every $x\in X\setminus \{0\}$ such that $\|x\|_X$ is not within $O(1)$ factors of the Euclidean norm $|x|/L_X$, the Wasserstein-$1$ distance between $\mu_x^+$ and $\mu_x^-$ must necessarily be $o(1)$.

\end{question}

\subsubsection{Beyond Banach spaces} Quantitative differentiation is a widely studied topic of importance to several mathematical disciplines, often (but not only) as a tool towards proofs of rigidity theorems. Given an appropriate (case-specific) replacement for the notation of ``affine mapping," one can formulate notions of ``differentiation" in many settings that do not necessarily involve linear spaces; examples of such ``qualitative" metric differentiation  results include~\cite{Pan89,Kir94,Che99,Pau01,Kei04,LN06,CK06,CK06,CK10-annals,CK10-inventiones}. Corresponding results about quantitative differentiation, which lead to refined (often quite subtle and important) rigidity results can be found in~\cite{Ben85,JLS96,Nao98,Mat99,CKN09,LR10,Pen11-1,Pen11-2,CKN11,LS11,Che12,EFW12,MN13,EFW13,EMR13,CN13,CN13-2,AS14,Li14,EMR15}. Due to the prominence of this topic and the fact that  many of the quoted results are probably not sharp, it would be of interest to develop new methods to prove sharper quantitative differentiation results. While the argument of~\cite{HLN} yielded the best-known bound for UMD Banach spaces, the methods of~\cite{HLN} relied extensively on the underlying linear structure. The present article uses the linear structure as well, but it suggests that heat flow methods may be useful for obtaining  quantitative differentiation results in situations where heat flow makes sense but the underlying metric space is not  a Banach space. It seems to be worthwhile to investigate whether ``affine approximations" (appropriately defined) of an appropriate evolute (which is a regularized object)  must  be close to the initial mapping on some macroscopically large ball. We did not attempt to investigate this approach when the underlying spaces are not Banach spaces, but we believe that this is an open-ended yet  worthwhile direction  for future research.

\section{The Lipschitz constant of heat evolutes}

In this short section we shall establish an estimate on the Lipschitz constant of heat evolutes. This control will be needed later in order to deduce the localized Dorronsoro estimate of Theorem~\ref{thm:local doro}  from the Carleson measure estimate for the heat semigroup of Theorem~\ref{thm:lip doro intro}. Throughout, we are given an $n$-dimensional normed space $(X,\|\cdot\|_X)$ that is also equipped with a Hilbertian norm $|\cdot|$ through which $X$ is identified as a real vector space with $\R^n$. In this setting, for $p\in (0,\infty]$ let
\begin{equation}\label{eq:defMq}
M_p(X)\eqdef \bigg(\fint_{S^{n-1}} \|\sigma\|_X^p \ud \sigma \bigg)^{\frac{1}{p}}.
\end{equation}
So, $M_\infty(X)=\sup_{\sigma\in S^{n-1}}\|\sigma\|_X$, but we shall use below the more common notation $b(X)\eqdef M_\infty(X)$. Also, recalling the notation~\eqref{eq:defIqMq}, we have $M(X)=M_1(X)$.

\begin{lemma}\label{lem:lip of heat} There exists a universal constant $C\in (0,\infty)$ with the following property. Fix  $n\in \N$, an $n$-dimensional normed space $(X,\|\cdot\|_X)$, and $L\in [1,\infty)$. Let $|\cdot|$ be a Hilbertian norm on $X$, thus identifying $X$ with $\R^n$. Suppose that $f:X\to Y$ satisfies  $\|f\|_{\Lip(B_X,Y)}\le 1$ and $\|f\|_{\Lip(X,Y)}\le L$. Then for every $x\in B_X$ we have
\begin{equation}\label{eq:if t is small lip}
0<t\le \frac{(1-\|x\|_X)^2}{C\left(M_1(X)\sqrt{n}+b(X)\sqrt{\log L}\right)^2}\implies \left\|\Taylor_x^1(H_tf)\right\|_{\Lip(X,Y)}\le 2.
\end{equation}
\end{lemma}
\begin{proof} By convolving $f$ with a smooth bump function with arbitrarily small support we may assume without loss of generality that $f$ is smooth, in which case  $\|(z\cdot\nabla)f(w)\|_Y\le \1_{B_X}(w)+L\1_{ \R^n\setminus B_X}(w)$ for every $w\in \R^n$ and $z\in \partial B_X$. Consequently, if we let $G$ denote a standard Gaussian vector in $\R^n$, i.e., the density of $G$ is proportional to $e^{-|x|^2/2}$, then for every $z\in \partial B_X$ we have
\begin{multline*}\label{eq:gaussian version}
\left\|(z\cdot \nabla) H_tf(x)\right\|_Y=\big\|(z\cdot \nabla)\E\big[f\big(x-\sqrt{2t}G\big)\big]\big\|_Y\le \E\big[\big\|(z\cdot \nabla)f\big(x-\sqrt{2t}G\big)\big\|_Y\big]\\\le \Pr\big[\big\|x-\sqrt{2t}G\big\|_X\le 1\big]+L\Pr\big[\big\|x-\sqrt{2t}G\big\|_X\ge 1\big]\le 1+L\Pr\bigg[\|G\|_X\ge \frac{1-\|x\|_X}{\sqrt{2t}}\bigg].
 \end{multline*}
Hence, using Markov's inequality we see that
 \begin{equation}\label{eq:markov inequality}
 \left\|\Taylor_x^1(H_tf)\right\|_{\Lip(X,Y)}=\sup_{z\in \partial B_X} \left\|(z\cdot \nabla) H_tf(x)\right\|_Y\le 1+L\inf_{p\in (0,\infty)} \left(\frac{\sqrt{2t}}{1-\|x\|_X}\right)^p \E\big[\|G\|_X^p\big].
 \end{equation}

 For every $p\in (0,\infty)$, by integrating in polar coordinates we have
\begin{equation}\label{eq:Gaussian moments}
\E\big[\|G\|_X^p\big]=\frac{|B^n|}{(2\pi)^{\frac{n}{2}}}\bigg(\int_0^\infty nr^{n+p-1}e^{-\frac{r^2}{2}}\ud r\bigg)\fint_{S^{n-1}} \|\sigma\|_X^p\ud \sigma =\frac{2^{\frac{p}{2}}\Gamma\left(\frac{n+p}{2}\right)}{\Gamma\left(\frac{n}{2}\right)}M_p(X)^p.
\end{equation}
Also, a theorem of Litvak, Milman and Schechtman~\cite{LMS98} (see also~\cite[Theorem~5.2.4]{BGVV14}) asserts that
\begin{equation}\label{eq:quote LMS}
\forall\, p\in [1,\infty],\qquad M_p(X)\asymp M(X)+\frac{\sqrt{p}}{\sqrt{n+p}}b(X).
\end{equation}
By combining~\eqref{eq:Gaussian moments} with~\eqref{eq:quote LMS}  and Stirling's formula  we therefore see that
\begin{equation}\label{gaussian moments m b}
\forall\, p\in [1,\infty],\qquad \left(\E\big[\|G\|_X^p\big]\right)^{\frac{1}{p}}\asymp \sqrt{n+p}\left(M(X)+\frac{\sqrt{p}}{\sqrt{n+p}}b(X)\right)\asymp M(X)\sqrt{n}+b(X)\sqrt{p}.
\end{equation}

Suppose that  $t$ satisfies the assumption that appears in~\eqref{eq:if t is small lip}, with $C\in (0,\infty)$ being a large enough universal constant that will be specified presently.  A substitution of~\eqref{gaussian moments m b} into~\eqref{eq:markov inequality} shows that there  exists a universal constant $\kappa\in (0,\infty)$ such that
\begin{align}\label{eq:substitute upper bound on t}
\nonumber \left\|\Taylor_x^1(H_tf)\right\|_{\Lip(X,Y)}&\le 1+L \inf_{p\in (0,\infty)} \left(\frac{\kappa\sqrt{2t}\left(M(X)\sqrt{n}+b(X)\sqrt{p}\right)}{1-\|x\|_X}\right)^p\\
&\le 1+L \inf_{p\in (0,\infty)} \left(\frac{\kappa\sqrt{2}}{\sqrt{C}}\cdot\frac{M(X)\sqrt{n}+b(X)\sqrt{p}}{M(X)\sqrt{n}+b(X)\sqrt{\log L}}\right)^p,
\end{align}
where in the last step of~\eqref{eq:substitute upper bound on t} we used the upper bound on $t$ that appears in~\eqref{eq:if t is small lip}. Our choice of $p$ in~\eqref{eq:substitute upper bound on t} is $p=C\left(M(X)\sqrt{n}+b(X)\sqrt{\log L}\right)^2/(8e^2\kappa^2b(X)^2)$. Since $b(X)\lesssim M(X)\sqrt{n}$ (e.g., this follows from the case $p=1$ of~\eqref{eq:quote LMS}), provided $C$ is a large enough universal constant we have $p\ge 1$ and  $p\ge nM(X)^2/b(X)^2$. This implies that $M(X)\sqrt{n}+b(X)\sqrt{p}\le 2b(X)\sqrt{p}$, and therefore,
 \begin{equation*}
 \left(\frac{\kappa\sqrt{2}}{\sqrt{C}}\cdot\frac{M(X)\sqrt{n}+b(X)\sqrt{p}}{M(X)\sqrt{n}+b(X)\sqrt{\log L}}\right)^p\le  \left(\frac{\kappa\sqrt{2}}{\sqrt{C}}\cdot\frac{2b(X)\sqrt{p}}{M(X)\sqrt{n}+b(X)\sqrt{\log L}}\right)^p=\frac{1}{e^p}\le \frac{1}{L^{\frac{C}{8e^2\kappa^2}}}\le \frac{1}{L},
 \end{equation*}
 where in the penultimate step we used the fact that $p\ge (C/8\kappa^2e^2\kappa^2)\log L$ and the final step holds true provided $C\ge 8e^2\kappa^2$. By~\eqref{eq:substitute upper bound on t}, this concludes the proof of the desired implication~\eqref{eq:if t is small lip}.
\end{proof}

\section{Deduction of Theorem~\ref{thm:lip doro intro} and Theorem~\ref{thm:local doro} from Theorem~\ref{thm:LPspatial}}\label{sec:assuming}

Theorem~\ref{thm:LPspatial}, i.e., the Littlwood--Paley--Stein inequalities for the heat semigroup, will be proven in Section~\ref{sec:segroup proof} below. In this section we will assume the validity of Theorem~\ref{thm:LPspatial} for the moment and show how Theorem~\ref{thm:lip doro intro} and Theorem~\ref{thm:local doro} follow from it. This implies our main result on quantitative differentiation, namely Theorem~\ref{thm:main intro}, as we explained in Section~\ref{sec:local}.  The main step is Theorem~\ref{thm:w1p doro} below, which asserts a statement in the spirit of Theorem~\ref{thm:lip doro intro} but with the assumption that $f$ is Lipschitz replaced by the weaker assumption that certain Sobolev $W^{1,q}$ norms of $f$ are finite. Similar refinements already appear in Dorronsoro's original work~\cite{Dorro:85} for scalar-valued functions.

\begin{theorem}\label{thm:w1p doro} Fix $q\in [2,\infty)$ and $n\in \N$. Suppose that $(X,\|\cdot\|_X)$ and $(Y,\|\cdot\|_Y)$ are Banach spaces that satisfy $\dim(X)=n$ and $\mathfrak{m}_q(Y)<\infty$. Let $|\cdot|$ be a Hilbertian norm on $X$, thus identifying $X$ with $\R^n$. Then for every $\gamma\in (0,\infty)$ and every smooth function $f:\R^n\to Y$ we have
\begin{multline}\label{eq:W1p doro}
\bigg(\int_{\R^n}\int_0^\infty \fint_{x+tB_X}\frac{\|f(y)-\Taylor_x^1(H_{\gamma t^2}f)(y)\|_Y^q}{t^{q+1}}\ud y \ud t \ud x \bigg)^{\frac{1}{q}}\\\lesssim
\frac{\mathfrak{m}_q(Y)}{\sqrt{\gamma}}\bigg(\fint_{B_X} |x|^q\left\|x\cdot \nabla f\right\|^q_{L_q(\R^n;Y)}\ud x \bigg)^{\frac{1}{q}}+\mathfrak{m}_q(Y)\sqrt{\gamma n}\fint_{S^{n-1}} \left\|\sigma\cdot \nabla f\right\|_{L_q(\R^n;Y)}\ud \sigma .
\end{multline}
Consequently, by choosing $\gamma\in (0,\infty)$ so as to minimize the right hand side of~\eqref{eq:W1p doro}, if we define
\begin{equation*}\label{eq:def gamma f}
\gamma(f)\eqdef \frac{1}{\sqrt{n}}\cdot\frac{\left(\fint_{B_X} |x|^q\left\|x\cdot \nabla f\right\|^q_{L_q(\R^n;Y)}\ud x\right)^{\frac{1}{q}}}{\fint_{S^{n-1}} \left\|\sigma\cdot \nabla f\right\|_{L_q(\R^n;Y)}\ud \sigma},
\end{equation*}
then
\begin{multline}\label{eq:optimized gamma f}
\bigg(\int_{\R^n}\int_0^\infty \fint_{x+tB_X}\frac{\|f(y)-\Taylor_x^1\big(H_{\gamma(f) t^2}f\big)(y)\|_Y^q}{t^{q+1}}\ud y \ud t \ud x \bigg)^{\frac{1}{q}}\\\lesssim \mathfrak{m}_q(Y)\sqrt[4]{n} \bigg(\fint_{B_X} |x|^q\left\|x\cdot \nabla f\right\|^q_{L_q(\R^n;Y)}\ud x\bigg)^{\frac{1}{2q}} \bigg(\fint_{S^{n-1}} \left\|\sigma\cdot \nabla f\right\|_{L_q(\R^n;Y)}\ud \sigma \bigg)^{\frac{1}{2}}.
\end{multline}
\end{theorem}

\begin{proof}
The validity of~\eqref{eq:optimized gamma f} follows from substituting (the optimal choice) $\gamma=\gamma(f)$ into~\eqref{eq:W1p doro}. So, it remains to prove~\eqref{eq:W1p doro}.  To do so, we shall prove the following two estimates.

\begin{align}\label{eq:I2 bound}
\nonumber J_1&\eqdef\bigg(\int_{\R^n}\int_0^\infty \fint_{B_X}\frac{\|H_{\gamma t^2}f(x+tz)-\Taylor_x^1(H_{\gamma t^2}f)(x+tz)\|_Y^q}{t^{q+1}} \ud z \ud t \ud x\bigg)^{\frac{1}{q}}\\&\lesssim
\frac{\mathfrak{m}_q(Y)}{\sqrt{\gamma}}\bigg(\fint_{B_X} |x|^q\left\|x\cdot \nabla f\right\|^q_{L_q(\R^n;Y)}\ud x\bigg)^{\frac{1}{q}},
\end{align}
and
\begin{equation}\label{eq:I1 bound}
J_2\eqdef \bigg(\int_0^\infty\int_{\R^n} \frac{\|f(y)-H_{\gamma t^2}f(y)\|^q_{Y}}{t^{q+1}}\ud y\ud t\bigg)^{\frac{1}{q}}\lesssim
\mathfrak{m}_q(Y)\sqrt{\gamma n}\fint_{S^{n-1}} \left\|\sigma\cdot \nabla f\right\|_{L_q(\R^n;Y)}\ud \sigma,
\end{equation}

Once proven, the validity of~\eqref{eq:I1 bound} and~\eqref{eq:I2 bound} would imply Theorem~\ref{thm:w1p doro} because, by adding and subtracting $H_{\gamma t^2}f(y)$ and applying the triangle inequality in $L_q(\R^n\times \R\times \R^n;Y)$, we have
\begin{equation*}\label{eq:bound by I1 I2}
\bigg(\int_{\R^n}\int_0^\infty \fint_{x+tB_X}\frac{\|f(y)-\Taylor_x^1(H_{\gamma t^2}f)(y)\|_Y^q}{t^{q+1}}\ud y \ud t \ud x \bigg)^{\frac{1}{q}}\le J_1+J_2.
\end{equation*}

To prove~\eqref{eq:I1 bound}, observe that $\dot H_tf\eqdef\partial_t H_tf=\Delta H_tf=\operatorname{div}H_t\nabla f$, and therefore
\begin{align}\label{eq:here heat}
\nonumber J_2&=\bigg(\int_0^\infty \Big\|\int_0^1 \gamma t^2\dot{H}_{u\gamma t^2}f\ud u\Big\|^q_{L_q(\R^n;Y)}\frac{\ud t}{t^{q+1}}\bigg)^{\frac{1}{q}}\\\nonumber &\le \gamma \int_0^1 \bigg(\int_0^\infty t^{q-1}\left\|\dot{H}_{u\gamma t^2}f\right\|_{L_q(\R^n;Y)}^q\ud t\bigg)^{\frac{1}{q}}\ud u\\ \nonumber
&=  \sqrt{\gamma} \bigg(\int_0^1 \frac{\ud u}{\sqrt{u}}\bigg)\bigg(\frac12\int_0^\infty \left\|\sqrt{s}\dot{H}_{s}f\right\|_{L_q(\R^n;Y)}^q\frac{\ud s}{s}\bigg)^{\frac{1}{q}} \\&=2^{1-\frac{1}{q}}\sqrt{\gamma}
\bigg(\int_0^\infty \left\|\sqrt{s}\operatorname{div}H_t\nabla f\right\|_{L_q(\R^n;Y)}^q\frac{\ud s}{s}\bigg)^{\frac{1}{q}}.
\end{align}
The desired estimate~\eqref{eq:I1 bound} on $J_2$ now follows from an application of~\eqref{eq:first spatial} with $\vec{f}=\nabla f$.

To prove~\eqref{eq:I2 bound}, observe first that by the integral representation for the error in Taylor's formula, for every $x,z\in \R^n$ and $t\in \R$ we have
$$
H_{\gamma t^2}f(x+tz)-\Taylor_x^1(H_{\gamma t^2}f)(x+tz)=\int_0^1 (tz\cdot\nabla)^2H_{\gamma t^2}f(x+stz)(1-s)\ud s.
$$
Consequently, using Jensen's inequality we see that
\begin{align}\label{eq:I2 taylor}
\nonumber J_1 &\le \int_0^1\bigg(\int_{\R^n}\int_0^\infty \fint_{B_X}t^{q-1}\left\|(z\cdot\nabla)^2H_{\gamma t^2}f(x+stz)\right\|_Y^q\ud z \ud t \ud x \bigg)^{\frac{1}{q}}(1-s)\ud s\\
&= \frac12 \bigg(\int_0^\infty \fint_{B_X}t^{q-1}\left\|(z\cdot\nabla)^2H_{\gamma t^2}f\right\|_{L_q(\R^n;Y)}^q\ud z \ud t \bigg)^{\frac{1}{q}}.
\end{align}
Since the operators $\{z\cdot\nabla\}_{z\in \R^n}$ and $\{H_s\}_{s\in [0,\infty)}$ commute,  for every $z\in B_X$ we have
\begin{align}\label{use lower john}
\bigg(\int_0^\infty t^{q-1}\left\|(z\cdot\nabla)^2H_{\gamma t^2}f\right\|_{L_q(\R^n;Y)}^q\ud t \bigg)^{\frac{1}{q}}&=\frac{1}{\sqrt{\gamma}}\bigg(\int_0^\infty \left\|\sqrt{s}(z\cdot\nabla)H_{s}(z\cdot\nabla)f\right\|_{L_q(\R^n;Y)}^q\frac{\ud s}{s} \bigg)^{\frac{1}{q}}\nonumber\\
&\lesssim \frac{|z|\mathfrak{m}_q(Y)}{\sqrt{\gamma}} \left\|z\cdot \nabla f\right\|_{L_q(\R^n;Y)},
\end{align}
where in the final step of~\eqref{use lower john} we used~\eqref{eq:second spatial} with $f$ replaced by $(z\cdot\nabla) f$. The desired estimate~\eqref{eq:I2 bound} now follows from a substitution of~\eqref{use lower john} into~\eqref{eq:I2 taylor}.
\end{proof}

\begin{proof}[Proof of Theorem~\ref{thm:lip doro intro}] By rescaling we may assume that $\|f\|_{\Lip(X,Y)}=1$, and by convolving $f$ with a smooth bump function of arbitrarily small support we may assume that $f$ is itself smooth. Then $\|z\cdot \nabla f(x)\|_Y\le \|z\|_X\1_{\supp(f)}(x)$ for every $x,z\in \R^n$. So, $\|z\cdot \nabla f\|_{L_q(\R^n;Y)}^q\le \|z\|_X^q|\supp(f)|$.  Hence,
$$
\fint_{S^{n-1}} \left\|\sigma\cdot \nabla f\right\|_{L_q(\R^n;Y)}\ud \sigma\le
|\supp(f)|^{\frac{1}{q}}\fint_{S^{n-1}} \|\sigma\|_X\ud \sigma= M(X)|\supp(f)|^{\frac{1}{q}},
$$
and
$$
\bigg(\fint_{B_X} |x|^q\left\|x\cdot \nabla f\right\|^q_{L_q(\R^n;Y)}\ud x \bigg)^{\frac{1}{q}}\le \bigg(\fint_{B_X} |x|^q\ud x \bigg)^{\frac{1}{q}}|\supp(f)|^{\frac{1}{q}}=I_q(X)|\supp(f)|^{\frac{1}{q}}.
$$
A substitution of these estimates into~\eqref{eq:W1p doro} shows that for every $\gamma\in (0,\infty)$ we have
\begin{multline}\label{eq:to minimize gamma MI}
\bigg(\int_{\R^n}\int_0^\infty \fint_{x+tB_X}\frac{\|f(y)-\Taylor_x^1(H_{\gamma t^2}f)(y)\|_Y^q}{t^{q+1}}\ud y \ud t \ud x \bigg)^{\frac{1}{q}}\\\lesssim \mathfrak{m}_q(Y)\left(\sqrt{\gamma n}M(X)+\frac{I_q(X)}{\sqrt{\gamma}}\right)|\supp(f)|^{\frac{1}{q}}.
\end{multline}
The value of $\gamma$ in~\eqref{eq:def gamma} minimizes the right hand side of~\eqref{eq:to minimize gamma MI}, thus yielding the desired estimate~\eqref{eq:our new doro}.
\end{proof}

\begin{proof}[Proof of Theorem~\ref{thm:local doro}] Suppose that $f:B_X\to Y$ satisfies $\|f\|_{\Lip(B_X,Y)}\le 1$. Since Theorem~\ref{thm:local doro} is translation-invariant, we may assume without loss of generality that $f(0)=0$. Define $F:X\to Y$ by setting $F(x)=f(x)$ for $x\in B_X$ and $F(x)=\max\{0,n+1-n\|x\|_X\}\cdot f(x/\|x\|_X)$ for $x\in \R^n\setminus B_X$. Then $\supp(F)\subset (1+1/n)B_X$ and it is straightforward to check that $\|F\|_{\Lip(X,Y)}\le n+2$.

Fix $x\in \R^n$ with$\|x\|_X\le 1-1/(2n)$ and $t\in (0,\infty)$. Recalling the definition of $\gamma$ in~\eqref{eq:def gamma} and letting $C\in (0,\infty)$ be the constant of Lemma~\ref{lem:lip of heat}, we know that $\left\|\Taylor_x^1(H_{\gamma t^2F})\right\|_{\Lip(X,Y)}\le 2$ provided
\begin{equation}\label{eq:gamma t condition}
\gamma t^2=\frac{I_q(X)t^2}{\sqrt{n}M(X)}\le \frac{1}{4Cn^2\left(M(X)\sqrt{n}+b(X)\sqrt{\log (n+2)}\right)^2}.
\end{equation}
Since by~\eqref{eq:quote LMS} (with $p=1$) we have $b(X)\lesssim M(X)\sqrt{n}$, there exists a universal constant $c\in (0,1/4)$ such that the condition~\eqref{eq:gamma t condition} is satisfied for every $t\in (0,T]$, where $T$ is defined in~\eqref{eq:def T}. Hence, for  $t\in (0,T]$ the mapping $\Taylor_x^1(H_{\gamma t^2}F)\in \mathscr{A}(X,Y)$ is $2$-Lipschitz. Consequently,
\begin{multline*}
\fint_r^{T} \bigg(\fint_{\left(1-\frac{1}{2n}\right)B_X}\inf_{\substack{\Lambda\in \mathscr{A}(X,Y)\\\|\Lambda\|_{\Lip(X,Y)}\le 2}}\frac{\fint_{x+\rho B_X}\|f(y)-\Lambda(y)\|_Y^q}{\rho^{q}}\ud y\ud x\bigg) \frac{\ud \rho}{\rho} \\\le  \frac{1}{\left(1-\frac{1}{2n}\right)^n|B_X|\log\left(\frac{T}{r}\right)} \int_{\left(1-\frac{1}{2n}\right)B_X}\int_r^T \fint_{x+tB_X}\frac{\|F(y)-\Taylor_x^1(H_{\gamma t^2}F)(y)\|_Y^q}{t^{q+1}}\ud y \ud t \ud x.
\end{multline*}
Hence, by Theorem~\ref{thm:lip doro intro} (which we have already proved assuming the validity of Theorem~\ref{thm:LPspatial}) applied to $F$,  the left hand side of the desired inequality~\eqref{eq:desired local} is bounded from above by
$$
\frac{K^q |\supp(F)|(n+2)^q}{\left(1-\frac{1}{2n}\right)^n|B_X|\log\left(\frac{T}{r}\right)}\le \frac{2(3Kn)^q \left(1+\frac{1}{n}\right)^n}{\left(1-\frac{1}{2n}\right)^n|\log r|}\le \frac{(9Kn)^q}{|\log r|},
$$
where we used the fact that $\|F\|_{\Lip(X,Y)}\le n+2$, that the support of $F$ is contained in $(1+1/n)B_X$ and therefore $|\supp(F)|\le (1+1/n)^n|B_X|$, that $r\le T^2\le 1$ and therefore $\log(T/r)\ge |\log r|/2$, that the sequence $\{(1+1/n)^n/(1-1/(2n))^n=((1+3/(2n-1))^{(2n-1)/3+1})^{3n/(2n+2)}\}_{n=1}^\infty$ is decreasing and therefore bounded from above by $4$, and that $q\ge 2$. This completes the proof of Theorem~\ref{thm:local doro}.
\end{proof}

\section{$\mathcal{G}$-function estimates}\label{sec:segroup proof}

Here we shall prove Theorem~\ref{thm:LPtemporal} and Theorem~\ref{thm:LPspatial}. The argument naturally splits into a part that holds true for general symmetric diffusion semigroups in Section~\ref{sec:diffusion step} below, followed by steps that use more special properties of the heat semigroups in Section~\ref{sec:spatial heat} and Section~\ref{sec:temporal heat} below.

\subsection{Diffusion semigroups}\label{sec:diffusion step} Following Stein~\cite[page~65]{Stein:topics}, a symmetric diffusion semigroup on a measure space $(\mathscr{M},\mu)$ is a one-parameter family of self-adjoint linear operators $\{T_t\}_{t\in [0,\infty)}$ that map (real-valued) measurable functions on $(\mathscr{M},\mu)$ to measurable functions on $(\mathscr{M},\mu)$, such that $T_0$ is the identity operator and $T_{t+s}=T_tT_s$ for every $s,t \in [0,\infty)$. Moreover, it is required that for every $t\in [0,\infty)$ and $p\in [1,\infty]$ the operator $T_t$ maps $L_p(\mu)$ to $L_p(\mu)$  with $\|T_t\|_{L_p(\mu)\to L_p(\mu)}\le 1$, for every $f\in L_2(\mu)$ we have $\lim_{t\to 0} T_tf=f$ (with convergence in $L_2(\mu)$), for every nonnegative measurable function $f:\mathscr{M}\to \R$ the function $T_tf$ is also nonnegative, and that $T_t\1_{\mathscr{M}}=\1_{\mathscr{M}}$.

As explained in~\cite[page~433]{MTX}, for every Banach space $Y$ the above semigroup $\{T_t\}_{t\in [0,\infty)}$ extends to a semigroup of contractions on $L_q(\mu;Y)$ for every $q\in [1,\infty]$. This is achieved by considering the tensor product  $T_t\otimes \mathrm{Id}_Y$, but in what follows it will be convenient to slightly abuse  notation by identifying  $T_t\otimes \mathrm{Id}_Y$ with $T_t$. Note that by a standard density argument for every $q\in [1,\infty)$ and $f\in L_q(\mu,Y)$ the mapping $t\mapsto T_tf$ is continuous as a mapping from $[0,\infty)$ to $L_q(\mu;Y)$.

\begin{proposition}\label{prop:Rota} Fix $q\in [2,\infty)$ and a Banach space $(Y,\|\cdot\|_Y)$ of martingale cotype $q$.  Suppose that $\{T_t\}_{t\in [0,\infty)}$ is a symmetric diffusion semigroup on a measure space $(\mathscr{M},\mu)$. Then for every $f\in L_q(\mu;Y)$, if $\{t_j\}_{j\in\Z}\subset (0,\infty)$ is an increasing sequence then
\begin{equation}\label{eq:semigroup version UC}
  \bigg(\sum_{j\in\Z}\Norm{T_{t_j}f-T_{t_{j+1}}f}{L_q(\mu;Y)}^q\bigg)^{\frac{1}{q}}\leq\mathfrak{m}_{q}(Y)\Norm{f}{L_q(\mu;Y)}.
\end{equation}
\end{proposition}

\begin{proof}
It suffices to prove~\eqref{eq:semigroup version UC} for finite sums, i.e., that for every $0<t_0<t_1<\ldots<t_N$ we have
\begin{equation}\label{eq:finite semigroup version UC}
   \bigg(\sum_{j=0}^N\Norm{T_{t_j}f-T_{t_{j+1}}f}{L_q(\mu;Y)}^q\bigg)^{\frac{1}{q}}\leq\mathfrak{m}_{q}(Y)\Norm{f}{L_q(\mu;Y)}.
\end{equation}
Since $t\mapsto T_t f$ is a continuous mapping from $[0,\infty)$ to $L_q(\mu;Y)$, we may further assume by approximation that each $t_j$ is an integer multiple of some $\delta\in (0,\infty)$, i.e., that  $t_j= k_j\delta$ with $k_j\in \N$.

Denoting $Q\eqdef T_{\delta/2}$, the desired bound~\eqref{eq:finite semigroup version UC} can be rewritten as follows.
\begin{equation}\label{eq:Q version}
  \bigg(\sum_{j=0}^N\left\|Q^{2k_j}f-Q^{2k_{j+1}}f\right\|_{L_q(\mu;Y)}^q\bigg)^{\frac{1}{q}}\leq\mathfrak{m}_{q}(Y)\Norm{f}{L_q(\mu;Y)}.
\end{equation}
The operator $Q$ satisfies the assumptions of Rota's representation theorem~\cite{Rota} in the form presented by Stein~\cite[page~106]{Stein:topics} (see~\cite[Theorem 2.5]{MTX} for an explanation of the vector-valued extension that is relevant to the present context), and hence its even powers admit the following representation.
\begin{equation}\label{eq:Q representation}
\forall\, k\in \N,\qquad   Q^{2k}=J^{-1}\circ E'\circ E_k\circ J,
\end{equation}
where
\begin{itemize}
  \item  $J:L_q(\mu;Y)\to L_q(\mathscr{S},\mathscr{F}',\nu;Y)\subset L_q(\mathscr{S},\mathscr{F},\nu;Y)$ is an isometric isomorphism for some $\sigma$-finite $\sigma$-algebras $\mathscr{F}'\subset\mathscr F$ of a measure space $(\mathscr{S},\mathscr{F},\nu)$,
  \item $E_k:L_q(\mathscr{S},\mathscr{F},\nu;Y)\to L_q(\mathscr{S},\mathscr{F}_k,\nu;Y)\subset L_p(\mathscr{S},\mathscr{F},\nu;Y)$ is the ``conditional expectation'' (naturally extended from a probabilistic to a $\sigma$-finite setting), where $\mathscr{F}\supset\mathscr{F}_1\supset\mathscr{F}_2\supset\cdots$ is a decreasing sequence of $\sigma$-finite sub-$\sigma$-algebras, and
  \item $E':L_q(\mathscr{S},\mathscr{F},\nu;Y)\to L_q(\mathscr{S},\mathscr{F}',\nu;Y)\subset L_q(\mathscr{S},\mathscr{F},\nu;Y)$ is another such ``conditional expectation'' for the sub-$\sigma$-algebra $\mathscr{F}'\subset\mathscr F$.
\end{itemize}
Consequently, the desired estimate~\eqref{eq:Q version} is proven as follows.
\begin{multline}\label{eq:use rota punchline}
  \sum_{j=0}^{N}\left\|(Q^{2k_j}-Q^{2k_{j+1}})f\right\|_{L_q(\mu;Y)}^q
  =\sum_{j=0}^{N}\left\|J^{-1}E'(E_{k_{j}}-E_{k_{j+1}}) Jf\right\|_{L_q(\mu;Y)}^q \\
  \leq \sum_{j=0}^{N}\left\|(E_{k_{j}}-E_{k_{j+1}}) Jf\right\|_{L_q(\mathscr{S},\mathscr{F},\nu;Y)}^q
  \leq\mathfrak{m}_{q}(Y)^q\Norm{Jf}{L_q(\mathscr{S},\mathscr{F}',\nu;Y)}^q=\mathfrak{m}_{q}(Y)^q\Norm{f}{L_q(\mu;Y)}^q,
\end{multline}
where the first step of~\eqref{eq:use rota punchline} uses~\eqref{eq:Q representation}, the second step of~\eqref{eq:use rota punchline} uses the fact that $J^{-1}$ is an isometry and that $E'$ is a contraction, the third step of~\eqref{eq:use rota punchline} uses the definition of $\mathfrak{m}_q(Y)$ applied to the (reverse) martingale $\{E_{k_j}Jf)\}_{j=0}^N$, and the final step of~\eqref{eq:use rota punchline} uses the fact that $J$ is an isometry.
\end{proof}

\begin{remark}
The above argument used the definition of martingale cotype $q$ when the martingales are with respect to $\sigma$-finite measures rather than probability measures. While the traditional way to define martingale cotype $q$ uses probability measures, as we have done in Section~\ref{sec:Y geom}, this is equivalent (with the same constant)  to  the case of $\sigma$-finite measures by a general approximation result~\cite[Theorem~3.95]{HNVW}. Alternatively, one can check that the available proofs of Pisier's inequality (Theorem~\ref{thm:pisier inequality}) extend effortlessly to the setting of martingales with respect to $\sigma$-finite measures.
\end{remark}

\begin{lemma}\label{lem:Rota} Fix $q\in [2,\infty)$, $\alpha\in (1,\infty)$ and a Banach space $(Y,\|\cdot\|_Y)$ of martingale cotype $q$. Suppose that $\{T_t\}_{t\in [0,\infty)}$ is a symmetric diffusion semigroup on a measure space $(\mathscr{M},\mu)$. Then
\begin{equation*}
\forall\, f\in L_q(\mu;Y),\qquad  \bigg(\int_0^\infty\Norm{(T_t-T_{\alpha t})f}{L_q(\R^n;Y)}^q\frac{\ud t}{t}\bigg)^{\frac{1}{q}}\leq(\log\alpha)^{\frac{1}{q}}\mathfrak{m}_{q}(Y)\Norm{f}{L_q(\mu;Y)}.
\end{equation*}
\end{lemma}

\begin{proof} The desired estimate is proven by the following computation.
\begin{multline}\label{eq:sum to integral}
  \int_0^\infty\Norm{(T_t-T_{\alpha t})f}{L_q(\mu;Y)}^q\frac{\ud t}{t}
  =\sum_{j\in\Z}\int_{\alpha^j}^{\alpha^{j+1}}\Norm{(T_t-T_{\alpha t})f}{L_q(\mu;Y)}^q\frac{\ud t}{t} \\
  =\int_{1}^{\alpha}\sum_{j\in\Z}\left\|(T_{\alpha^j t}-T_{\alpha^{j+1} t})f\right\|_{L_q(\mu;Y)}^q\frac{\ud t}{t}
  \leq \mathfrak{m}_{q}(Y)^q\Norm{f}{L_q(\mu;Y)}^q\int_{1}^{\alpha}\frac{\ud t}{t},
\end{multline}
where the last step of~\eqref{eq:sum to integral} is an application of Proposition~\ref{prop:Rota} with $t_j=\alpha^j t$.
\end{proof}

\subsection{The spatial derivatives of the heat semigroup}\label{sec:spatial heat} Here we shall prove Theorem~\ref{thm:LPtemporal}.

\begin{lemma}\label{lem:spatial vec} Fix $q\in [1,\infty]$ and a Banach space $(Y,\|\cdot\|_Y)$. For every $\vec{f}\in \ell_q^n(L_q(\R^n;Y))$ we have
\begin{equation}\label{eq:use stirling}
\left\|\sqrt{t}\operatorname{div} H_t\vec{f}\,\right\|_{L_q(\R^n;Y)}\le \frac{\Gamma\left(\frac{n+1}{2}\right)}{\Gamma\left(\frac{n}{2}\right)}\fint_{S^{n-1}}\left\|\sigma\cdot\vec{f}\,\right\|_{L_q(\R^n;Y)}\ud\sigma\lesssim\sqrt{n}\fint_{S^{n-1}}
\left\|\sigma\cdot\vec{f}\,\right\|_{L_q(\R^n;Y)}\ud\sigma.
\end{equation}
\end{lemma}

\begin{proof}  Observe that $\partial_j h_t(y)=-y_jh_t(y)/(2t)$ for every $j\in \n$, $y=(y_1,\ldots,y_n)\in \R^n$ and $t\in [0,\infty)$. Hence, for every $x\in \R^n$ and $t\in [0,\infty)$ we have
\begin{equation}\label{eq:divergence heat}
\div H_t \vec{f}(x)=\sum_{j=1}^n \partial_j h_t*f_j(x)=-\frac{1}{2t}\int_{\R^n} h_t(y) y\cdot \vec{f}(x-y)\ud y=-\frac{1}{2t}\int_{\R^n} h_t(y) y\cdot \vec{f}_y(x)\ud y,
\end{equation}
where $\vec{f}_y:\R^n\to \ell_q^n(L_q(\R^n;Y))$ is defined by setting $\vec{f}_y(x)=\vec{f}(x-y)$ for every $x,y\in \R^n$.

By translation invariance $\left\|y\cdot \vec{f}_y\right\|_{L_q(\R^n;Y)}=\left\|y\cdot \vec{f}\,\right\|_{L_q(\R^n;Y)}$ for all $y\in \R^n$, so~\eqref{eq:divergence heat} implies that
\begin{multline*}\label{eq:gamma ratio}
\left\|\sqrt{t}\div H_t\vec{f}\,\right\|_{L_q(\R^n;Y)}\le \frac{1}{2\sqrt{t}}\int_{\R^n} h_t(y) \left\| y\cdot \vec{f}\,\right\|_{L_q(\R^n;Y)}\ud y=\frac12 \int_{\R^n} h_1(z) \left\| z\cdot \vec{f}\,\right\|_{L_q(\R^n;Y)}\ud z\\
=\bigg(\frac{\pi^{\frac{n}{2}}}{\Gamma\left(\frac{n}{2}\right)}\int_{0}^\infty r^n\frac{e^{-\frac{r^2}{4}}}{(4\pi)^{\frac{n}{2}}}\ud r\bigg) \fint_{S^{n-1}}\left\|\sigma\cdot\vec{f}\,\right\|_{L_q(\R^n;Y)}\ud\sigma
=\frac{\Gamma\left(\frac{n+1}{2}\right)}{\Gamma\left(\frac{n}{2}\right)}\fint_{S^{n-1}}\left\|\sigma\cdot\vec{f}\,\right\|_{L_q(\R^n;Y)}\ud\sigma,
\end{multline*}
where the penultimate step follows from passing to polar coordinates. This completes the proof of the first inequality in~\eqref{eq:use stirling}. The second inequality in~\eqref{eq:use stirling} follows from Stirling's formula.
\end{proof}

\begin{lemma}\label{lem:directional heat} Suppose that $q\in [1,\infty]$. For every $f\in L_q(\R^n;Y)$, $t\in [0,\infty)$ and $z\in\R^n$ we have
\begin{equation*}
   \left\|\sqrt{t}(z\cdot\nabla)H_tf\right\|_{L_q(\R^n;Y)}\le\frac{\abs{z}}{\sqrt{\pi}} \Norm{f}{L_q(\R^n;Y)}.
\end{equation*}
\end{lemma}

\begin{proof} Since for every $y\in \R^n$ we have $(z\cdot\nabla)h_t(y)=-(z\cdot y) h_t(y)/(2t)$, every $x\in \R^n$ satisfies
\begin{equation*}
  (z\cdot\nabla)H_t f(x)=(z\cdot\nabla)h_t*f=-\frac{1}{2t}\int_{\R^n}(z\cdot y)h_t(y)f(x-y)\ud y=-\frac{1}{2t}\int_{\R^n}(z\cdot y)h_t(y)f_y(x)\ud y.
\end{equation*}
Consequently,
\begin{multline*}
  \left\|\sqrt{t}(z\cdot\nabla)H_t f\right\|_{L_q(\R^n;Y)}
  \leq\frac{1}{2\sqrt{t}}\int_{\R^n}\abs{z\cdot y}h_t(y)\Norm{f_y}{L_q(\R^n;Y)}\ud y \\
  =\frac{1}{2}\Norm{f}{L_q(\R^n;Y)} \int_{\R^n}\abs{z\cdot y}k_1(y)\ud y
  =\frac{\abs{z}}{2}\Norm{f}{L_q(\R^n;Y)}\int_{\R^n}\abs{y_1}k_1(y)\ud y=\frac{\abs{z}}{\sqrt{\pi}} \Norm{f}{L_q(\R^n;Y)}.\tag*{\qedhere}
\end{multline*}
\end{proof}

\begin{proof}[Proof of Theorem~\ref{thm:LPspatial}]
Since $q\in (1,\infty)$, we have $\lim_{t\to \infty} \left\|H_t\vec{f}\,\right\|_{L_q(\R^n;Y)^n}=0$. Consequently,
\begin{equation}\label{eq:vector telscope}
  H_t\vec f=\sum_{k=-1}^\infty \big(H_{2^{k+1} t}-H_{2^{k+2}t}\big)\vec f
  =\sum_{k=-1}^\infty H_{2^k t}\big(H_{2^k t}-H_{3\cdot 2^{k}t}\big)\vec f.
\end{equation}
By the triangle inequality in $L_q((0,\infty),\ud t/t;L_q(\R^n;Y))$, it follows from~\eqref{eq:vector telscope} that
\begin{align}
  \nonumber \bigg(\int_0^\infty\left\|\sqrt{t}\operatorname{div}H_t\vec{f}\,\right\|_{L_q(\R^n;Y)}^q\frac{\ud t}{t}\bigg)^{\frac{1}{q}}
   &\leq\sum_{k=-1}^\infty \bigg(\int_0^\infty\left\|\sqrt{t}\operatorname{div}H_{2^k t}\big(H_{2^k t}-H_{3\cdot 2^{k}t}\big)\vec{f}\,\right\|_{L_q(\R^n;Y)}^q\frac{\ud t}{t}\bigg)^{\frac{1}{q}} \\ \label{eq:change varaible to s}
   &=\sum_{k=-1}^\infty \frac{1}{2^{\frac{k}{2}}}\bigg(\int_0^\infty\left\|\sqrt{s}\operatorname{div}H_{s}(H_{s}-H_{3s})\vec{f}\,\right\|_{L_q(\R^n;Y)}^q\frac{\ud s}{s}\bigg)^{\frac{1}{q}}\\&\asymp \bigg(\int_0^\infty\left\|\sqrt{s}\operatorname{div}H_{s}(H_{s}-H_{3s})\vec{f}\,\right\|_{L_q(\R^n;Y)}^q\frac{\ud s}{s}\bigg)^{\frac{1}{q}},\label{eq:use telescope}
\end{align}
where in~\eqref{eq:change varaible to s} we made the change of variable $s=2^kt$ in each of the summands.

For every $s\in [0,\infty)$, an application of Lemma~\ref{lem:spatial vec} with $\vec{f}$ replaced by $(H_s-H_{3s})\vec{f}$ shows that
\begin{align}\label{lem:at time s}
\nonumber \left\|\sqrt{s}\operatorname{div}H_{s}(H_{s}-H_{3s})\vec{f}\,\right\|_{L_q(\R^n;Y)}&\lesssim \sqrt{n} \fint_{S^{n-1}}
\left\|\sigma\cdot (H_s-H_{3s})\vec{f}\, \right\|_{L_q(\R^n;Y)}\ud\sigma\\
&= \sqrt{n} \fint_{S^{n-1}}
\left\|(H_s-H_{3s})\left(\sigma\cdot \vec{f}\,\right)\right\|_{L_q(\R^n;Y)}\ud\sigma.
\end{align}
By combining~\eqref{eq:use telescope} and~\eqref{lem:at time s} we therefore conclude that
\begin{align}\label{eq:to use alpha lemma}
\nonumber\bigg(\int_0^\infty\left\|\sqrt{t}\operatorname{div}H_t\vec{f}\,\right\|_{L_q(\R^n;Y)}^q\frac{\ud t}{t}\bigg)^{\frac{1}{q}} &\lesssim \sqrt{n}\bigg(\int_0^\infty
\bigg(\fint_{S^{n-1}}\left\|(H_s-H_{3s})\left(\sigma\cdot \vec{f}\,\right)\right\|_{L_q(\R^n;Y)}\ud\sigma\bigg)^q\frac{\ud s}{s}\bigg)^{\frac{1}{q}}
\\\nonumber&\le \sqrt{n}\fint_{S^{n-1}}\bigg(\int_0^\infty\left\|(H_s-H_{3s})\left(\sigma\cdot \vec{f}\,\right)\right\|_{L_q(\R^n;Y)}^q\frac{\ud s}{s}\bigg)^{\frac{1}{q}}\ud\sigma
\\&\lesssim \sqrt{n}\cdot \mathfrak{m}_q(Y)\fint_{S^{n-1}}
\left\|\sigma\cdot \vec{f}\,\right\|_{L_q(\R^n;Y)}\ud\sigma,
\end{align}
where the penultimate step of~\eqref{eq:to use alpha lemma} uses the triangle inequality in  $L_q((0,\infty),\ud s/s)$ and the final step of~\eqref{eq:to use alpha lemma} uses Lemma~\ref{lem:Rota}. This completes the proof of~\eqref{eq:first spatial}.

To prove~\eqref{eq:second spatial}, by the triangle inequality applied to the identity~\eqref{eq:vector telscope} with $\vec{f}$ replaced by $f$, and making the same changes of variable as in~\eqref{eq:change varaible to s}, we see that
\begin{multline}\label{eq:final proof second spatial}
   \bigg(\int_0^\infty\left\|\sqrt{t}(z\cdot\nabla)H_t f\right\|_{L_q(\R^n;Y)}^q\frac{\ud t}{t}\bigg)^{\frac{1}{q}}
   \lesssim\bigg(\int_0^\infty\left\|\sqrt{s}(z\cdot\nabla)H_{s}(H_{s}-H_{3s})f\right\|_{L_q(\R^n;Y)}^q\frac{\ud s}{s}\bigg)^{\frac{1}{q}} \\
   \lesssim \abs{z}\bigg(\int_0^\infty  \left\|(H_{s}-H_{3s})f\right\|_{L_q(\R^n;Y)}^q \frac{\ud s}{s}\bigg)^{\frac{1}{q}}
   \lesssim\abs{z}\mathfrak{m}_{q}(Y)  \Norm{f}{L_q(\R^n;Y)},
\end{multline}
where the penultimate step of~\eqref{eq:final proof second spatial} uses Lemma~\ref{lem:directional heat} and the final step of~\eqref{eq:final proof second spatial} uses Lemma~\ref{lem:Rota}.
\end{proof}

\subsection{The time derivative of the heat semigroup}\label{sec:temporal heat} Here we shall prove Theorem~\ref{thm:LPspatial}.

\begin{lemma}\label{lem:time t convolution}
Suppose that $q\in [1,\infty]$. Then for every $f\in L_q(\R^n;Y)$ and $t\in [1,\infty)$ we have
$$
\left\|t\dot{H}_tf\right\|_{L_q(\R^n;Y)}\lesssim \sqrt{n}\cdot \left\|f\right\|_{L_q(\R^n;Y)}.
$$
\end{lemma}

\begin{proof}
We claim that for every $t\in (0,\infty)$ we have
\begin{equation}\label{eq:time der heat kernel}
\left\|t\dot{h}_t\right\|_{L_1(\R^n)}=\frac{2}{\Gamma\left(\frac{n}{2}\right)}\left(\frac{n}{2e}\right)^{\frac{n}{2}}\asymp\sqrt{n}.
\end{equation}
This would imply the desired estimate because
$$
\left\|t\dot{H}_tf\right\|_{L_q(\R^n;Y)}=\left\|(t\dot{h}_t)*f\right\|_{L_q(\R^n;Y)}\le \left\|t\dot{h}_t\right\|_{L_1(\R^n)}\cdot \left\|f\right\|_{L_q(\R^n;Y)}\asymp \sqrt{n}\cdot \left\|f\right\|_{L_q(\R^n;Y)}.
$$

Verifying  the validity of the identity~\eqref{eq:time der heat kernel} amounts to the following simple computation. By direct differentiation we see that every $x\in \R^n$ and $t\in (0,\infty)$ satisfy
$$
\dot{h}_t(x)=-\frac{1}{2t}\left(n-\frac{|x|^2}{2t}\right)h_t(x).
$$
Hence, by passing to polar coordinates we have
\begin{equation}\label{eq:abs int}
\forall\, t\in (0,\infty),\qquad \left\|t\dot{h}_t\right\|_{L_1(\R^n)} =\frac{\pi^{\frac{n}{2}}}{\Gamma\left(\frac{n}{2}\right)}\int_0^\infty \left|n-\frac{r^2}{2t}\right|h_t(r)r^{n-1}\ud r.
\end{equation}
It therefore remains to evaluate the integral in the right hand side of~\eqref{eq:abs int} as follows.
\begin{multline*}
\int_0^\infty \left|n-\frac{|x|^2}{2t}\right|h_t(r)r^{n-1}\ud r=\int_0^{\sqrt{2tn}}\frac{n-\frac{r^2}{2t}}{(4\pi t)^{\frac{n}{2}}}e^{-\frac{r^2}{4t}}r^{n-1}\ud r+\int_{\sqrt{2tn}}^\infty\frac{\frac{r^2}{2t}-n}{(4\pi t)^{\frac{n}{2}}}e^{-\frac{r^2}{4t}}r^{n-1}\ud r\\
=\frac{1}{(4\pi t)^{\frac{n}{2}}}\int_0^{\sqrt{2tn}} \frac{\partial}{\partial r}\left(r^ne^{-\frac{r^2}{4t}}\right)\ud r-\frac{1}{(4\pi t)^{\frac{n}{2}}}\int_{\sqrt{2tn}}^\infty \frac{\partial}{\partial r}\left(r^ne^{-\frac{r^2}{4t}}\right)\ud r=2\left(\frac{n}{2\pi e}\right)^{\frac{n}{2}}.\tag*{\qedhere}
\end{multline*}
\end{proof}

\begin{proof}[Proof of Theorem \ref{thm:LPtemporal}] The semigroup identity $H_{t+s}=H_tH_s$  implies that $\dot{H}_{t+s}=\dot{H}_tH_s$ for every $s,t\in (0,\infty)$. Hence, recalling~\eqref{eq:vector telscope}, we have
\begin{equation*}
 \forall\, t\in (0,\infty),\qquad  \dot H_t f=\sum_{k=-1}^\infty \left(\dot H_{2^{k+1} t}-\dot H_{2^{k+2}t}\right)f
  =\sum_{k=-1}^\infty \dot H_{2^k t}\left(H_{2^k t}-H_{3\cdot 2^{k}t}\right)f.
\end{equation*}
Consequently, by arguing as in~\eqref{eq:use telescope}, we conclude that
\begin{multline*}
   \bigg(\int_0^\infty\left\|t\dot H_t f\right\|_{L_q(\R^n;Y)}^q\frac{\ud t}{t}\bigg)^{\frac{1}{q}}
   \leq\sum_{k=-1}^\infty \bigg(\int_0^\infty\left\|t\dot H_{2^k t}(H_{2^k t}-H_{3\cdot 2^{k}t}) f \right\|_{L_q(\R^n;Y)}^q\frac{\ud t}{t}\Big)^{\frac{1}{q}} \\
   =\sum_{k=-1}^\infty \frac{1}{2^{k}}\bigg(\int_0^\infty\left\|s \dot H_{s}(H_{s}-H_{3s}) f\right\|_{L_q(\R^n;Y)}^q\frac{\ud s}{s}\bigg)^{\frac{1}{q}}\lesssim \bigg(\int_0^\infty\left\|t \dot H_{t}(H_{t}-H_{3t}) f\right\|_{L_q(\R^n;Y)}^q\frac{\ud t}{t}\bigg)^{\frac{1}{q}}.
\end{multline*}
It remains to note that, by an application of Lemma~\ref{lem:time t convolution} (with $f$ replaced by $(H_{t}-H_{3t}) f$) followed by integration with respect to $t$ and an application of Lemma~\ref{lem:Rota}, we have
\begin{align*}
\bigg(\int_0^\infty\left\|t \dot H_{t}(H_{t}-H_{3t}) f\right\|_{L_q(\R^n;Y)}^q\frac{\ud t}{t}\bigg)^{\frac{1}{q}}&\lesssim \sqrt{n}
\bigg(\int_0^\infty\left\|(H_{t}-H_{3t}) f\right\|_{L_q(\R^n;Y)}^q\frac{\ud t}{t}\bigg)^{\frac{1}{q}}\\&\lesssim \sqrt{n}\cdot \mathfrak{m}_q(Y)\|f\|_{L_q(\R^n;Y)}. \qedhere
\end{align*}
\end{proof}

\section{Auxiliary  geometric estimates when $X$ is isotropic}\label{sec:convex geometry}

In this section we shall present several geometric estimates on $(X,\|\cdot\|_X)$ that include justifications of statements that were already presented in the Introduction, such as the bound~\eqref{eq:current best} on $I_q(X)M(X)$ when $X$ is isotropic. Recall that the quantities $I_q(X)$ and $M(X)$ were defined in~\eqref{eq:defIqMq}. Also, recall the notation $M_p(X)$ in~\eqref{eq:defMq} and that   we use the more common notation $b(X)=M_\infty(X)$. Some of the ensuing inequalities are elementary, while others are quite deep, since they are deduced below (in a straightforward manner) from a combination of  major results in convex geometry.

\begin{proof}[Proof of~\eqref{eq:current best}] Recall that in the setting of~\eqref{eq:current best} we are given an $n$-dimensional normed space $(X,\|\cdot\|_X)$ and a Hilbertian norm $|\cdot|$ on $X$ (thus identifying $X$ with $\R^n$) such that the isotropicity requirement~\eqref{eq:def isotropic} holds true. By applying~\eqref{eq:def isotropic} with the vector $y\in \R^n$  ranging over an orthonormal basis and summing the (squares of) the resulting identities, we see that $I_2(X)=\sqrt{n}L_X$.

Note that $I_\infty(X)=\max_{x\in B_X}|x|$ is the circumradius of $B_X$ and $b(X)=\max_{x\in S^{n-1}} \|x\|_X$ is the reciprocal of the inradius of $B_X$, and therefore by~\cite{MP89} (see also~\cite[Section~3.2.1]{BGVV14}) we have
\begin{equation}\label{eq:radii}
I_\infty(X)\lesssim \sqrt{n}I_2(X)\qquad \mathrm{and}\qquad  b(X)\lesssim \frac{\sqrt{n}}{I_2(X)}
\end{equation}
A  theorem of Giannopoulos and E.~Milman~\cite{GM14} asserts that
\begin{equation}\label{eq:GM quote}
M(X)\lesssim \frac{(n\log n)^{\frac25}}{I_2(X)}.
\end{equation}
A  substitution of~\eqref{eq:GM quote}  and the second inequality in~\eqref{eq:radii} into the result~\eqref{eq:quote LMS} of Litvak, Milman and Schechtman that we already used earlier in the proof of Lemma~\ref{lem:lip of heat} shows that
\begin{equation}\label{eq:Mq altogether}
\forall\, p\in [1,\infty),\qquad M_p(X)\lesssim \left((n\log n)^{\frac25}+\frac{\sqrt{pn}}{\sqrt{n+p}}\right)\frac{1}{I_2(X)}.
\end{equation}
By a theorem of Paouris~\cite{Pao06} we have
\begin{equation}\label{eq:paouris quote}
\forall\, q\in [2,\infty),\qquad I_q(X)\lesssim \left(1+\frac{q\sqrt{n}}{n+q}\right)I_2(X).
\end{equation}
Note that~\eqref{eq:paouris quote} is stated in~\cite[Theorem~1.2]{Pao06} only for the range $q\in [2,n]$, but when $q\ge n$ the estimate~\eqref{eq:paouris quote} becomes $I_q(X)\lesssim\sqrt{n}I_2(X)$, which follows from the first inequality in~\eqref{eq:radii} since $I_q(X)\le I_\infty(X)$. In conclusion, it follows from~\eqref{eq:Mq altogether} and~\eqref{eq:paouris quote} that for all $p\ge 1$ and $q\ge 2$ we have
\begin{equation}\label{eq:current best-q-p}
I_q(X) M_p(X)\lesssim \left(1+\frac{q\sqrt{n}}{n+q}\right)\left((n\log n)^{\frac25}+\frac{\sqrt{pn}}{\sqrt{n+p}}\right).
\end{equation}
The estimate~\eqref{eq:current best-q-p}  is the currently best known bound on $I_q(X)M_p(X)$ when $X$ is isotropic, the case $p=1$ of which becomes
\begin{equation}\label{eq:current best-q}
I_q(X) M(X)\lesssim \left\{\begin{array}{ll}(n\log n)^{\frac25} & \mathrm{if\ } n\ge q^2,\\
 \frac{q(\log q)^{\frac{2}{5}}}{\sqrt[10]{n}}& \mathrm{if\ }  q\le n\le q^2,\\
 n^{\frac{9}{10}}(\log n)^{\frac25}& \mathrm{if\ }  n\le q.\end{array}\right.
\end{equation}
The first range in the right hand side of~\eqref{eq:current best-q} is precisely the desired estimate~\eqref{eq:current best}.
\end{proof}

\begin{remark}\label{rem:type 2}
For $p\in [1,2]$, the Rademacher type $p$ constant of a normed space $(X,\|\cdot\|_X)$, denoted $T_p(X)$, is the smallest $T\in [1,\infty)$ such that for every $k\in \N$ and every $x_1,\ldots,x_k\in X$ we have
$$
\frac{1}{2^k}\sum_{\e\in \{-1,1\}^k} \bigg\|\sum_{j=1}^k \e_j x_k\bigg\|_X\le T\bigg(\sum_{j=1}^k\|x_j\|_X^p\bigg)^{\frac{1}{p}}.
$$
Observe that $T_2(\ell_2)\le 1$ by the parallelogram identity, and therefore $T_2(X)\le c_2(X)\le \sqrt{\dim (X)}$.

For an isotropic $n$-dimensional normed space  $(X,\|\cdot\|_X)$ with good control on $T_2(X)$, the estimate~\eqref{eq:current best} can be improved by incorporating the work of E.~Milman~\cite{Mil06} into the above argument. Indeed, it is shown in~\cite{Mil06} (see also~\cite[Theorem~9.3.3]{BGVV14}) that $M(X)\lesssim T_2(X)/I_2(X)$. This bound, in combination with the second inequality in~\eqref{eq:radii} and the estimates~\eqref{eq:quote LMS}  and~\eqref{eq:paouris quote} shows that
$$
\forall\, q\in [2,\infty),\qquad n\ge q^2\implies I_q(X)M(X)\lesssim T_2(X).
$$
More generally, for every $p,q\in [1,\infty)$ we have
$$
I_q(X)M_p(X)\lesssim \left(1+\frac{q\sqrt{n}}{n+q}\right)\left(T_2(X)+\frac{\sqrt{pn}}{\sqrt{n+p}}\right).
$$
In particular, $I_q(X)M(X)\lesssim \sqrt{q}\cdot T_2(X)$ for every $q\in [1,\infty)$.
\end{remark}

\begin{remark}\label{rem:unconditional}
Fix $C\ge 1$ and suppose that $(X,\|\cdot\|_X)$ is $C$-unconditional with respect to a Hilbertian norm $|\cdot|$ on $X$, i.e., after identification with $\R^n$ we have $\|(\e_1 x_1,\ldots,\e_n x_n)\|_X\le C\|x\|_X$ for all $x\in \R^n$ and $\e_1,\ldots,\e_n\in \{-1,1\}$. We stated in the Introduction that  $I_q(X)M(X)\lesssim_q C^2\sqrt{\log n}$ for every $q\in [2,\infty)$. Indeed, by a result of Milman and Pajor~\cite[Proposition~b]{MP89}  we have $I_2(X)\lesssim C\sqrt{n}$. Hence, using~\eqref{eq:paouris quote} we see that $I_q(X)\lesssim C(\sqrt{n}+\min\{n,q\})$. Also, by~\cite[Proposition~2.5]{BN03} we have $\|x\|_X\lesssim C\|x\|_{\ell_\infty^n}$ for every $x\in \R^n$, and therefore $M_p(X)\lesssim CM_p(\ell_\infty^n)$ for every $p\in [1,\infty)$.  A standard computation (see e.g.~\cite[Section~5.7]{MS86}) shows that $M(\ell_\infty^n)\asymp\sqrt{\log n}/\sqrt{n}$. Since $b(\ell_\infty^n)=1$, it therefore follows from~\eqref{eq:quote LMS} that $M_p(\ell_\infty^n)\asymp \sqrt{p+\log n}/\sqrt{p+n}$. The above bounds yield
\begin{align*}
\nonumber I_q(X)M_p(X)&\lesssim C^2\left(\sqrt{n}+\min\{n,q\}\right) \sqrt{\frac{p+\log n}{p+n}}.
\end{align*}
In particular, when $p=1$ we see that for every $q\in [2,\infty)$ we have
$$
I_q(X)M(X)\lesssim \left\{\begin{array}{ll}C^2\sqrt{\log n}&\mathrm{if}\  n\ge q^2,\\
\frac{C^2q\sqrt{\log q}}{\sqrt{n}}&\mathrm{if}\  q\le n\le q^2,\\
C^2\sqrt{n\log n}&\mathrm{if}\  n\le q.\end{array}\right.
$$
This implies the desired estimate, and in fact it gives that $I_q(X)M(X) \lesssim C^2\sqrt{q\log n}$.
\end{remark}

We next record some elementary volumetric estimates that yield simple lower bounds on the quantity $I_q(X)M(X)$ that were already used in the Introduction.

\begin{lemma}\label{lem:UV}
Fix $n\in \N$ and $q\in (0,\infty)$. Suppose that $\|\cdot\|_U$ and $\|\cdot\|_V$ are two norms on $\R^n$, with unit balls $B_U$ and $B_V$, respectively.  Then
\begin{equation}\label{eq:UV version}
\bigg(\fint_{B_U}\|u\|_V^q\ud u\bigg)^{\frac{1}{q}}\ge  \left(\frac{n}{n+q}\right)^{\frac{1}{q}}\left(\frac{|B_U|}{|B_V|}\right)^{\frac{1}{n}}.
\end{equation}
\end{lemma}

\begin{proof} Fixing a Euclidean norm $|\cdot|$ on $\R^n$, by integrating in polar coordinates (twice) we see that
\begin{multline}\label{eq:volume polar}
\int_{B_U} \left(\frac{\|u\|_U}{\|u\|_V}\right)^n\ud u=\int_{\R^n} \left(\frac{\|u\|_U}{\|u\|_V}\right)^n\1_{\{\|u\|_U\le 1\}}\ud u=
|B^n|\fint_{S^{n-1}}\left(\frac{\|\sigma\|_U}{\|\sigma\|_V}\right)^n\int_0^{\frac{1}{\|\sigma\|_U}}nr^{n-1}\ud r\ud \sigma\\
=|B^n|\fint_{S^{n-1}}\frac{\ud \sigma}{\|\sigma\|_V^n}=|B^n|\fint_{S^{n-1}}\int_0^{\frac{1}{\|\sigma\|_V}}nr^{n-1}\ud r\ud \sigma=\int_{\R^n}\1_{\{\|v\|_V\le 1\}}\ud v=|B_V|.
\end{multline}
Similarly,
\begin{multline}\label{eq:Iq polar}
\int_{B_U} \left(\frac{\|u\|_V}{\|u\|_U}\right)^q\ud u=
|B^n|\fint_{S^{n-1}}\left(\frac{\|\sigma\|_V}{\|\sigma\|_U}\right)^q \int_0^{\frac{1}{\|\sigma\|_U}}nr^{n-1}\ud r\ud \sigma=|B^n|\fint_{S^{n-1}}\frac{\|\sigma\|_V^q}{\|\sigma\|_U^{q+n}}\ud \sigma\\
=|B^n|\fint_{S^{n-1}} \|\sigma\|_V^q\int_0^{\frac{1}{\|\sigma\|_U}}(n+q)r^{n+q-1}\ud r\ud \sigma= \frac{n+q}{n}\int_{B_U} \|u\|_V^q\ud u.
\end{multline}
Hence, by H\"older's inequality with exponents $(n+q)/n$ and $(n+q)/q$, we see that
\begin{multline}\label{eq:n+1 holder}
|B_U|=\int_{B_U} \left(\frac{\|u \|_V}{\|u\|_U}\right)^{\frac{nq}{n+q}}\left(\frac{\|u\|_U}{\|u \|_V}\right)^{\frac{nq}{n+q}}\ud u\le \bigg(\int_{B_U} \left(\frac{\|u\|_V}{\|u\|_U}\right)^q\ud u\bigg)^{\frac{n}{n+q}}
\bigg(\int_{B_U} \left(\frac{\|u\|_U}{\|u\|_V}\right)^n\ud u\bigg)^{\frac{q}{n+q}}\\\stackrel{\eqref{eq:volume polar}\wedge \eqref{eq:Iq polar}}{=}\left(\frac{n+q}{n}\int_{B_U} \|u\|_V^q\ud u\right)^{\frac{n}{n+q}}|B_V|^{\frac{q}{n+q}}.
\end{multline}
Now, the inequality~\eqref{eq:n+1 holder} simplifies to give the desired estimate~\eqref{eq:UV version}.
\end{proof}

\begin{corollary}\label{coro:lower product}
Fix $p,q\in (0,\infty)$ and $n\in \N$. Suppose that $(X,\|\cdot\|_X)$ is an $n$-dimensional normed space and that $|\cdot|$ is a Hilbertian norm on $X$, thus identifying $X$ with $\R^n$. Then
\begin{equation}\label{eq:IqMq lower}
I_q(X)M_p(X)\ge \left(\frac{n}{n+q}\right)^{\frac{1}{q}}.
\end{equation}
\end{corollary}

\begin{proof}
By an application of Lemma~\ref{lem:UV} with $\|\cdot\|_U=|\cdot|$ and $\|\cdot\|_V=\|\cdot\|_X$, combined with integration in polar coordinates, we see that
\begin{multline*}
\left(\frac{n}{n+p}\right)^{\frac{1}{p}}\left(\frac{|B^n|}{|B_X|}\right)^{\frac{1}{n}}\le \bigg(\fint_{B^n}\|x\|_X^p\ud x\bigg)^{\frac{1}{p}}\\=\bigg(\fint_{S^{n-1}}\int_0^1 nr^{n+p-1}\|\sigma\|_X^p\ud r\ud \sigma\bigg)^{\frac{1}{p}}= \left(\frac{n}{n+p}\right)^{\frac{1}{p}}M_p(X).
\end{multline*}
Hence,
\begin{equation}\label{eq:Mq lower ratio}
M_p(X)\ge \left(\frac{|B^n|}{|B_X|}\right)^{\frac{1}{n}}.
\end{equation}
Also, another application of Lemma~\ref{lem:UV}, this time with $\|\cdot\|_U=\|\cdot\|_X$ and $\|\cdot\|_V=|\cdot|$ shows that
\begin{equation}\label{eq:Iq lower ratio}
I_q(X)\ge \left(\frac{n}{n+q}\right)^{\frac{1}{q}}\left(\frac{|B_X|}{|B^n|}\right)^{\frac{1}{n}}.
\end{equation}
The desired lower bound~\eqref{eq:IqMq lower} now follows by taking the product of~\eqref{eq:Mq lower ratio} and~\eqref{eq:Iq lower ratio}.
\end{proof}

\begin{remark}\label{rem:IqMq Lx} Fix $p,q\in [2,\infty)$. Since $\sqrt[n]{|B^n|}\asymp 1/\sqrt{n}$ and when $X$ is isotropic we have $|B_X|=1$, it follows from~\eqref{eq:Mq lower ratio} that $M_p(X)\gtrsim 1/\sqrt{n}$. Also, $I_q(X)\ge I_2(X)=L_X\sqrt{n}$, so $I_q(X)M_p(X)\gtrsim L_X$.
\end{remark}

\subsection{Wasserstein symmetries}\label{sec:wass} Here we shall provide justifications for statements that we made in Question~\ref{Q:proj lip}. These issues relate to geometric questions  that originated from investigations into quantitative differentiation, but are of interest in their own right. As such, the contents of this section are not needed for the purpose of proving the new results that we stated in the Introduction.

Recall that in Question~\ref{Q:proj lip} we are assuming   that $X$ is isotropic. For every Banach space $(Y,\|\cdot\|_Y)$ and $f\in L_1(B_X;Y)$, the definition of $\Proj f$ in~\eqref{eq:def proj} says that for every $x\in X$ we have
\begin{equation}\label{eq:lambda formula}
\Proj f (x)=\int_{B_X} f(z)\ud z+\frac{1}{L_X^2}\sum_{j=1}^n x_j\int_{B_X} z_jf(z)\ud z=\int_{B_X} f(z)\ud z+\frac{1}{L_X^2}\int_{B_X} (x\cdot z)f(z)\ud z.
\end{equation}

Recall that if $\mu,\nu$ are nonnegative Borel measures on $B_X$ with $\mu(B_X)=\nu(B_X)<\infty$ then a coupling of $\mu$ and $\nu$ is a Borel measure $\pi$ on $B_X\times B_X$ such that $\pi(A\times B_X)=\mu(A)$ and $\pi(B_X\times A)=\nu(A)$ for all Borel $A\subset B_X$. The set of all coupling of $\mu$ and $\nu$ is denoted $\Pi(\mu, \nu)$. Note that $\Pi(\mu,\nu)\neq \emptyset$ because $\mu$ and $\nu$ have the same total mass (specifically, $(\mu\times \nu)/\mu(B_X)\in \Pi(\mu,\nu)$). The Wasserstein-$1$ distance between $\mu$ and $\nu$ associated to the metric that is induced by $\|\cdot\|_X$ is
$$
\mathsf{W}_1^{\|\cdot\|_X}(\mu,\nu)\eqdef \inf_{\pi\in \Pi(\mu,\nu)} \iint_{B_X\times B_X} \|x-y\|_X\ud \pi(x,y).
$$

If $\tau$ is a Borel measure on $X$ with $|\tau|(B_X)<\infty$ and $\tau(B_X)=0$ then write
\begin{equation}\label{eq:def wass norm}
\|\tau\|_{\mathsf{W}_1(B_X,\|\cdot\|_X)}\eqdef \mathsf{W}_1^{\|\cdot\|_X}(\tau^+,\tau^-),
\end{equation}
where $\tau=\tau^+-\tau^-$ and $\tau^+,\tau^-$ are nonnegative Borel measures on $B_X$, which have the same total mass since we are assuming that $\tau(B_X)=0$, so that the definition~\eqref{eq:def wass norm} makes sense. This definition turns the space of all Borel measures $\tau$ on $B_X$ with $|\tau|(B_X)<\infty$ and $\tau(B_X)=0$ into a Banach space, which we denote below by $\mathsf{W}_1(B_X,\|\cdot\|_X)$.

Let $\Lip_0(B_X,\R)$ denote the space of all functions $f:B_X\to \R$ with $f(0)=0$, equipped with the norm $\|\cdot\|_{\Lip(B_X,\R)}$. By the Kantorovich--Rubinstein duality theorem (see~\cite[Theorem~1.14]{Vil03}), we have $\Lip_0(B_X,\R)^*=\mathsf{W}_1(B_X,\|\cdot\|_X)$, with the identification being that a measure $\mu\in  \mathsf{W}_1(B_X,\|\cdot\|_X)$ acts on a function $f\in\Lip_0(B_X,\R)$ through integration, i.e., $\mu(f)=\int_{B_X}f(y)\ud \mu(y)$.

Since $\Proj f=f$ for every constant function $f$, we have
$$
\|\Proj \|_{\Lip(B_X,\R)\to \Lip(B_X,\R)}=\|\Proj \|_{\Lip_0(B_X,\R)\to \Lip_0(B_X,\R)}.
$$
Moreover, if we define an operator $T:\Lip_0(B_X,\R)\to X^*$ by setting
$$
\forall\, x\in X,\qquad Tf(x)\eqdef \frac{1}{L_X^2}\sum_{j=1}^n x_j\int_{B_X} z_jf(z)\ud z,
$$
then by~\eqref{eq:lambda formula} for every $f\in \Lip_0(B_X,\R)$ the linear part of the affine mapping $\Proj f$ is precisely $\frac{1}{L_X^2}Tf$. Hence $L_X^2\|\Proj f\|_{\Lip(B_X,\R)}=\|Tf\|_{X^*}$, and therefore
$$
L_X^2\|\Proj \|_{\Lip_0(B_X,\R)\to \Lip_0(B_X,\R)}=\|T\|_{\Lip_0(B_X,\R)\to X^*}=\|T^*\|_{ X\to \mathsf{W}_1(B_X,\|\cdot\|_X)}.
$$
One computes directly that the adjoint operator $T^*:X\to \Lip_0(B_X,\R)^*=\mathsf{W}_1(B_X,\|\cdot\|_X)$ is such that $T^*x$ is the measure whose density is $y\mapsto (x\cdot y)\1_{B_X}(y)$. Altogether, these observations give
\begin{equation}\label{eq:duality identity w1 proj}
\sup_{x\in \partial B_X} \mathsf{W}_1^{\|\cdot\|_X}\!\left(y\mapsto (x\cdot y)^+\1_{B_X}(y),y\mapsto (x\cdot y)^-\1_{B_X}(y)\right)=L_X^2\|\Proj \|_{\Lip(B_X,\R)\to \Lip(B_X,\R)}.
\end{equation}

\begin{proof}[Proof of~\eqref{eq:reduction to 1 dim}] Since trivially $\|\Proj\|_{\Lip(B_X,\R)\to   \Lip(B_X,\R)}\le \|\Proj\|_{\Lip(B_X,Y)\to   \Lip(B_X,Y)}$ for every Banach space $Y$ with $\dim(Y)\neq 0$, the goal here is to establish the reverse inequality.  Fix $\e\in (0,1)$ and $x\in X$. Let $\nu_x^+,\nu_x^-$ be the measures supported on $B_X$ whose densities are $w\mapsto (x\cdot w)^+$ and $w\mapsto (x\cdot w)^-$, respectively. By~\eqref{eq:duality identity w1 proj} there exists a coupling $\pi_x^\e\in \Pi(\nu_x^+,\nu_x^-)$ with
\begin{equation}\label{eq:plus eps coupling}
\iint_{B_X\times B_X}\|w-z\|_X\ud\pi_x^\e(w,z)\le L_X^2\|\Proj \|_{\Lip(B_X,\R)\to \Lip(B_X,\R)}(\|x\|_X+\e).
\end{equation}
Suppose that $f\in \Lip(B_X,Y)$. Then for every $x,y\in B_X$ we have
\begin{multline*}
\|\Proj f(x)-\Proj f(y)\|_Y=\frac{1}{L_X^2}\bigg\|\int_{B_X} f(w)\ud \nu_{x-y}^+(w)-\int_{B_X} f(w)\ud \nu_{x-y}^-(w)\bigg\|_Y\\
=\frac{1}{L_X^2}\bigg\|\iint_{B_X\times B_X} \big(f(w)-f(z)\big)\ud\pi_{x-y}^\e(w,z)\bigg\|_Y\le \frac{\|f\|_{\Lip(B_X,Y)}}{L_X^2}\iint_{B_X\times B_X}\|w-z\|_X\ud\pi_{x-y}^\e(w,z).
\end{multline*}
Hence, $\|\Proj f\|_{\Lip(B_X,Y)}\le \|\Proj \|_{\Lip(B_X,\R)\to \Lip(B_X,\R)}\|f\|_{\Lip(B_X,Y)}$, by the above estimate combined with~\eqref{eq:plus eps coupling} (and letting $\e\to 0$). So, $\|\Proj\|_{\Lip(B_X,Y)\to   \Lip(B_X,Y)}\le \|\Proj\|_{\Lip(B_X,\R)\to   \Lip(B_X,\R)}$.
\end{proof}

\begin{proof}[Proof of~\eqref{eq:wasserstein symmetry}]  By Borell's lemma~\cite{Bor74} (see \cite[Theorem~2.4.6]{BGVV14}) for every $x\in X$ we have
$$
\int_{B_X}|x\cdot z|\ud z\asymp \bigg(\int_{B_X} (x\cdot y)^2\ud y\bigg)^{\frac12}\stackrel{\eqref{eq:def isotropic}}{=}L_X|x|.
$$
Hence, recalling the definition of $\mu_x^+,\mu_x^-$ in~\eqref{eq:def muxplus}, for every $x\in X\setminus \{0\}$ we have
$$
\mathsf{W}_1^{\|\cdot\|_X}(\mu_x^+,\mu_x^-)\asymp \frac{\mathsf{W}_1^{\|\cdot\|_X}\!\left(y\mapsto (x\cdot y)^+\1_{B_X}(y),y\mapsto (x\cdot y)^-\1_{B_X}(y)\right)}{L_X|x|}.
$$
Therefore~\eqref{eq:wasserstein symmetry} follows from~\eqref{eq:duality identity w1 proj}.
\end{proof}

\section{From $L_q$ affine approximation to Rademacher cotype}\label{sec:deduce cotype}

The notion of Rademacher type $p$ of a Banach space was already recalled in Remark~\ref{rem:type 2}. Specifically, a Banach space $(Y,\|\cdot\|_Y)$ is said to have Rademacher type $p\in [1,2]$ if $T_p(Y)<\infty$, with $T_p(Y)$ as in Remark~\ref{rem:type 2}. A Banach space $(Y,\|\cdot\|_Y)$ is said to have Rademacher cotype $q\in [2,\infty]$ if there exists $C\in (0,\infty)$ such that  for every $k\in \N$ and every $x_1,\ldots,x_k\in Y$ we have
$$
\bigg(\sum_{j=1}^k\|x_j\|_Y^q\bigg)^{\frac{1}{q}}\le \frac{C}{2^k}\sum_{\e\in \{-1,1\}^k} \bigg\|\sum_{j=1}^k \e_j x_k\bigg\|_Y.
$$
The supremum over those $p\in [1,2]$ for which $(Y,\|\cdot\|_Y)$ has Rademacher type $p$ is denoted $p_Y$. The infimum over those $q\in [2,\infty]$ for which $(Y,\|\cdot\|_Y)$ has Rademacher cotype $q$ is denoted $q_Y$.

Recalling the notation in Remark~\ref{rem:Lq approx}, the main content of Proposition~\ref{eq:prop cotype} below is that if $(X,\|\cdot\|_X)$ and $(Y,\|\cdot\|_Y)$ are Banach space with $\dim(X)<\infty$ and  such that for some $K,q,Q\in [2,\infty)$ we have $r_Q^{X\to Y}(\e)\ge \exp(-K/\e^q)$ for every $\e\in (0,1/2]$, then necessarily $q_Y\le q$. The initial idea here is to use an example that was constructed in~\cite{HLN}, which yields a sharp upper bound on the modulus of $L_Q$ affine approximation at $\e$ for a certain function that takes values in $\ell_{q_Y}^m$ for some $m=m(\e)$. The Maurey--Pisier theorem~\cite{MP76} asserts that $Y$ contains a copy of $\ell_{q_Y}^m$, so we can certainly embed this  example of~\cite{HLN} into $Y$. However, a substantial complication occurs here because the affine approximant is now allowed to take values in $Y$ that may fall  outside the given copy of $\ell_{q_Y}^m$, thus precluding our ability to apply the impossibility  result of~\cite{HLN} as a ``black box." This would not be a problem if there existed a projection from $Y$ onto a copy of $\ell_{q_Y}^m$ with norm $O(1)$. However, obtaining such complemented copies of $\ell_{q_Y}^m$ is not possible in general, as exhibited in a remarkable example of Pisier~\cite{Pis83}. Maurey and Pisier considered this complementation issue in~\cite[Remarques~2.9]{MP76}, obtaining a partial result along these lines when $Y$ has Rademacher cotype $q_Y$, $Y^*$ has Rademacher type $p_Y^*$ (i.e., $q_Y$ and $p_{Y^*}$ are {\em attained}), and $1/p_{Y^*}+1/q_Y=1$. The latter condition is satisfied in our setting since $Y$ must be superreflexive by~\cite{BJLPS}, but there is no reason for the critical Rademacher type and cotype to be attained. We overcome this by adapting the proof of~\cite[Remarques~2.9]{MP76} so as to obtain a copy of $\ell_{q_Y}^m$ in $Y$ on which there exists a projection whose norm is bounded by a certain function of $m=m(\e)$ that grows to $\infty$ sufficiently slowly so as to yield the desired result (using the bound on $m(\e)$ as a function of $\e$ that is obtained in the {\em proof of}~\cite[Lemma~16]{HLN}).

\begin{proposition}\label{eq:prop cotype}
Let $(Y,\|\cdot\|_Y)$ be a Banach space such that there exist $K,q,Q\in [1,\infty)$, $n\in \N$ and an $n$-dimensional Banach space $(X,\|\cdot\|_X)$ such that $r_Q^{X\to Y}(\e)\ge \exp(-K/\e^q)$ for every $\e\in (0,1/2]$.
Then $Y$ is superreflexive and $q_Y\le q$. Hence, if in addition $Y$ is a Banach lattice then for every $s\in (q,\infty]$ it admits an equivalent norm whose modulus of uniform convexity is of power type $s$.
\end{proposition}
\begin{proof} The conclusion that $Y$ is superreflexive is a consequence of the work of Bates, Johnson, Lindenstrauss, Preiss and Schechtman~\cite{BJLPS}, namely Theorem~\ref{thm:BJLPS quote} as stated in the Introduction. Indeed, since in the present setting the dependence of the constant $K$ on the dimension $n$ is irrelevant, one may assume that $X=\ell_2^n$, in which case by combining~\cite[Lemma~10]{HLN} and~\cite[Lemma~13]{HLN} we see that the affine mapping $\Lambda$ can also be taken to satisfy $\|\Lambda\|_{\Lip(X,Y)}\le 1$. By~\cite[Lemma~4]{HLN} it follows that the assumption of Proposition~\ref{eq:prop cotype} implies that $r^{X\to Y}(\e)>0$ for every $\e\in (0,1/2]$, and therefore $Y$ is superreflexive by Theorem~~\ref{thm:BJLPS quote}. The stronger conclusion when $Y$ is also a Banach lattice, i.e., that in this case for every $s\in (q,\infty]$ it admits an equivalent norm whose modulus of uniform convexity is of power type $s$, is a formal consequence of the (yet to be proven) conclusion $q_Y\le q$, by well-known structural results for Banach latices~\cite{FJ74,Fie76}  (see also~\cite[Section~1.f]{LT79}).

Due to these comments, the proof of Proposition~\ref{eq:prop cotype} will be complete if we show  that $q_Y\le q$. To that end, by the Maurey--Pisier theorem~\cite{MP76} for every $M\in \N$ there exist $y_1,\ldots,y_M\in Y$ such that
\begin{equation}\label{eq:use MP}
\forall\, a=(a_1,\ldots,a_M)\in \R^M,\qquad \bigg(\sum_{j=1}^M |a_j|^{q_Y}\bigg)^{\frac{1}{q_Y}}\le \bigg\|\sum_{j=1}^M a_j y_j\bigg\|_Y\le 2\bigg(\sum_{j=1}^M |a_j|^{q_Y}\bigg)^{\frac{1}{q_Y}}.
\end{equation}
In particular, $\|y_j\|_Y\ge 1$ for all $j\in \{1,\ldots,M\}$, so by Hahn--Bananch there exist $y_1^*,\ldots,y_M^*\in B_{Y^*}$ such that $y_k^*(y_j)=\d_{kj}$ for every $k,j\in \{1,\ldots,M\}$. Consequently, for every $a\in \R^M$ we have
\begin{equation}\label{eq:py lower}
\bigg\|\sum_{k=1}^M a_ky_k^*\bigg\|_{Y^*}\ge \frac{\left(\sum_{k=1}^M a_k y_k^*\right)\left(\sum_{j=1}^M \sign(a_j)|a_j|^{\frac{1}{q_Y-1}}y_j\right)}{\left\|\sum_{j=1}^M \sign(a_j)|a_j|^{\frac{1}{q_Y-1}}y_j\right\|_Y}\stackrel{\eqref{eq:use MP}}{\ge} \frac12\bigg(\sum_{j=1}^M |a_j|^{\frac{q_Y}{q_Y-1}}\bigg)^{\frac{q_Y-1}{q_Y}}.
\end{equation}

Fix $m\in \N$. It follows from~\eqref{eq:py lower} that $\|y_j^*-y_k^*\|_{Y^*}\ge 2^{-1/q_Y}$, so by~\cite[Lemme~1.5]{MP76} (which itself uses and important construction of Brunel and Sucheston~\cite{BS75}), provided $M$ is large enough (as a function of $m$), there exist $k_1,\ldots,k_{2m}\in \{1,\ldots,M\}$ with $k_1<k_2<\ldots< k_{2m}$ such that the vectors $\{y^*_{k_{2j}}-y^*_{k_{2j-1}}\}_{j=1}^m$ are a $3$-unconditional basic sequence in $Y^*$, i.e.,
\begin{equation}\label{eq:3 unconditional}
\forall (b,\e)\in \R^m\times \{-1,1\}^m,\qquad  \bigg\|\sum_{j=1}^m b_j(y^*_{k_{2j}}-y^*_{k_{2j-1}})\bigg\|_{Y^*}\le 3\bigg\|\sum_{j=1}^m \e_jb_j(y^*_{k_{2j}}-y^*_{k_{2j-1}})\bigg\|_{Y^*}.
\end{equation}

Since we have already shown that $Y$ is superreflexive, by the results of~\cite{FP74} and~\cite{Pisier:1975} we have $p_Y>1$, and therefore by Pisier's $K$-convexity theorem~\cite{Pis82} we have $p_{Y^*}=q_Y/(q_Y-1)$. Hence $T_p(Y^*)<\infty$ for every $p\in [1,q_Y/(q_Y-1))$, which implies that for every $b_1,\ldots,b_m\in \R$ we have
\begin{multline}\label{eq:type dual of qy}
\bigg\|\sum_{j=1}^m b_j(y^*_{k_{2j}}-y^*_{k_{2j-1}})\bigg\|_{Y^*}\stackrel{\eqref{eq:3 unconditional}}{\le} \frac{3}{2^m} \sum_{\e\in \{-1,1\}^m} \bigg\|\sum_{j=1}^m \e_jb_j(y^*_{k_{2j}}-y^*_{k_{2j-1}})\bigg\|_{Y^*}\\\le 3T_p(Y^*)\bigg(\sum_{j=1}^m |b_j|^p\left\|y^*_{k_{2j}}-y^*_{k_{2j-1}}\right\|_{Y^*}^p\bigg)^{\frac{1}{p}}\le 6T_p(Y^*)m^{\frac{1}{p}+\frac{1}{q_Y}-1}\bigg(\sum_{j=1}^m |b_j|^{\frac{q_Y}{q_Y-1}}\bigg)^{\frac{q_Y-1}{q_Y}}.
\end{multline}

Consider the subspace $W= \spn\{y_{k_2},y_{k_4},\ldots,y_{k_{2m}}\}\subset Y$. Let $S:\R^m\to W$ be defined by setting $S(a_1,\ldots,a_m)=\sum_{j=1}^m a_jy_{k_{2j}}$. Then due to~\eqref{eq:use MP} we know that $\|a\|_{\ell_{q_Y}^m}\le \|Sa\|_Y\le 2\|a\|_{\ell_{q_Y}^m}$ for every $a\in \R^m$. Next, consider the linear operator $P:Y\to W$  that is defined by
\begin{equation}\label{eq:def projection}
\forall\, y\in Y,\qquad Py\eqdef \sum_{j=1}^m \left(y_{k_{2j}}^*(y)-y_{k_{2j-1}}^*(y)\right)y_{k_{2j}}.
\end{equation}
Then $P$ is a projection onto $W$. We claim that $P$ satisfies the following operator norm bound.
\begin{equation}\label{eq:norm P}
\|P\|_{Y\to W}\le 12T_p(Y^*)m^{\frac{1}{p}+\frac{1}{q_Y}-1}\|y\|_Y.
\end{equation}
Indeed, fixing $y\in Y$, if we define
\begin{equation}\label{eq:def by choice}
\forall\, j\in \{1,\ldots,m\},\qquad b_j(y)\eqdef \frac{\big|y_{k_{2j}}^*(y)-y_{k_{2j-1}}^*(y)\big|^{q_Y-1}\sign\big(y_{k_{2j}}^*(y)-y_{k_{2j-1}}^*(y)\big) }{\Big(\sum_{j=1}^n\big|y_{k_{2j}}^*(y)-y_{k_{2j-1}}^*(y)\big|^{q_Y}\Big)^{1-\frac{1}{q_Y}}}\in \R,
\end{equation}
i.e., $(b_j(y))_{j=1}^m\in B_{\ell_{q_Y/(q_Y-1)}}^m$ is the normalizing functional of $(y_{k_{2j}}^*(y)-y_{k_{2j-1}}^*(y))_{j=1}^m\in \ell_{q_Y}^m$, then
\begin{multline*}
\|Py\|_Y\stackrel{\eqref{eq:use MP}}{\le} 2\bigg(\sum_{j=1}^n\left|y_{k_{2j}}^*(y)-y_{k_{2j-1}}^*(y)\right|^{q_Y}\bigg)^{\frac{1}{q_Y}}
\stackrel{\eqref{eq:def by choice}}{=}2\bigg(\sum_{j=1}^m b_j(y)(y^*_{k_{2j}}-y^*_{k_{2j-1}})\bigg)(y)\\\le 2\bigg\|\sum_{j=1}^m b_j(y)(y^*_{k_{2j}}-y^*_{k_{2j-1}})\bigg\|_{Y^*}\|y\|_Y\stackrel{\eqref{eq:type dual of qy}\wedge \eqref{eq:def by choice}}{\le} 12T_p(Y^*)m^{\frac{1}{p}+\frac{1}{q_Y}-1}\|y\|_Y,
\end{multline*}
thus establishing the validity of~\eqref{eq:norm P}.

By~\cite[Lemma~16]{HLN} there is a universal constant $\eta\in (0,1)$ with the following property. For every $m\in \N$ there exists a function $\phi^m:\R\to \ell_{q_Y}^m$ with $\|\f^m\|_{\Lip(\R,\ell_{q_Y}^m)}\le 1$ such that for every $Q\in [1,\infty]$ and every affine mapping $\Lambda:\R\to \ell_{q_Y}^m$, if $a,b\in [-1,1]$ satisfy $a\le b$ and $b-a\ge 4/2^m$ then
\begin{equation}\label{eq:quote R lemma}
\bigg(\frac{1}{b-a}\int_a^b\|\f^m(x)-\Lambda(x)\|_{\ell_{q_Y}^m}^Q\ud x\bigg)^{\frac{1}{Q}}\ge  \frac{\eta}{m^{\frac{1}{q_Y}}}\cdot\frac{b-a}{2}.
\end{equation}
We note that the above assertion does not appear in the statement of Lemma~16 of~\cite{HLN} but it is stated explicitly in its (short) proof.  In what follows it will be convenient to denote the coordinates of $\f^m$ by $\f^m_1,\ldots,\f^m_m:\R\to \R$, thus  $\f^m(x)=(\f^m_1(x),\ldots,\f^m_m(x))\in \R^m$ for every $x\in \R$.

By John's theorem~\cite{Joh48} we can identify $X$ (as a real vector space) with $\R^n$ so that for every $x\in X$ we have $\|x\|_{\ell_\infty^n}\le \|x\|_X\le n\|x\|_{\ell_\infty^n}$. (By~\cite{Gia95} the factor of $n$ here can be improved to $O(n^{5/6})$, but this is not important in the present context.) Define $f^m:\R^n\to Y$ by
\begin{equation}\label{eq:def fm}
\forall\, x\in \R^n,\qquad f^m(x)\eqdef \frac12 S\circ \f^m(x_1)= \frac{1}{2}\sum_{j=1}^m \f^m_j(x_1)y_{k_{2j}}.
\end{equation}
Thus $f^m(x)$ depends only on the first coordinate of $x$. Since $\|\f^m\|_{\Lip(\R,\ell_{q_Y}^m)}\le 1$ and $\|\cdot\|_{\ell_\infty^n}\le\|\cdot\|_X$, by~\eqref{eq:use MP} we have $\|f^m\|_{\Lip(X,Y)}\le 1$.  So, by our underlying assumption that $r_Q^{X\to Y}(\e)\ge \exp(-K/\e^q)$ for every $\e\in (0,1/2]$, there exists a radius $\rho\in (0,1)$ with
\begin{equation}\label{eq:rho lower counter example}
\rho\ge \exp\left(-\frac{K}{\eta^q}5^{2q}T_p(Y^*)^qn^{\frac{(n+Q)q}{Q}}m^{\frac{q}{p}+\frac{2q}{q_Y}-q}\right),
\end{equation}
a point $x\in B_X$ with $x+\rho B_X\subset B_X$, and an affine mapping $\Lambda:\R^n\to Y$ such that
\begin{equation}\label{eq:use rho lower prop}
 \bigg(\fint_{x+\rho B_X} \left\|f^m(y)-\Lambda(y)\right\|_Y^Q\ud y\bigg)^{\frac{1}{Q}}\le \frac{\eta}{25T_p(Y^*)n^{\frac{n}{Q}+1}m^{\frac{1}{p}+\frac{2}{q_Y}-1}}\rho.
\end{equation}

Since $\|\cdot\|_{\ell_\infty^n}\le \|\cdot\|_X\le n\|\cdot\|_{\ell_\infty^n}$ we have $x+\frac{\rho}{n}[-1,1]^n\subset x+\rho B_X\subset x+\rho[-1,1]^n$. So, by~\eqref{eq:use rho lower prop},
\begin{equation}\label{eq:use rho lower prop on cube}
 \bigg(\fint_{x+\frac{\rho}{n}[-1,1]^n}\left\|f^m(y)-\Lambda(y)\right\|_Y^Q\ud y\bigg)^{\frac{1}{Q}}\le \frac{\eta\rho}{25T_p(Y^*)n m^{\frac{1}{p}+\frac{2}{q_Y}-1}}.
\end{equation}
We claim that~\eqref{eq:use rho lower prop on cube} implies that necessarily $\rho/n<1/2^{m-1}$. Once proven, this assertion may be contrasted with~\eqref{eq:rho lower counter example} to deduce that
\begin{equation}\label{eq:contrast radii}
\forall(m,q)\in \N\times \left[1,\frac{q_Y}{q_Y-1}\right),\qquad 2^{m-1}<n\exp\bigg(\frac{K}{\eta^q}5^{2q}T_p(Y^*)^qn^{\frac{(n+Q)q}{Q}}m^{\frac{q}{p}+\frac{2q}{q_Y}-q}\bigg).
\end{equation}
By letting $m\to\infty$ in~\eqref{eq:contrast radii} we see that $
q/p+2q/q_Y-q\ge 1$. By letting $p\to q_Y/(q_Y-1)$, we conclude that $q(q_Y-1)/q_Y+2q/q_Y-q\ge 1$, which simplifies to give the desired estimate $q_Y\le q$.

It therefore remains to prove that $\rho/n<1/2^{m-1}$. To this end, assume for the sake of obtaining a contradiction that $2\rho/n\ge 4/2^{m}$. Since $x+\rho B_X\subset B_X$ we have $\|x\|_{\ell_\infty^n}\le \|x\|_X\le (1-\rho)$. Consequently, $x_1-\rho/n,x_1+\rho/n\in [-1,1]$. Hence, by an application of~\eqref{eq:quote R lemma} with $a=x_1-\rho/n$ and $b=x_1+\rho/n$ (so that our contrapositive assumption implies that indeed $b-a\ge 4/2^m$), for every fixed $(y_2,\ldots,y_{n})\in (x_2,\ldots,x_n)+ [-\rho/n,\rho/n]^{n-1}$ we may consider the affine function $(y_1\in \R)\mapsto S^{-1}\circ P\circ \Lambda(y_1,y_2,\ldots,y_n)\in \ell_{q_Y}^m$ to deduce that
\begin{equation}\label{eq:to average Q}
\fint_{x_1-\frac{\rho}{n}}^{x_1+\frac{\rho}{n}} \left\|\phi^m(y_1)-S^{-1}\circ P\circ \Lambda(y_1,y_2,\ldots,y_n)\right\|_{\ell_{q_Y}^m}^Q\ud y_1\ge \left(\frac{\eta\rho}{nm^{\frac{1}{q_Y}}}\right)^Q.
\end{equation}
By averaging~\eqref{eq:to average Q} over $(y_2,\ldots,y_{n})\in (x_2,\ldots,x_n)+ [-\rho/n,\rho/n]^{n-1}$ , we therefore have
\begin{align}
\nonumber \frac{\eta \rho}{nm^{\frac{1}{q_Y}}}&\le \bigg( \fint_{x+\frac{\rho}{n}[-1,1]^n}\left\|\f^m(y_1)-S^{-1}\circ P\circ\Lambda(y)\right\|_{\ell_{q_Y}^m}^Q\ud y\bigg)^{\frac{1}{Q}}\\ &=2\bigg( \fint_{x+\frac{\rho}{n}[-1,1]^n}\left\|S^{-1}\circ P \big(f^m(y)-\Lambda(y)\big)\right\|_{\ell_{q_Y}^m}^Q\ud y\bigg)^{\frac{1}{Q}}\label{eq:use def fm}\\&\le 24T_p(Y^*)m^{\frac{1}{p}+\frac{1}{q_Y}-1}\bigg(\fint_{x+\frac{\rho}{n}[-1,1]^n}\left\|f^m(y)-\Lambda(y)\right\|_Y^Q\ud y\bigg)^{\frac{1}{Q}},\label{eq:use norm bound}
\end{align}
where in~\eqref{eq:use def fm} we used the definition of $f^m$ in~\eqref{eq:def fm} and the fact that $Pf^m=f^m$ (since $f^m$ takes values in the subspace $W$ and $P$ is a projection onto $W$), and in~\eqref{eq:use norm bound} we used the norm bound~\eqref{eq:norm P}.
The desired contradiction now follows by contrasting~\eqref{eq:use rho lower prop on cube} with~\eqref{eq:use norm bound}.
\end{proof}

The assumption~\eqref{eq:superreflexiv char} of Theorem~\ref{thm:superreflexive char 1} implies a local Dorronsoro inequality as in Theorem~\ref{thm:local doro}, which in turn implies an estimate of the form $r_q^{X\to Y}(\e)\ge \exp(-K/\e^q)$ as in Proposition~\ref{eq:prop cotype}. In these implications the constants deteriorate, but for the purpose of Theorem~\ref{thm:superreflexive char 1} and Proposition~\ref{eq:prop cotype} constants are not important (all that matters is that they are independent of $\e$). So, due to Proposition~\ref{eq:prop cotype} we have proven Theorem~\ref{thm:superreflexive char 1} and the statement in the paragraph that  follows Question~\ref{Q:deduce q conv}.

\section{Explicit computations for real-valued mappings on Euclidean space}\label{sec:computations}
For concreteness, below we shall fix the following normalization for the Fourier transform on $\R^n$.
$$
\forall\, f\in L_1(\R^n),\ \forall\xi\in \R^n,\qquad \hat{f}(\xi)\eqdef \frac{1}{(2\pi)^{\frac{n}{2}}}\int_{\R^n} f(x)e^{-ix\cdot\xi}\ud x.
$$
Thus, by the Plancherel theorem, every smooth compactly supported function $f:\R^n\to \R$ satisfies
\begin{equation}\label{eq:plancherel gradient}
\bigg(\int_{\R^n} |\xi|^2\cdot \left|\hat{f}(\xi)\right|^2\ud \xi\bigg)^{\frac12}=\bigg(\int_{\R^n} |\nabla f(x)|^2\ud x\bigg)^{\frac12}=\sqrt{n} \bigg(\fint_{S^{n-1}} \left\|\sigma\cdot \nabla f\right\|_{L_2(\R^n)}^2\ud \sigma \bigg)^{\frac12}.
\end{equation}
Also, $\hat{P_tf}(\xi)=e^{-t|\xi|}\hat{f}(\xi)$ and $\hat{H_tf}(\xi)=e^{-t|\xi|^2}\hat{f}(\xi)$ for every $t\in [0,\infty)$, $\xi\in \R^n$ and $f\in L_1(\R^n)$.


\medskip
\noindent{\bf The heat semigroup.}
Fix a smooth compactly supported function $f:\R^n\to \R$ and a parameter $\gamma\in (0,\infty)$. For every $t\in (0,\infty)$ and $z\in \R^n$ the Fourier transform of the function
$$
(x\in \R^n)\mapsto f(x+tz)-\Taylor_x^1(H_{\gamma t^2}f)(x+tz)=f(x+tz)-H_{\gamma t^2}f(x)-tz\cdot\nabla H_{\gamma t^2}f(x)
$$
is given by
$$
(\xi\in\R^n)\mapsto e^{itz\cdot\xi}\hat{f}(\xi)-\left(e^{-\gamma t^2|\xi|^2}\hat{f}(\xi)+it z\cdot \xi e^{-\gamma t^2|\xi|^2}\hat{f}(\xi)\right)=\left( e^{itz\cdot\xi}-(1+itz\cdot \xi)e^{-\gamma t^2|\xi|^2}\right)\hat{f}(\xi).
$$
By the Plancherel theorem we therefore have
\begin{equation}\label{eq:use plancherel heat}
\!\!\! \int_{\R^n} \left|f(x+tz)-\Taylor_x^1(H_{\gamma t^2}f)(x+tz)\right|^2 \ud x= \int_{\R^n} \left| e^{itz\cdot\xi}-(1+itz\cdot \xi)e^{-\gamma t^2|\xi|^2}\right|^2\cdot\left|\hat{f}(\xi)\right|^2\ud \xi.
\end{equation}

By rotation invariance the following identity holds true for every $\xi\in \R^n$ and $t\in (0,\infty)$.
\begin{align*}
\nonumber \fint_{B^n} \left| e^{itz\cdot\xi}-(1+itz\cdot \xi)e^{-\gamma t^2|\xi|^2}\right|^2\ud z&=\fint_{B^n} \left| e^{it|\xi|z_1}-(1+it|\xi|z_1)e^{-\gamma t^2|\xi|^2}\right|^2\ud z \\&=
\frac{|B^{n-1}|}{|B^n|}\int_{-1}^1\left| e^{it|\xi|u}-(1+it|\xi|u)e^{-\gamma t^2|\xi|^2}\right|^2(1-u^2)^{\frac{n-1}{2}}\ud u.
\end{align*}
Hence, using the change of variable $s=t|\xi|$ and a substitution of the values of $|B^{n-1}|$ and $|B^n|$,
\begin{multline}\label{eq:Gamma ratio}
\int_0^\infty\fint_{B^n} \left| e^{itz\cdot\xi}-(1+itz\cdot \xi)e^{-\gamma t^2|\xi|^2}\right|^2\ud z\frac{\ud t}{t^3}\\=\sqrt{\pi}\cdot\frac{\Gamma\left(\frac{n}{2}+1\right)}{\Gamma\left(\frac{n+1}{2}\right)}|\xi|^2\int_0^\infty\int_{-1}^1\left| e^{isu}-(1+isu)e^{-\gamma s^2}\right|^2(1-u^2)^{\frac{n-1}{2}}\ud u\frac{\ud s}{s^3}.
\end{multline}
Consequently, if we introduce the notation
\begin{align}\label{eq:def knc}
\nonumber\mathcal{k}(n,\gamma)&\eqdef n\sqrt{\pi}\cdot\frac{\Gamma\left(\frac{n}{2}+1\right)}{\Gamma\left(\frac{n+1}{2}\right)}\int_0^\infty\int_{-1}^1\left| e^{isu}-(1+isu)e^{-\gamma s^2}\right|^2(1-u^2)^{\frac{n-1}{2}}\ud u\frac{\ud s}{s^3}\\
&\asymp n^{\frac32}\int_{0}^1\int_0^\infty\left(\left( \cos(su)-e^{-\gamma s^2}\right)^2+\left( \sin(su)-sue^{-\gamma s^2}\right)^2\right)\frac{(1-u^2)^{\frac{n-1}{2}}}{s^3}\ud s\ud u,
\end{align}
then a combination of~\eqref{eq:plancherel gradient}, \eqref{eq:use plancherel heat} and~\eqref{eq:Gamma ratio} implies the validity of the following identity.
\begin{multline}\label{eq:crucial heat identity}
\bigg(\int_{\R^n}\int_0^\infty \fint_{x+tB^n}\frac{\big(f(y)-\Taylor_x^1(H_{\gamma t^2}f)(y)\big)^2}{t^{3}}\ud y \ud t\ud x\bigg)^{\frac12}\\ =  \sqrt{\frac{\mathcal{k}(n,\gamma)}{n}} \bigg(\int_{\R^n} |\xi|^2\cdot \left|\hat{f}(\xi)\right|^2\ud \xi\bigg)^{\frac12}= \sqrt{\mathcal{k}(n,\gamma)}\bigg(\fint_{S^{n-1}} \left\|\sigma\cdot \nabla f\right\|_{L_2(\R^n)}^2\ud \sigma \bigg)^{\frac12}.
\end{multline}

Despite the fact that~\eqref{eq:crucial heat identity} is stated for real-valued mappings, the corresponding identity automatically holds true for Hilbert-space valued mappings as well by an application of~\eqref{eq:crucial heat identity} to each of the coordinates with respect to an orthonormal basis, i.e.,  if $\mathcal{H}$ is a Hilbert space then for every $n\in \N$, every $\gamma\in (0,\infty)$ and every smooth compactly supported  $f:\R^n\to \mathcal{H}$ we have
\begin{multline}\label{eq:hilbert identity}
\bigg(\int_{\R^n}\int_0^\infty \fint_{x+tB^n}\left\|f(y)-\Taylor_x^1(H_{\gamma t^2}f)(y)\right\|_{\mathcal{H}}^2\ud y \frac{\ud t}{t^3}\ud x\bigg)^{\frac12}\\= \sqrt{\mathcal{k}(n,\gamma)} \bigg(\fint_{S^{n-1}} \|\sigma\cdot \nabla f\|_{L_2(\R^n;\mathcal{H})}^2\ud \sigma\bigg)^{\frac12}.
\end{multline}

The following lemma contains a (sharp) upper bound on the quantity $\mathcal{k}(n,\gamma)$.

\begin{lemma}\label{lem:k bound} For every $n\in \N$ and $\gamma\in (0,\infty)$ we have
\begin{equation}\label{eq:knc bound}
\mathcal{k}(n,\gamma)\lesssim \gamma n+\int_0^\infty v^2e^{-v^2}\log\left(2+\frac{v^2+\gamma n}{v\sqrt{\gamma n}}\right)\ud v.
\end{equation}
\end{lemma}

Prior to proving Lemma~\ref{lem:k bound}, we record the following corollary (corresponding to a substitution of the special case $\gamma=1/n$ of Lemma~\ref{lem:k bound} into~\eqref{eq:hilbert identity})  that was already stated in the Introduction, where it was noted that it implies the improved estimate~\eqref{eq:r2 hilbert} on the modulus of $L_2$ affine approximation.

\begin{corollary} Suppose that $\mathcal{H}$ is a Hilbert space. Then for every $n\in \N$, every smooth compactly supported function $f:\R^n\to \mathcal{H}$ satisfies
$$
\bigg(\int_{\R^n}\int_0^\infty \fint_{x+tB^n}\left\|f(y)-\Taylor_x^1\Big(H_{\frac{t^2}{n}}f\Big)(y)\right\|_{\mathcal{H}}^2\ud y \frac{\ud t}{t^3}\ud x\bigg)^{\frac12}\lesssim \bigg(\fint_{S^{n-1}} \|\sigma\cdot \nabla f\|_{L_2(\R^n;\mathcal{H})}^2\ud \sigma\bigg)^{\frac12}.
$$
\end{corollary}

\begin{proof}[Proof of Lemma~\ref{lem:k bound}] We shall estimate  the two integrals that correspond to each of the summands that appear in the right hand side of~\eqref{eq:def knc}  separately. Firstly, consider the elementary estimate
$$
\forall\, a,b\in [0,\infty),\qquad \left|\cos(a)-e^{-b}\right|\le |\cos(a)-1|+\left|1-e^{-b}\right|\asymp \min\{a^2,1\}+\min\{b,1\},
$$
which implies that for every $u\in (0,\infty)$ we have
$$
\int_0^\infty\left( \cos(su)-e^{-\gamma s^2}\right)^2\frac{\ud s}{s^3}\lesssim u^4\int_0^{\frac{1}{u}} s\ud u+\int_{\frac{1}{u}}^\infty\frac{\ud s}{s^3}+\gamma^2\int_0^{\frac{1}{\sqrt{\gamma}}}s\ud s+\int_{\frac{1}{\sqrt{\gamma}}}^\infty \frac{\ud s}{s^3}\asymp u^2+\gamma.
$$
Therefore,
\begin{equation}\label{eq:first u int}
n^{\frac32}\int_{0}^1\int_0^\infty\left( \cos(su)-e^{-\gamma s^2}\right)^2\frac{(1-u^2)^{\frac{n-1}{2}}}{s^3}\ud s\ud u\lesssim n^{\frac32}\int_0^1 (u^2+\gamma)e^{-\frac{n-1}{2}u^2}\ud u\asymp 1+\gamma n.
\end{equation}
Secondly, consider the elementary estimate
$$
\forall (a,b)\in [0,1]\times [0,\infty),\qquad \left|\sin(a)-ae^{-b}\right|\le |\sin(a)-a|+a\left|1-e^{-b}\right|\asymp a^2+a\min\{1,b\},
$$
which implies that for every $u\in (0,1]$ we have
\begin{multline*}
\int_0^{\frac{1}{u}}\left( \sin(su)-sue^{-\gamma s^2}\right)^2\frac{\ud s}{s^3}\lesssim u^4\int_0^{\frac{1}{u}}s\ud s+u^2\int_0^{\frac{1}{u}} \frac{\min\{1,\gamma^2s^4\}}{s}\ud s\\\asymp u^2+\min\left\{\frac{\gamma^2}{u^2},u^2\right\}+u^2\log\left(\max\left\{1,\frac{\sqrt{\gamma}}{u}\right\}\right)\asymp u^2\log\left(2+\frac{\sqrt{\gamma}}{u}\right).
\end{multline*}
Consequently, if $n\ge 2$ then using the elementary inequality $(1-u^2)^{(n-1)/2}\le e^{-nu^2/4}$ we see that
\begin{multline}\label{eq:up to 1/u}
n^{\frac32}\int_0^1\int_0^{\frac{1}{u}}\left( \sin(su)-sue^{-\gamma s^2}\right)^2\frac{(1-u^2)^{\frac{n-1}{2}}}{s^3}\ud s\ud u\lesssim n^{\frac32}\int_0^1 u^2e^{-\frac{nu^2}{4}}\log\left(2+\frac{\sqrt{\gamma}}{u}\right)\ud u\\
= 8\int_0^{\frac{\sqrt{n}}{2}} v^2e^{-v^2}\log\left(2+\frac{\sqrt{\gamma n}}{2v}\right)\ud v\lesssim \int_0^{\infty}
v^2e^{-v^2}\log\left(2+\frac{\sqrt{\gamma n}}{v}\right)\ud v.
\end{multline}
When $n=1$ the leftmost term in~\eqref{eq:up to 1/u} is bounded from above by a universal constant, and therefore it is bounded above by a constant multiple of the rightmost term in~\eqref{eq:up to 1/u} in the case $n=1$ as well.

In a similar fashion, consider the elementary estimate
$$
\forall (a,b)\in [1,\infty)\times [0,\infty),\qquad \left|\sin(a)-ae^{-b}\right|\le |\sin(a)|+ae^{-b}\le 1+ae^{-b},
$$
which implies that for every $u\in (0,1]$ we have
\begin{multline*}
\int_{\frac{1}{u}}^\infty\left( \sin(su)-sue^{-\gamma s^2}\right)^2\frac{\ud s}{s^3}\lesssim \int_{\frac{1}{u}}^\infty\frac{\ud s}{s^3}+ u^2\int_{\frac{1}{u}}^\infty\frac{e^{-2\gamma s^2}}{s}\ud s\\\asymp u^2+u^2\int_{\frac{\sqrt{2\gamma}}{u}}^\infty\frac{e^{-t^2}}{t}\ud t\lesssim u^2+u^2\int_{\frac{\sqrt{2\gamma}}{u}}^{\max\left\{\frac{\sqrt{2\gamma}}{u},1\right\}}\frac{\ud t}{t}+u^2\int_1^\infty e^{-t^2}\ud t\lesssim u^2\log\left(2+\frac{u}{\sqrt{\gamma}}\right).
\end{multline*}
By integrating this inequality with respect to $u$, when $n\ge 2$ we therefore have
\begin{multline}\label{eq:from 1/u}
n^{\frac32}\int_0^1\int_{\frac{1}{u}}^\infty\left( \sin(su)-sue^{-\gamma s^2}\right)^2\frac{(1-u^2)^{\frac{n-1}{2}}}{s^3}\ud s\ud u\lesssim n^{\frac32}\int_0^1 u^2e^{-\frac{nu^2}{4}}\log\left(2+\frac{u}{\sqrt{\gamma}}\right)\ud u\\
=8\int_0^{\frac{\sqrt{n}}{2}}v^2e^{-v^2}\log\left(2+\frac{2v}{\sqrt{\gamma n}}\right)\ud v \lesssim \int_0^\infty v^2e^{-v^2}\log\left(2+\frac{v}{\sqrt{\gamma n}}\right)\ud v.
\end{multline}
As before, the leftmost term in~\eqref{eq:from 1/u} is bounded by a universal constant multiple of the rightmost term of~\eqref{eq:from 1/u} in the case $n=1$ as well.

By summing~\eqref{eq:up to 1/u} and~\eqref{eq:from 1/u} while using the fact that $\log[(2+a)(2+1/a)]\asymp\log (2+(a^2+1)/a)$ for every $a\in (0,\infty)$, we conclude that
\begin{equation}\label{eq:sinus integral}
n^{\frac32}\int_0^1\int_{0}^\infty\left( \sin(su)-sue^{-\gamma s^2}\right)^2\frac{(1-u^2)^{\frac{n-1}{2}}}{s^3}\ud s\ud u\lesssim \int_0^\infty v^2e^{-v^2}\log\left(2+\frac{v^2+\gamma n}{v\sqrt{\gamma n}}\right)\ud v.
\end{equation}
Recalling~\eqref{eq:def knc}, the desired estimate~\eqref{eq:knc bound} now follows from~\eqref{eq:first u int} and~\eqref{eq:sinus integral}.
\end{proof}

\medskip
\noindent{\bf  The Poisson semigroup.}
Fix $\gamma\in (0,\infty)$ and a nonconstant smooth compactly supported function $f:\R^n\to \R$. Arguing analogously to~\eqref{eq:use plancherel heat}, by the Plancherel theorem we have
\begin{equation*}
\!\!\! \int_{\R^n} \left|f(x+tz)-\Taylor_x^1(P_{\gamma t}f)(x+tz)\right|^2 \ud x= \int_{\R^n} \left| e^{itz\cdot\xi}-(1+itz\cdot \xi)e^{-\gamma t|\xi|}\right|^2\cdot\left|\hat{f}(\xi)\right|^2\ud \xi.
\end{equation*}
From here, the same reasoning that led to the identity~\eqref{eq:crucial heat identity} shows that
\begin{multline}\label{eq:poisson bad}
\int_{\R^n}\int_0^\infty \fint_{x+tB^n}\frac{\big(f(y)-\Taylor_x^1(H_{\gamma t^2}f)(y)\big)^2}{t^{3}}\ud y \ud t\ud x\\ =   c_n\bigg(\int_0^\infty\int_{-1}^1\left| e^{isu}-(1+isu)e^{-\gamma s}\right|^2(1-u^2)^{\frac{n-1}{2}}\ud u\frac{\ud s}{s^3}\bigg)\fint_{S^{n-1}} \left\|\sigma\cdot \nabla f\right\|_{L_2(\R^n)}^2\ud \sigma,
\end{multline}
where $c_n=n\sqrt{\pi}\Gamma(1+n/2)/\Gamma((n+1)/2)$. But, for fixed $u\in (-1,1)$ when $s\to 0$ the integrand of the first integral in the right hand side of~\eqref{eq:poisson bad} is asymptotic to $(1-u^2)^{(n-1)/2}\gamma^2/s$. So, the first integral in the right hand side of~\eqref{eq:poisson bad} diverges, implying that the left hand side of~\eqref{eq:poisson bad} is infinite.

\bigskip
\noindent{\bf Acknowledgements.} We are grateful to Apostolos Giannopoulos and Gilles Pisier for providing helpful pointers to the literature.

\bibliography{heatSemigr}

\begin{thebibliography}{10}

\bibitem{AS12}
J.~Azzam and R.~Schul.
\newblock Hard {S}ard: quantitative implicit function and extension theorems
  for {L}ipschitz maps.
\newblock {\em Geom. Funct. Anal.}, 22(5):1062--1123, 2012.

\bibitem{AS14}
J.~Azzam and R.~Schul.
\newblock A quantitative metric differentiation theorem.
\newblock {\em Proc. Amer. Math. Soc.}, 142(4):1351--1357, 2014.

\bibitem{Bal88}
K.~Ball.
\newblock Logarithmically concave functions and sections of convex sets in
  {${\bf R}^n$}.
\newblock {\em Studia Math.}, 88(1):69--84, 1988.

\bibitem{Bal92}
K.~Ball.
\newblock Markov chains, {R}iesz transforms and {L}ipschitz maps.
\newblock {\em Geom. Funct. Anal.}, 2(2):137--172, 1992.

\bibitem{BCL94}
K.~Ball, E.~A. Carlen, and E.~H. Lieb.
\newblock Sharp uniform convexity and smoothness inequalities for trace norms.
\newblock {\em Invent. Math.}, 115(3):463--482, 1994.

\bibitem{BJLPS}
S.~Bates, W.~B. Johnson, J.~Lindenstrauss, D.~Preiss, and G.~Schechtman.
\newblock Affine approximation of {L}ipschitz functions and nonlinear
  quotients.
\newblock {\em Geom. Funct. Anal.}, 9(6):1092--1127, 1999.

\bibitem{Ben85}
Y.~Benyamini.
\newblock The uniform classification of {B}anach spaces.
\newblock In {\em Texas functional analysis seminar 1984--1985 ({A}ustin,
  {T}ex.)}, Longhorn Notes, pages 15--38. Univ. Texas Press, Austin, TX, 1985.

\bibitem{BL00}
Y.~Benyamini and J.~Lindenstrauss.
\newblock {\em Geometric nonlinear functional analysis. {V}ol. 1}, volume~48 of
  {\em American Mathematical Society Colloquium Publications}.
\newblock American Mathematical Society, Providence, RI, 2000.

\bibitem{BN03}
S.~G. Bobkov and F.~L. Nazarov.
\newblock On convex bodies and log-concave probability measures with
  unconditional basis.
\newblock In {\em Geometric aspects of functional analysis}, volume 1807 of
  {\em Lecture Notes in Math.}, pages 53--69. Springer, Berlin, 2003.

\bibitem{Bor74}
C.~Borell.
\newblock Convex measures on locally convex spaces.
\newblock {\em Ark. Mat.}, 12:239--252, 1974.

\bibitem{Bou83}
J.~Bourgain.
\newblock Some remarks on {B}anach spaces in which martingale difference
  sequences are unconditional.
\newblock {\em Ark. Mat.}, 21(2):163--168, 1983.

\bibitem{Bou86}
J.~Bourgain.
\newblock On high-dimensional maximal functions associated to convex bodies.
\newblock {\em Amer. J. Math.}, 108(6):1467--1476, 1986.

\bibitem{Bour:Lip}
J.~Bourgain.
\newblock Remarks on the extension of {L}ipschitz maps defined on discrete sets
  and uniform homeomorphisms.
\newblock In {\em Geometrical aspects of functional analysis (1985/86)}, volume
  1267 of {\em Lecture Notes in Math.}, pages 157--167. Springer, Berlin, 1987.

\bibitem{BGVV14}
S.~Brazitikos, A.~Giannopoulos, P.~Valettas, and B.-H. Vritsiou.
\newblock {\em Geometry of isotropic convex bodies}, volume 196 of {\em
  Mathematical Surveys and Monographs}.
\newblock American Mathematical Society, Providence, RI, 2014.

\bibitem{BS75}
A.~Brunel and L.~Sucheston.
\newblock On {$J$}-convexity and some ergodic super-properties of {B}anach
  spaces.
\newblock {\em Trans. Amer. Math. Soc.}, 204:79--90, 1975.

\bibitem{Che99}
J.~Cheeger.
\newblock Differentiability of {L}ipschitz functions on metric measure spaces.
\newblock {\em Geom. Funct. Anal.}, 9(3):428--517, 1999.

\bibitem{Che12}
J.~Cheeger.
\newblock Quantitative differentiation: a general formulation.
\newblock {\em Comm. Pure Appl. Math.}, 65(12):1641--1670, 2012.

\bibitem{CK06}
J.~Cheeger and B.~Kleiner.
\newblock On the differentiability of {L}ipschitz maps from metric measure
  spaces to {B}anach spaces.
\newblock In {\em Inspired by {S}. {S}. {C}hern}, volume~11 of {\em Nankai
  Tracts Math.}, pages 129--152. World Sci. Publ., Hackensack, NJ, 2006.

\bibitem{CK10-annals}
J.~Cheeger and B.~Kleiner.
\newblock Differentiating maps into {$L^1$}, and the geometry of {BV}
  functions.
\newblock {\em Ann. of Math. (2)}, 171(2):1347--1385, 2010.

\bibitem{CK10-inventiones}
J.~Cheeger and B.~Kleiner.
\newblock Metric differentiation, monotonicity and maps to {$L^1$}.
\newblock {\em Invent. Math.}, 182(2):335--370, 2010.

\bibitem{CKN09}
J.~Cheeger, B.~Kleiner, and A.~Naor.
\newblock A {$(\log n)^{\Omega(1)}$} integrality gap for the sparsest cut
  {SDP}.
\newblock In {\em 2009 50th {A}nnual {IEEE} {S}ymposium on {F}oundations of
  {C}omputer {S}cience ({FOCS} 2009)}, pages 555--564. IEEE Computer Soc., Los
  Alamitos, CA, 2009.

\bibitem{CKN11}
J.~Cheeger, B.~Kleiner, and A.~Naor.
\newblock Compression bounds for {L}ipschitz maps from the {H}eisenberg group
  to {$L_1$}.
\newblock {\em Acta Math.}, 207(2):291--373, 2011.

\bibitem{CN13-2}
J.~Cheeger and A.~Naber.
\newblock Lower bounds on {R}icci curvature and quantitative behavior of
  singular sets.
\newblock {\em Invent. Math.}, 191(2):321--339, 2013.

\bibitem{CN13}
J.~Cheeger and A.~Naber.
\newblock Quantitative stratification and the regularity of harmonic maps and
  minimal currents.
\newblock {\em Comm. Pure Appl. Math.}, 66(6):965--990, 2013.

\bibitem{DF16}
K.~Danailov and C.~Fefferman.
\newblock A note on quantitative differentiation.
\newblock Forthcoming manuscript, 2016.

\bibitem{DS93}
G.~David and S.~Semmes.
\newblock {\em Analysis of and on uniformly rectifiable sets}, volume~38 of
  {\em Mathematical Surveys and Monographs}.
\newblock American Mathematical Society, Providence, RI, 1993.

\bibitem{Dorro:85}
J.~R. Dorronsoro.
\newblock A characterization of potential spaces.
\newblock {\em Proc. Amer. Math. Soc.}, 95(1):21--31, 1985.

\bibitem{Dvo60}
A.~Dvoretzky.
\newblock Some results on convex bodies and {B}anach spaces.
\newblock In {\em Proc. {I}nternat. {S}ympos. {L}inear {S}paces ({J}erusalem,
  1960)}, pages 123--160. Jerusalem Academic Press, Jerusalem; Pergamon,
  Oxford, 1961.

\bibitem{EFW12}
A.~Eskin, D.~Fisher, and K.~Whyte.
\newblock Coarse differentiation of quasi-isometries {I}: {S}paces not
  quasi-isometric to {C}ayley graphs.
\newblock {\em Ann. of Math. (2)}, 176(1):221--260, 2012.

\bibitem{EFW13}
A.~Eskin, D.~Fisher, and K.~Whyte.
\newblock Coarse differentiation of quasi-isometries {II}: {R}igidity for {S}ol
  and lamplighter groups.
\newblock {\em Ann. of Math. (2)}, 177(3):869--910, 2013.

\bibitem{EMR13}
A.~Eskin, H.~Masur, and K.~Rafi.
\newblock Large scale rank of {T}eichm\"uller space.
\newblock Preprint, available at \url{http://arxiv.org/abs/1307.3733}, 2013.

\bibitem{EMR15}
A.~Eskin, H.~Masur, and K.~Rafi.
\newblock Rigidity of of {T}eichm\"uller space.
\newblock Preprint, available at \url{http://arxiv.org/abs/1506.04774}, 2015.

\bibitem{Fef86}
C.~Fefferman.
\newblock The {$N$}-body problem in quantum mechanics.
\newblock {\em Comm. Pure Appl. Math.}, 39(S, suppl.):S67--S109, 1986.
\newblock Frontiers of the mathematical sciences: 1985 (New York, 1985).

\bibitem{Fie76}
T.~Figiel.
\newblock On the moduli of convexity and smoothness.
\newblock {\em Studia Math.}, 56(2):121--155, 1976.

\bibitem{FJ74}
T.~Figiel and W.~B. Johnson.
\newblock A uniformly convex {B}anach space which contains no {$l_{p}$}.
\newblock {\em Compositio Math.}, 29:179--190, 1974.

\bibitem{FP74}
T.~Figiel and G.~Pisier.
\newblock S\'eries al\'eatoires dans les espaces uniform\'ement convexes ou
  uniform\'ement lisses.
\newblock {\em C. R. Acad. Sci. Paris S\'er. A}, 279:611--614, 1974.

\bibitem{GM14}
A.~Giannopoulos and E.~Milman.
\newblock {$M$}-estimates for isotropic convex bodies and their
  {$L_q$}-centroid bodies.
\newblock In {\em Geometric aspects of functional analysis}, volume 2116 of
  {\em Lecture Notes in Math.}, pages 159--182. Springer, Cham, 2014.

\bibitem{Gia95}
A.~A. Giannopoulos.
\newblock A note on the {B}anach-{M}azur distance to the cube.
\newblock In {\em Geometric aspects of functional analysis ({I}srael,
  1992--1994)}, volume~77 of {\em Oper. Theory Adv. Appl.}, pages 67--73.
  Birkh\"auser, Basel, 1995.

\bibitem{GNS}
O.~Giladi, A.~Naor, and G.~Schechtman.
\newblock Bourgain's discretization theorem.
\newblock {\em Ann. Fac. Sci. Toulouse Math. (6)}, 21(4):817--837, 2012.

\bibitem{HLN}
T.~{Hyt{\"o}nen}, S.~{Li}, and A.~{Naor}.
\newblock {Quantitative affine approximation for UMD targets}.
\newblock {\em Discrete Analysis}, 2016(6):1--37, 2016.

\bibitem{HNVW}
T.~Hyt\"onen, J.~van Neerven, M.~Veraar, and L.~Weis.
\newblock Analysis in {B}anach spaces. {V}ol. {I}.
\newblock Monograph in preparation, draft available at
  \url{http://fa.its.tudelft.nl/~neerven/ABS/volume1_Hytonen_Neerven_Veraar_Weis.pdf},
  2016.

\bibitem{Joh48}
F.~John.
\newblock Extremum problems with inequalities as subsidiary conditions.
\newblock In {\em Studies and essays presented to R. Courant on his 60th
  birthday}, pages 187--204. Interscience Publishers, Inc., 1948.

\bibitem{JLS96}
W.~B. Johnson, J.~Lindenstrauss, and G.~Schechtman.
\newblock Banach spaces determined by their uniform structures.
\newblock {\em Geom. Funct. Anal.}, 6(3):430--470, 1996.

\bibitem{Jon89}
P.~W. Jones.
\newblock Square functions, {C}auchy integrals, analytic capacity, and harmonic
  measure.
\newblock In {\em Harmonic analysis and partial differential equations ({E}l
  {E}scorial, 1987)}, volume 1384 of {\em Lecture Notes in Math.}, pages
  24--68. Springer, Berlin, 1989.

\bibitem{Kei04}
S.~Keith.
\newblock A differentiable structure for metric measure spaces.
\newblock {\em Adv. Math.}, 183(2):271--315, 2004.

\bibitem{Kir94}
B.~Kirchheim.
\newblock Rectifiable metric spaces: local structure and regularity of the
  {H}ausdorff measure.
\newblock {\em Proc. Amer. Math. Soc.}, 121(1):113--123, 1994.

\bibitem{Kla06}
B.~Klartag.
\newblock On convex perturbations with a bounded isotropic constant.
\newblock {\em Geom. Funct. Anal.}, 16(6):1274--1290, 2006.

\bibitem{KM07}
J.~Kristensen and G.~Mingione.
\newblock The singular set of {L}ipschitzian minima of multiple integrals.
\newblock {\em Arch. Ration. Mech. Anal.}, 184(2):341--369, 2007.

\bibitem{KMS14}
T.~Kuusi, G.~Mingione, and Y.~Sire.
\newblock A fractional {G}ehring lemma, with applications to nonlocal
  equations.
\newblock {\em Atti Accad. Naz. Lincei Rend. Lincei Mat. Appl.},
  25(4):345--358, 2014.

\bibitem{LN14}
V.~Lafforgue and A.~Naor.
\newblock Vertical versus horizontal {P}oincar\'e inequalities on the
  {H}eisenberg group.
\newblock {\em Israel J. Math.}, 203(1):309--339, 2014.

\bibitem{LN06}
J.~R. Lee and A.~Naor.
\newblock ${L}_p$ metrics on the {H}eisenberg group and the {G}oemans-{L}inial
  conjecture.
\newblock In {\em Proceedings of 47th Annual IEEE Symposium on Foundations of
  Computer Science (FOCS 2006)}, pages 99--108, 2006.

\bibitem{LR10}
J.~R. Lee and P.~Raghavendra.
\newblock Coarse differentiation and multi-flows in planar graphs.
\newblock {\em Discrete Comput. Geom.}, 43(2):346--362, 2010.

\bibitem{LS11}
J.~R. Lee and A.~Sidiropoulos.
\newblock Near-optimal distortion bounds for embedding doubling spaces into
  {$L_1$} [extended abstract].
\newblock In {\em S{TOC}'11---{P}roceedings of the 43rd {ACM} {S}ymposium on
  {T}heory of {C}omputing}, pages 765--772. ACM, New York, 2011.

\bibitem{Li14}
S.~Li.
\newblock Coarse differentiation and quantitative nonembeddability for {C}arnot
  groups.
\newblock {\em J. Funct. Anal.}, 266(7):4616--4704, 2014.

\bibitem{LiNaor:13}
S.~Li and A.~Naor.
\newblock Discretization and affine approximation in high dimensions.
\newblock {\em Israel J. Math.}, 197(1):107--129, 2013.

\bibitem{LT79}
J.~Lindenstrauss and L.~Tzafriri.
\newblock {\em Classical {B}anach spaces. {II}}, volume~97 of {\em Ergebnisse
  der Mathematik und ihrer Grenzgebiete [Results in Mathematics and Related
  Areas]}.
\newblock Springer-Verlag, Berlin-New York, 1979.
\newblock Function spaces.

\bibitem{LMS98}
A.~E. Litvak, V.~D. Milman, and G.~Schechtman.
\newblock Averages of norms and quasi-norms.
\newblock {\em Math. Ann.}, 312(1):95--124, 1998.

\bibitem{MTX}
T.~Mart{\'{\i}}nez, J.~L. Torrea, and Q.~Xu.
\newblock Vector-valued {L}ittlewood-{P}aley-{S}tein theory for semigroups.
\newblock {\em Adv. Math.}, 203(2):430--475, 2006.

\bibitem{Mat99}
J.~Matou{\v{s}}ek.
\newblock On embedding trees into uniformly convex {B}anach spaces.
\newblock {\em Israel J. Math.}, 114:221--237, 1999.

\bibitem{Mau75}
B.~Maurey.
\newblock Syst\`eme de {H}aar.
\newblock In {\em S\'eminaire {M}aurey-{S}chwartz 1974--1975: {E}spaces
  {L{$\sup{p}$}}, applications radonifiantes et g\'eom\'etrie des espaces de
  {B}anach, {E}xp. {N}os. {I} et {II}}, pages 26 pp. (erratum, p. 1). Centre
  Math., \'Ecole Polytech., Paris, 1975.

\bibitem{MP76}
B.~Maurey and G.~Pisier.
\newblock S\'eries de variables al\'eatoires vectorielles ind\'ependantes et
  propri\'et\'es g\'eom\'etriques des espaces de {B}anach.
\newblock {\em Studia Math.}, 58(1):45--90, 1976.

\bibitem{MN13}
M.~Mendel and A.~Naor.
\newblock Markov convexity and local rigidity of distorted metrics.
\newblock {\em J. Eur. Math. Soc. (JEMS)}, 15(1):287--337, 2013.

\bibitem{MN14}
M.~Mendel and A.~Naor.
\newblock Nonlinear spectral calculus and super-expanders.
\newblock {\em Publ. Math. Inst. Hautes \'Etudes Sci.}, 119:1--95, 2014.

\bibitem{Mil06}
E.~Milman.
\newblock Dual mixed volumes and the slicing problem.
\newblock {\em Adv. Math.}, 207(2):566--598, 2006.

\bibitem{MP89}
V.~D. Milman and A.~Pajor.
\newblock Isotropic position and inertia ellipsoids and zonoids of the unit
  ball of a normed {$n$}-dimensional space.
\newblock In {\em Geometric aspects of functional analysis (1987--88)}, volume
  1376 of {\em Lecture Notes in Math.}, pages 64--104. Springer, Berlin, 1989.

\bibitem{MS86}
V.~D. Milman and G.~Schechtman.
\newblock {\em Asymptotic theory of finite-dimensional normed spaces}, volume
  1200 of {\em Lecture Notes in Mathematics}.
\newblock Springer-Verlag, Berlin, 1986.
\newblock With an appendix by M. Gromov.

\bibitem{Nao98}
A.~Naor.
\newblock Geometric problems in non-linear functional analysis.
\newblock Master's thesis, Hebrew University, 1998.

\bibitem{NS07}
A.~Naor and G.~Schechtman.
\newblock Planar earthmover is not in {$L_1$}.
\newblock {\em SIAM J. Comput.}, 37(3):804--826 (electronic), 2007.

\bibitem{NS16}
A.~Naor and G.~Schechtman.
\newblock Metric {$X_p$} inequalities.
\newblock {\em Forum Math. Pi}, 4:e3, 81 pp., 2016.

\bibitem{Ost13}
M.~I. Ostrovskii.
\newblock {\em Metric embeddings}, volume~49 of {\em De Gruyter Studies in
  Mathematics}.
\newblock De Gruyter, Berlin, 2013.
\newblock Bilipschitz and coarse embeddings into Banach spaces.

\bibitem{Pan89}
P.~Pansu.
\newblock M\'etriques de {C}arnot-{C}arath\'eodory et quasiisom\'etries des
  espaces sym\'etriques de rang un.
\newblock {\em Ann. of Math. (2)}, 129(1):1--60, 1989.

\bibitem{Pao06}
G.~Paouris.
\newblock Concentration of mass on convex bodies.
\newblock {\em Geom. Funct. Anal.}, 16(5):1021--1049, 2006.

\bibitem{Pau01}
S.~D. Pauls.
\newblock The large scale geometry of nilpotent {L}ie groups.
\newblock {\em Comm. Anal. Geom.}, 9(5):951--982, 2001.

\bibitem{Pen11-1}
I.~Peng.
\newblock Coarse differentiation and quasi-isometries of a class of solvable
  {L}ie groups {I}.
\newblock {\em Geom. Topol.}, 15(4):1883--1925, 2011.

\bibitem{Pen11-2}
I.~Peng.
\newblock Coarse differentiation and quasi-isometries of a class of solvable
  {L}ie groups {II}.
\newblock {\em Geom. Topol.}, 15(4):1927--1981, 2011.

\bibitem{Pisier:1975}
G.~Pisier.
\newblock Martingales with values in uniformly convex spaces.
\newblock {\em Israel J. Math.}, 20(3-4):326--350, 1975.

\bibitem{Pis75}
G.~Pisier.
\newblock Un exemple concernant la super-r\'eflexivit\'e.
\newblock In {\em S\'eminaire {M}aurey-{S}chwartz 1974--1975: {E}spaces
  {$L^{p}$}\ applications radonifiantes et g\'eom\'etrie des espaces de
  {B}anach, {A}nnexe {N}o. 2}, page~12. Centre Math. \'Ecole Polytech., Paris,
  1975.

\bibitem{Pis82}
G.~Pisier.
\newblock Holomorphic semigroups and the geometry of {B}anach spaces.
\newblock {\em Ann. of Math. (2)}, 115(2):375--392, 1982.

\bibitem{Pis83}
G.~Pisier.
\newblock Counterexamples to a conjecture of {G}rothendieck.
\newblock {\em Acta Math.}, 151(3-4):181--208, 1983.

\bibitem{Pisier:1986}
G.~Pisier.
\newblock Probabilistic methods in the geometry of {B}anach spaces.
\newblock In {\em Probability and analysis ({V}arenna, 1985)}, volume 1206 of
  {\em Lecture Notes in Math.}, pages 167--241. Springer, Berlin, 1986.

\bibitem{Qiu}
Y.~Qiu.
\newblock On the {UMD} constants for a class of iterated {$L_p(L_q)$} spaces.
\newblock {\em J. Funct. Anal.}, 263(8):2409--2429, 2012.

\bibitem{Rib76}
M.~Ribe.
\newblock On uniformly homeomorphic normed spaces.
\newblock {\em Ark. Mat.}, 14(2):237--244, 1976.

\bibitem{Rota}
G.-C. Rota.
\newblock An ``{A}lternierende {V}erfahren'' for general positive operators.
\newblock {\em Bull. Amer. Math. Soc.}, 68:95--102, 1962.

\bibitem{Sch09}
R.~Schul.
\newblock Bi-{L}ipschitz decomposition of {L}ipschitz functions into a metric
  space.
\newblock {\em Rev. Mat. Iberoam.}, 25(2):521--531, 2009.

\bibitem{See89}
A.~Seeger.
\newblock A note on {T}riebel-{L}izorkin spaces.
\newblock In {\em Approximation and function spaces ({W}arsaw, 1986)},
  volume~22 of {\em Banach Center Publ.}, pages 391--400. PWN, Warsaw, 1989.

\bibitem{Stein:topics}
E.~M. Stein.
\newblock {\em Topics in harmonic analysis related to the {L}ittlewood-{P}aley
  theory.}
\newblock Annals of Mathematics Studies, No. 63. Princeton University Press,
  Princeton, N.J.; University of Tokyo Press, Tokyo, 1970.

\bibitem{Triebel:89}
H.~Triebel.
\newblock Local approximation spaces.
\newblock {\em Z. Anal. Anwendungen}, 8(3):261--288, 1989.

\bibitem{Vil03}
C.~Villani.
\newblock {\em Topics in optimal transportation}, volume~58 of {\em Graduate
  Studies in Mathematics}.
\newblock American Mathematical Society, Providence, RI, 2003.

\bibitem{Xu98}
Q.~Xu.
\newblock Littlewood-{P}aley theory for functions with values in uniformly
  convex spaces.
\newblock {\em J. Reine Angew. Math.}, 504:195--226, 1998.

\end{thebibliography}
\bibliographystyle{abbrv}

\end{document}